\documentclass[12pt]{amsart}%a4paper
\usepackage{amsmath, parskip, amscd, amssymb, amsthm,amsfonts,mathrsfs,mathtools,bbm,tikz,mathdots,arydshln,pinlabel, stmaryrd,mathtools}
\pdfoutput=1
\usepackage[margin=1in]{geometry}
\usepackage[small,nohug]{diagrams}
\usepackage[T1]{fontenc}
\makeatletter

\usepackage{color}
\usepackage{tikz}
\usetikzlibrary{matrix,arrows}

\newtheorem{theorem}[equation]{Theorem}
\newtheorem{lemma}[equation]{Lemma}
\newtheorem{proposition}[equation]{Proposition}
\newtheorem{corollary}[equation]{Corollary}
\newtheorem{conjecture}[equation]{Conjecture}

\newtheorem*{hypk}{\textbf{Hyp}$(k)$}

\newtheorem*{thm-extGroups1}{Theorem \ref{thm-extGroups1}}
\newtheorem*{thm-extGroups2}{Theorem \ref{thm-extGroups2}}

\theoremstyle{definition}
\newtheorem{definition}[equation]{Definition}
\newtheorem{construction}[equation]{Construction}

\theoremstyle{remark}
\newtheorem{example}[equation]{Example}
\newtheorem{remark}[equation]{Remark}
\newtheorem{observation}[equation]{Observation}

\numberwithin{equation}{section}

\newcommand{\gr}{\operatorname{gr}}

\newcommand{\inv}{^{-1}}
\newcommand{\R}{\mathbb{R}}
\newcommand{\Z}{\mathbb{Z}}
\newcommand{\N}{\mathbb{N}}
\newcommand{\Q}{\mathbb{Q}}
\newcommand{\C}{\mathbb{C}}
\newcommand{\e}{\varepsilon}
\renewcommand{\d}{\delta}
\renewcommand{\a}{\alpha}\renewcommand{\b}{\beta}

\renewcommand{\emptyset}{\varnothing}

\newcommand{\Sym}{\operatorname{Sym}}

\newcommand{\Endg}{\hyperref[def-homg]{\operatorname{END}}}
\newcommand{\Homg}{\hyperref[def-homg]{\operatorname{HOM}}}

\newcommand{\Id}{\operatorname{Id}}
\newcommand{\Hom}{\operatorname{Hom}}

\renewcommand{\matrix}[1]{\begin{bmatrix}#1\end{bmatrix}}

\newcommand{\Ext}{\operatorname{Ext}}

\renewcommand{\deg}{\operatorname{deg}}

\newcommand{\Kom}{\operatorname{Kom}}

\newcommand{\bn}[1]{\mathcal{T}\hskip-2pt\mathcal{L}_{#1}}
\newcommand{\Tot}{\operatorname{Tot}}
\newcommand{\Cob}{\operatorname{Cob}}

\newcommand{\Mat}{\operatorname{Mat}}
\newcommand{\TL}{\operatorname{TL}}
\newcommand{\Cone}{\operatorname{Cone}}

\renewcommand{\sl}{\mathfrak{sl}}
\newcommand{\Tw}{\operatorname{Tw}}
\newcommand{\smMatrix}[1]{\left[\begin{smallmatrix}#1\end{smallmatrix}\right]}

\newcommand{\FKzero}{\wpic{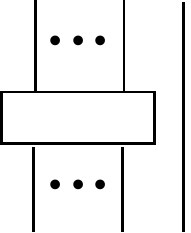}  }
\newcommand{\FKone}{\wpic{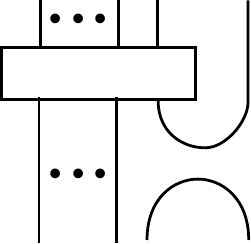} }

\newcommand{\FKtwo}{\wwpic{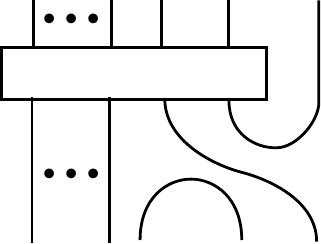} }

\newcommand{\FKpenult}{\wwpic{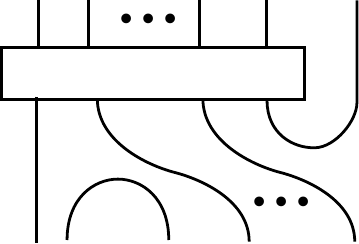}}
\newcommand{\FKult}{\wpic{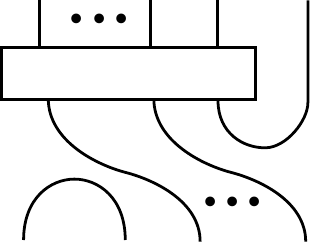}}

\newarrow{Congruent}===={}

\newcommand{\pic}[1]{
\begin{minipage}{.22in}
\begin{center}\includegraphics[scale=.3]{#1}\end{center}
\end{minipage}
}
\newcommand{\wpic}[1]{
\begin{minipage}{.38in}
\begin{center}\includegraphics[scale=.3]{#1}\end{center}
\end{minipage}
}

\newcommand{\wwpic}[1]{
\begin{minipage}{.58in}
\begin{center}\includegraphics[scale=.3]{#1}\end{center}
\end{minipage}
}

\newcommand{\wwwpic}[1]{
\begin{minipage}{.7in}
\begin{center}\includegraphics[scale=.3]{#1}\end{center}
\end{minipage}
}

\newcommand{\bpic}[1]{
\begin{minipage}{.38in}
\begin{center}\includegraphics[scale=.4]{#1}\end{center}
\end{minipage}
}

\newcommand{\mpic}[1]{
 \begin{minipage}{.14in}
\begin{center}\includegraphics[scale=.23]{#1}\end{center}%\vskip -21pt {\ }
\end{minipage}
}
\newcommand{\wmpic}[1]{
 \begin{minipage}{.32in}
\begin{center}\includegraphics[scale=.25]{#1}\end{center}%\vskip -21pt {\ }
\end{minipage}
}

\begin{document}

\title{A polynomial action on colored $\sl_2$ link homology}
\author{Matt Hogancamp}

\begin{abstract}
We construct an action of a polynomial ring on the colored $\sl_2$ link homology of Cooper-Krushkal, over which this homology is finitely generated.  We define a new, related link homology which is finite dimensional, extends to tangles, and categorifies a scalar-multiple of the $\sl_2$ Reshetikhin-Turaev invariant.  We expect this homology to be functorial under 4-dimensional cobordisms.  The polynomial action is related to a conjecture of Gorsky-Oblomkov-Rasmussen-Shende on the stable Khovanov homology of torus knots, and as an application we obtain a weak version of this conjecture.  A key new ingredient is the construction of a bounded chain complex which categorifies a scalar multiple of the Jones-Wenzl projector, in which the denominators have been cleared.
\end{abstract}

\maketitle

%\tableofcontents \addtocontents{toc}

\section{Introduction}
%The Jones-Wenzl projectors are fundamental objects in quantum topology and representation theory.  On the low-dimensional topology side, they are essential in the construction of the $\sl_2$ quantum invariants of links and 3-manifolds.  On the representation theory side, the Jones-Wenzl projectors are the idempotent operators corresponding to the projection $V_1^{\otimes n}\twoheadrightarrow V_n$, where $V_n$ is a $q$-deformed version of the $n+1$-dimensional irreducible representation of $\sl_2$.  

%In the seminal work \cite{Kh00}, M.~Khovanov categorified the (1-colored) Jones polynomial.  One of the most important and exciting features of Khovanov's homology theory is its connection to 4-dimensional topology, coming from its functoriality up to sign under 4-dimensional cobordisms.  Since Khovanov's original work, all of the Reshetikhin-Turaev invariants have been categorified, often in multiple ways (see \S \ref{}).

%The colored Jones polynomial is a $\Z[q,q\inv]$-valued invariant of framed, oriented links $L\subset S^3$ whose components are labeled by non-negative integers, called the colors.  In his seminal work \cite{Kh00}, M.~Khovanov categorified the (1-colored) Jones polynomial, and later extended this to a categorification of the colored Jones polynomial \cite{Kh05}.  His homolo  extend to tangles.

In this paper, we address the problem of infinite dimensionality of colored link homology theories, specifically the colored $\sl_2$ link homology constructed by B.~Cooper and V.~Krushkal \cite{CK12a}.  Our motivation is the program of constructing categorifications of the Reshetikhin-Turaev link invariants which are functorial under 4-dimensional link cobordisms.  Functoriality requires that the homology of the unknot have finite total rank, which fails for most colored link homology theories.   Indeed B.~Webster \cite{W10a,W10b} has categorified the Reshetikhin-Turaev invariants for links $L\subset S^3$ for general Lie algebras $\mathfrak{g}$, but Webster's homology of the $V$-colored unknot is finite dimensional only when $V$ is a \emph{minuscule} representation of $\mathfrak{g}$.   In the special case of $\mathfrak{g}=\sl_N$, there are many categorifications of the Reshetikhin-Turaev invariants (see \S \ref{subsec-otherLieAlgs} for a discussion).  These, too, tend to be infinite dimensional unless one restricts to minuscule representations (exterior powers $\Lambda^i(\C^N)$ of the natural representation).  Roughly speaking, infinite complexes are required to categorify denominators appearing in the Reshetikhin-Turaev tangle invariant.

There is one exception to this rule, given by Khovanov's colored $\sl_2$ homology \cite{Kh05}.  Khovanov's colored homology avoids infinite complexes by categorifying the colored link invariant, and not its extension to tangles.  In this paper we prefer not to sacrifice the extension to tangles, and instead fix the problem of infinite complexes in a different way.  Essentially we provide a categorical analogue of clearing denominators for Cooper-Krushkal homology.

\subsection{A categorical analogue of clearing denominators}
For a definition of the Temperley-Lieb algebra $\TL_n$ and the Jones-Wenzl projectors $p_n\in \TL_n$ see \S \ref{sec-TL}.  One can define an integral form $\TL^\Z_n\subset \TL_n$, which is the $\Z[q,q\inv]$-subalgebra generated by diagrams (crossingless matchings of $2n$ points in a rectangle).  The Jones-Wenzl projector is not an element of $\TL_n^\Z$, but from the explicit recursion (\ref{eq-JWrecursion}) one can see that:
\begin{equation}\label{introDenomClearedEq}
\prod_{k=2}^n (1-q^{2k})p_n\in\TL_n^\Z
\end{equation}
Thus, we say that $p_n$ has denominators of the form $1-q^{2k}$.

In \S \ref{subsec-tangleCat} we recall Bar-Natan's category  $\bn{n}$ of tangles and dotted cobordisms.  This would be denoted by $\Mat(({\Cob^3_\bullet}(n)_{/l})$ in \cite{B-N05}.  Equivalently, we could work with modules over Khovanov's rings $H^n$ \cite{Kh02}.  There is an isomorphism from the split Grothendieck group of $\bn{n}$ to $\TL_n^\Z$, which implies that a bounded complex over $\bn{n}$ has a well-defined Euler characteristic $\chi(A)\in \TL_n^\Z$.  In contrast, unbounded complexes do not all have a well-defined Euler characteristic.  Nonetheless, all of the unbounded complexes considered in this paper have a well-defined Euler characteristic, which takes values in the algebra obtained from $\TL_n^\Z$ by extending scalars to Laurent series $\Z[q\inv]\llbracket q\rrbracket$ (see \S \ref{subsec-eulerChar}).   Homotopy equivalent complexes have equal Euler characteristics.  We say that a complex categorifies its Euler characteristic, and that an equivalence of complexes categorifies the corresponding identity in $\TL_n$, \emph{et cetera}.

The main ingredient in Cooper-Krushkal homology is the construction in \cite{CK12a} of family of complexes $P_n\in\Kom(n)$ which categorify the Jones-Wenzl projectors.  This requires unbounded complexes, in order to categorify the denominators $(1-q^{2k})\inv = \sum_{i=0}^\infty q^{2ki}$.  It is natural to ask whether the rescaled version of $p_n$ given in expression (\ref{introDenomClearedEq}) admits a categorification by a bounded complex.  As we will see, the answer is yes.%  We take this %The point of this paper is to construct such complexes, study some of their properties, and then use them to construct finite dimensional link homology theories (simply by taking the Cooper-Krushkal construction, and replacing $P_n$ by our bounded complexes).  The bounded complexes retain a relationship with $P_n$, which   

The heart of this paper is the construction, in section \S \ref{sec-quasiprojectors}, of complexes $Q_n$ which categorify the elements $(1-q^{2n})p_n\in \TL_n$ in which certain denominators have been cleared.  The complex $Q_n$ is built out of $P_{n-1}=\pic{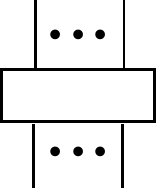}$ in the following way:
\[
Q_n = \Big(\FKzero\rightarrow \FKone \rightarrow \cdots \rightarrow \FKult\rightarrow \FKult\rightarrow \cdots \rightarrow \FKone\rightarrow \FKzero\Big)
\]
Each picture above denotes a chain complex $E_i$ over $\bn{n}$, and each arrow corresponds to a certain chain map $E_i\rightarrow E_{i+1}$.  See Definition \ref{def-symmetricSeq} for details. The above notation means that $Q_n$ is equal to a direct sum of the $E_i$ (grading shifts omitted), endowed with a differential which is the sum of  the differentials internal to each $E_i$, the arrows connecting adjacent $E_i$'s, and morphisms corresponding length $>1$ arrows pointing to the right.  Such a complex is called a convolution (Definition \ref{def-convolution}), and can also be thought of as an iterated mapping cone
\[
Q_n = \Cone\bigg(E_{1-2n}\rightarrow \Cone\Big(\cdots \rightarrow \Cone\Big(E_{-2}\rightarrow \Cone(E_{-1}\rightarrow E_0)\Big)\cdots \Big)\bigg).
\]
These complexes $Q_n$ are of fundamental importance. For instance, the Cooper-Krushkal projectors $P_n$ can be recovered from $Q_n$ by a categorical analogue of multiplication by $(1-q^{2n})\inv = \sum_{k= 0}^\infty q^{2nk}$: 
\begin{theorem}\label{introPandQthm}
There is a chain map $\partial_n\in \Endg(Q_n)$ of bidegree $\deg(\partial_n)=(2n-1,-2n)$ such that
\begin{equation}
\label{introKoszulEq}
P_n \simeq \Z[u_n]\otimes Q_n \ \ \ \ \text{with differential} \ \ \ \ 1\otimes d_{Q_n}+u_n\otimes \partial_n
\end{equation}
where $u_n$ is a formal indeterminate of bidegree $(2-2n,2n)$.
\end{theorem}
Here, $\Homg(A,B)$ denotes the chain complex generated by all bihomogenous linear maps between chain complexes, with differential $f\mapsto d_B\circ f-(-1)^{\deg_h(f)}f\circ d_A$.  The bidegree is $\deg(f) = (\deg_h(f),\deg_q(f))$.  Theorem \ref{introKoszulEq} makes it obvious that $\Z[u_n]$ acts on $P_n$.  In fact, $u_n$ includes $\Z[u_n]\otimes Q_n$ as subcomplex of itself, with quotient $Q_n$ (this is an imprecise statement, since $\bn{n}$ is not an abelian category).  More precisely:

\begin{theorem}\label{introPandQthm2}
Let $U_n\in\Endg(P_n)$ denote the chain map coming from the action of $u_n$ on the complex (\ref{introKoszulEq}).  Then
\begin{equation}
\label{introKoszulEq1}
Q_n\simeq \Cone(U_n)
\end{equation}
\end{theorem}
The presentation of $Q_n$ as the mapping cone on an endomorphism makes it obvious that the exterior algebra $\Lambda[\partial_n]$ acts on $Q_n$, and we easily recover the statement of Theorem \ref{introPandQthm}.  The dual relationship between Theorem \ref{introPandQthm} and Theorem \ref{introPandQthm2} is precisly that of Koszul duality between modules over polynomial and exterior algebras.  

On the level of Euler characteristic, equivalence (\ref{introKoszulEq}) becomes $\chi(P_n) = (1-q^{2n})\inv \chi(Q_n)$ and equivalence (\ref{introKoszulEq1}) becomes $\chi(Q_n) = (1-q^{2n})\chi(P_n)$.  Thus, the constructions in Theorems \ref{introPandQthm} and \ref{introPandQthm2} provide categorical analogues of division and multiplication by $(1-q^{2n})$.  

Recall that we are looking for a bounded complex which categorifies the scalar multiple of $p_n$ appearing in expression (\ref{introDenomClearedEq}).  Such a complex is provided by a tensor product of the $Q_k$'s, for $2\leq k\leq n$.  To see this, first let $P_k'$ denote the complex $\Z[u_k]\otimes Q_k$ with differential as in (\ref{introKoszulEq}).  Let $(-)\sqcup 1_i:\Kom(j)\rightarrow \Kom(i+j)$ denote the functor which places $i$ parallel strands to the right of each diagram.  From the unique characterization of Cooper-Krushkal projectors, one can see that $(P_k'\sqcup 1_{n-k})\otimes P_n'\simeq P_n'$.  In particular
\[
P_n\simeq (P_2'\sqcup 1_{n-2})\otimes(P_3'\sqcup 1_{n-3})\otimes \cdots \otimes P_n'
\]
Written differently this is:
\begin{corollary}
We have
\begin{equation}
\label{introKoszulEq2}
P_n\  \simeq \ \Z[u_2,u_3,\ldots, u_n]\otimes K_n \ \ \ \ \text{with differential} \ \ \ \ 1\otimes d_{K_n}+\sum_{k=2}^n u_k\otimes \partial_k
\end{equation}
where $K_n := (Q_2\sqcup 1_{n-2})\otimes(Q_3\sqcup 1_{n-3})\otimes \cdots\otimes Q_n$.
\end{corollary}
On the level of Euler-characteristic, this equivalence becomes:
\[
\chi(P_n) = \prod_{k=2}^n (1-q^{2k})\inv \chi(K_n)
\]
That is to say,
\[
\chi(K_n) = \prod_{k=2}^n (1-q^{2k})\chi(P_n),
\]
as desired.  We have seen that this element lies in the integral form $\TL_n^\Z$.  This fact is categorified by the following, which is proven in \S \ref{subsec-standardCx}:
\begin{theorem}\label{thm-introKnBdd}
The complex $K_n = (Q_2\sqcup 1_{n-2})\otimes(Q_3\sqcup 1_{n-3})\otimes \cdots\otimes Q_n$ is homotopy equivalent to a bounded complex.
\end{theorem}

\begin{remark}
We know that the result of Theorem \ref{thm-introKnBdd} is not optimal.  For example $Q_3$ is homotopy equivalent to a bounded complex (see Example \ref{example-Q3}).  On the other hand $Q_k$ is not homotopy equivalent to a bounded complex for $k\geq 4$.  For a discussion on precisely which tensor products of the $Q_k$'s are bounded, see \S \ref{subsec-improvingBddness}.
\end{remark}

There is another special property that the Euler characteristic of $K_n$ has, namely that it is a scalar multiple of an idempotent.  In general, if $e\in A$ idempotent element of an algebra, and $\a$ is a scalar, then any multiple $f=\a e$ satisfies $f^2=\a f$.  It turns out that this property has a categorical analogue as well:
\begin{theorem}\label{thm-introQnProps}
The complexes $Q_k$ satisfy
\begin{enumerate}
\item $Q_n^{\otimes 2}\simeq Q_n\oplus t^{1-2n}q^{2n}Q_n$.
\item $(Q_k\sqcup 1_{n-k})\otimes Q_n\cong Q_n\otimes (Q_k\sqcup 1_{n-k})$.
\end{enumerate}
In particular $K_n^{\otimes 2}$ is equivalent to a direct sum of $2^{n-1}$ copies of $K_n$, with grading shifts.
\end{theorem}
This is proven in \S \ref{subsec-quasiProjProps}.

\subsection{The polynomial action on $P_n$}
Note that any chain complex $C\in\Kom(n)$ is an $(R,R)$-bimodule, where $R=\Z[x_1,\ldots,x_n]/(x_1^2,\cdots, x_n^2 )$.  The action of $x_i$ is given by the identity cobordism with a dot on one of the sheets or, equivalently, multiplication by $X$ in Khovanov homology.  The actions of $x_i$ on $P_n$ are homotopic up to sign.  We will pick our favorite of these endomorphisms and call it $U_1^{(n)}$.  For instance, $U_1^{(n)}:=\pic{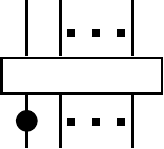}\in\Endg(P_n)$.  This induces an action of $\Z[u_1]$ on $P_n$.  The equivalence (\ref{introKoszulEq2}) allows us to extend this action to an action of a larger polynomial ring:
\begin{definition}\label{def-introPoly}
For $2\leq k\leq n$, let $U_k^{(n)}\in\Endg(P_n)$ denote the map induced from the action of $u_k$ on the complex (\ref{introKoszulEq2}).  As in the preceding remarks, put $U_1^{(n)}:=\pic{projector_SWdot}$.
\end{definition}
The homology of $\Endg(P_n)$ is graded commutative \cite{H12a}, so these chain maps define an action of $\Z[u_1,\ldots,u_n]$ on $P_n$, up to homotopy.  We will see in a moment that the maps $U_k^{(n)}$ do not depend on any choices, up to homotopy and sign.  But first, the following illustrates the usefulness of equivalence (\ref{introKoszulEq2}) in computations.
\begin{theorem}\label{thm-introGorThm}
There is a representative for $P_n$ such that $\Endg(P_n)$ deformation retracts onto a differential bigraded $\Z[u_1,\ldots,u_n]$-module $W_n=\Z[u_1,\ldots,u_n]/(u_1^2)\otimes \Lambda[\xi_2,\xi_3,\ldots,\xi_n]$ with differential satisfying:
\begin{enumerate}
\item $d(u_k)=0$ for each $k=1,2,\ldots,n$.
\item $d(\xi_k)\in 2u_1u_k+\Z[u_2,\ldots,u_{k-1}]$ for each $k=2,3,\ldots,n$.
\end{enumerate}
The data of the deformation retract $\Endg(P_n)\rightarrow W_n$ are $\Z[u_1,\ldots,u_n]$-equivariant, where $u_k$ acts on $\Endg(P_n)$ via post-composition with $U_k^{(n)}$.
\end{theorem}

This is proven in \S \ref{subsec-endSimplification}.  As a corollary, we have
\begin{theorem}\label{thm-introExtGrpsAndQs}
For $1\leq k\leq n$, the group of chain maps $t^{2-2k}q^{2k}P_n\rightarrow P_n$ modulo chain homotopy is isomorphic to $\Z$, generated by $U_k^{(n)}$.  The mapping cones satisfy $\Cone(U_k^{(n)})\simeq (Q_k\sqcup 1_{n-k})\otimes P_n$.  In particular,  $Q_n$ is uniquely characterized by the equivalence $Q_n\simeq \Cone(U_n^{(n)})$.
\end{theorem}
For the uniqueness statement, see Theorem \ref{thm-QnUniqueness}.

\subsection{Application to link homology}
\label{introLinkHomSection}
Let us outline the construction of our family of link homology theories, explained in detail in \S \ref{subsec-quasiHomology}.  First:

\begin{definition}\label{def-introQuasiProj}
Recall the maps $U_k^{(n)}$ from Definition \ref{def-introPoly}.  For any sequence $1\leq i_1,\ldots,i_r\leq n$, let $P_n(i_1,\ldots,i_r)$ denote the tensor product of complexes $\Cone(U_{i_k}^{(n)})$.  For the empty sequence, we put $P_n(\emptyset):=P_n$.  We call these complexes $P_n(\mathbf{i})$ \emph{quasi-projectors}.
\end{definition}  

Let $D$ be an oriented tangle diagram, together with a finite number of marked points $x_1,\ldots,x_k$ on $D$, away from the crossings, such that there is at least one $x_i$ on each component of the underlying link.  Let $\mathcal{K} = \{K_1,K_2,\ldots\}$ be a family of complexes of the form $K_n = P_n(\mathbf{i}_n)$ for some sequences $\mathbf{i}_n$.  Associated to these data we associate a chain complex $\llbracket D; \mathcal{K} \rrbracket$ as follows: replace an $n$-colored component of $D$ with $n$ parallel copies of itself, insert the appropriate $K_n$ near each marked point, and define $\llbracket D; \mathcal{K} \rrbracket$ using the planar composition operations on Bar-Natan's categories (formally similar to tensor product).  %In case $\mathcal{P} = \{P_1, P_2\ldots,\}$ is the family of Cooper-Krushkal projectors, then $\llbracket D; \mathcal{P} \rrbracket$ coincides with the Cooper-Krushkal complex associated to $D$, up to homotopy equivalence.

By a decorated tangle we will mean a pair $(T,V)$, where $T \subset B^3$ is a colored, framed, oriented tangle together $V \subset T$ is a finite set of points in the interior of $T$.  We regard these modulo (framed) isotopy of pairs.  We prove the following in \S \ref{subsec-quasiHomology}:
\begin{theorem}\label{introQuasiLocalThm}
The chain homotopy type of $\llbracket D; \mathcal{K} \rrbracket$ is an invariant of the underlying decorated tangle.  This invariant satisfies:
\begin{enumerate}
\item Suppose $D$ and $D'$ are identical, except that $D'$ has one fewer marked point than $D$, on a component colored by $n$.  Then $\llbracket D; \mathcal{K} \rrbracket$ is chain homotopy equivalent to a direct sum of copies of $\llbracket D'; \mathcal{K} \rrbracket$ with degree shifts, depending only on $n$.
\item If $D$ and $D'$ differ only in the orientation, then $\llbracket D;\mathcal{K}\rrbracket \simeq \llbracket D';\mathcal{K}\rrbracket$ up to an overall degree shift.
\item For some choices of $\mathcal{K}$, $\llbracket D;\mathcal{K}\rrbracket $ is homotopy equivalent to a bounded complex.
\end{enumerate}
\end{theorem}
\begin{remark}
This new invariant categorifies the $\sl_2$ Reshetikhin-Turaev invariant (see \S \ref{subsec-eulerChar}) up to a scalar multiple depending only on the colors and numbers of marked points.  By choosing exactly 1 marked point on each component of $L$, we can forget the data of markings altogether.  The resulting unmarked invariant is defined for tangles, but respects gluing only up to a direct sum.  One may call such an invariant \emph{quasi-local}.
\end{remark}

\begin{remark}
To obtain an actual homology theory, one must apply the functor $\Hom(\emptyset, -)$ to $\llbracket D, \mathcal{K}\rrbracket$, which is defined only if the tangle underlying $D$ is actually a knot or link $L$. We will denote the homology of this complex of $\Z$-modules by $H_{\sl_2}(L;\mathcal{K})$.
\end{remark}

In the special case $\mathcal{P}=\{P_1,P_2\ldots,\}$ is the family of Cooper-Krushkal projectors, we call the homology $H_{\sl_2}(L;\mathcal{P})$ the Cooper-Krushkal homology of $L$.  Actually, this homology is dual to that constructed in \cite{CK12a} (see Observation \ref{obs-CKhomology}).  In \S \ref{subsec-CKhomology}, we prove that the $\Z[u_1,\ldots,u_n]$-action on $P_n$ descends to a well-defined action on link homology:
\begin{theorem}
Let $L\subset \R^3$ be a framed, oriented link whose components are colored $n_1,\ldots,n_r$.  Let $R(L)$ denote the tensor product
\[
R(L):=\bigotimes_{i=1}^r\Z[u_1,\ldots,u_{n_i}]
\]
graded so that $\deg(u_k)=(2-2k,2k)$.  Then the Cooper-Krushkal homology $H_{\sl_2}(L;\mathcal{P})$ is a well-defined isomorphism class of finitely generated, bigraded $R(L)$-modules.
\end{theorem}

Finally, the various link homologies can be related by spectral sequences (see \S \ref{subsec-CKhomology}):
\begin{theorem}
Fix a family $\mathcal{K}=\{K_1,K_2,\ldots,\}$ of quasi-projectors, and let $L\subset S^3$ be a colored, framed, oriented link.  There is a polynomial ring $R(L)$ and a spectral sequence of bigraded $R(L)$-modules $R(L)\otimes_\Z H_{\sl_2}(L;\mathcal{K})\Rightarrow H_{\sl_2}(L;\mathcal{P})$.
\end{theorem}

\subsection{Connection to conjectures of Gorsky-Oblomkov-Rasmussen-Shende}
Recent work \cite{GOR12,GORS12} has indicated that the Khovanov and Khovanov-Rozansky homology of the torus knots are very interesting objects, and are related to affine Lie algebras.  This surprising connection itself has its origins in a fascinating conjecture \cite{ORS12} relating the triply-graded Khovanov-Rozansky homology of an algebraic link with the Hilbert scheme of points on its defining complex curve.  This is a very exciting area of research, and reflects that there is much to discover in the landscape of link homology.

Of specific importance to us is the following conjecture, which appears in \cite{GOR12}:
\begin{conjecture}\label{introGorConj}
Let $V_n=\Z[u_1,\ldots,u_n]\otimes \Lambda[\xi_1,\ldots,\xi_n]$ denote the differential bigraded algebra with bigrading $\deg(u_m) = (2-2m,2m)$, $\deg(\xi_m) = (1-2m,2+2m)$, and differential given by
\[
d(u_m)=0\hskip.5in d(\xi_m) = \sum_{i+j=m+1}u_iu_j
\]
for all $1\leq m\leq n$, together with the graded Leibniz rule.  Then the homology of $V_n$ is isomorphic to the limiting Khovanov homology of the $(n,r)$ torus links as $r\to \infty$.
\end{conjecture}

In this paper we make significant progress toward this conjecture.  It is known that the limiting Khovanov homology of $(n,r)$ torus links as $r\to \infty$ is isomorphic to the homology of the closure of the categorifiied Jones-Wenzl projector \cite{Roz10a}.  This is simply the Cooper-Krushkal homology of the $n$-colored unknot.  We know that this homology coincides with the homology of the differential bigraded algebra $\Endg(P_n)$.  The additional structure on $P_n$ afforded by the expression (\ref{eq-periodicPn}) allows us to greatly simplify this latter algebra (see Theorem \ref{thm-introGorThm} earlier in the introduction, and \S \ref{subsec-endSimplification} in the main body of the paper).  Specifically, we construct a projector $P_n$ and a deformation retract $\Endg(P_n)\rightarrow \Z[u_1,\ldots,u_n]/(u_1^2)\otimes \Lambda[\xi_2,\ldots,\xi_n]=:W_n$ with some $\Z[u_1,\ldots,u_n]$-equivariant differential.

This is not yet a proof of Conjecture \ref{introGorConj}, since we are unable to give an explicit formula for $d(\xi_k)$.  Moreover, we do not know if the Leibniz rule holds for $W_n$.  On the other hand, the deformation retract $\Endg(P_n)\rightarrow W_n$ endows $W_n$ with the structure of an $A_\infty$ algebra, the existence of which is not apparent in \cite{GOR12}.  We can check directly that $\mu_2(\xi_2,\xi_2)=u_2^3$ is nonzero, whereas the obvious multiplication in $W_n$ gives $\xi_2^2=0$.  So there is the possibility that the $A_\infty$ structure on $W_n$ is interesting.

\subsection{Other Lie algebras}
\label{subsec-otherLieAlgs}
Let $L\subset S^3$ be a framed, oriented link.  Fix a complex semi-simple Lie algebra $\mathfrak{g}$, for example $\mathfrak{g} = \sl_n(\C)$, and label the components of $L$ by finite dimensional irreducible representations of $\mathfrak{g}$.  The Reshetikhin-Turaev invariant of $L$ is an element of $\Z[q,q\inv]$, defined using the braiding operation on tensor products of representations of the quantum group $U_q(\mathfrak{g})$.  In case $\mathfrak{g}=\sl_2$, the corresponding link invariant is called the colored Jones polynomial.  The finite dimensional irreducible representations of $\sl_2$ are determined up to isomorphism by their dimension, so the colored Jones polynomial is naturally an invariant of framed, oriented links $L\subset S^3$ whose components are labelled by non-negative integers, called the colors.  The color $n$ corresponds to the $n+1$ dimensional representation.

More generally, the finite dimensional irreducible representations of $\sl_N(\C)$ up to isomorphism are indexed by partitions (that is, non-increasing sequences of integers) $\lambda = (\lambda_1,\ldots,\lambda_N)$ with $\lambda_N=0$.  The representation associated with $\lambda$ is denoted $L(\lambda)$.  If $\omega_i = (1,\ldots,1,0,\ldots,0)$ with $i$ ones and $n-i$ zeroes, then $L(\omega_i) = \Lambda^i(\C^N)$ is the $i$-th exterior power of the standard representation.  At the other extreme, if $\lambda=(n,0,\ldots,0)$, then $L(\lambda)=\Sym^n(\C^N)$.  The minuscule representations are precisely the exterior powers $\Lambda^i(\C^N)$, for $i=1,2,\ldots,N-1$.

We list a small sample of the many various ways in which the $\sl_N$ polynomial has been categorified, for various colors and various $N$.  For the uncolored $\sl_N$ invariant see \cite{Kh00, Kh04, KR08, MSV09,CauKam08b}, for the $\Lambda^i(\C^N)$-colored invariant see \cite{Wu09, Sus07, Man07}, for the $\Sym^i(\C^N)$ colored invariant $(N\in\{2,3\})$, see \cite{Kh05, CK12a,Roz10a,FSS12,Rose12}, for arbitrary colors and arbitrary $N$, see \cite{W10b, Cau12}. Except for \cite{Kh05}, these are all expected to be isomorphic or, at worst, related by Koszul duality \cite{SS11}.  In all of these examples except \cite{Kh05}, the homology of the $V$-colored unknot is infinite dimensional unless $V=\Lambda^i(\C^N)$.

All of the results in this paper concern the $\Sym^n(\C^2)$-colored $\sl_2$ link invariant, but we expect our results to extend without difficulty to the $\Sym^n(\C^N)$-colored $\sl_N$ link invariant, in an essentially obvious way.  The diagram $\pic{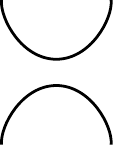}$ gets replaced by its web analogue: the letter ``I'', and Bar-Natan's cobordism categories get replaced by categores of $\sl_N$ matrix factorizations \cite{KR08} or foams \cite{MSV09}, or perhaps Soergel bimodules \cite{Kh07}.

\subsection{Organization of the paper}

In \S \ref{sec-TL} we recall Bar-Natan's categorification of the\\ Temperley-Lieb algebras $\TL_n$ and the Cooper-Krushkal categorification of the Jones-Wenzl projector $p_n\in \TL_n$.

In \S \ref{sec-quasiprojectors} we show how to construct complexes $Q_n$ over Bar-Natan's categories which categorify the expressions $(1-q^{2n})p_n\in\TL_n$ and find a unique characterization of the $Q_n$.  We then establish some basic properties, particularly the relationship with the Cooper-Krushkal projectors $P_n$, from which our polynomial action will originate.

In \S \ref{sec-linkHomology} we show that the $Q_n$ can be used to construct a new family of colored $\sl_2$ link homologies, and give a characterization of which of these give finite homologies.  These new homologies retain a close relationship with Cooper-Krushkal homology, in the form of a certain spectral sequence.  As a by-product, we will see the Cooper-Krushkal homology can be refined to an invariant which takes values in finitely generated bigraded modules over a polynomial ring, up to isomorphism.

Finally, the appendix \S\ref{sec-homAlg} introduces some basic notions in homological algebra such as convolutions and deformation retracts, and establishes some useful tools for simplifying some of the unbounded chain complexes which appear in this subject.

\subsection*{Acknowledgements}
This work forms a part of my Ph.D.~thesis, and I would like to thank my advisor, Slava Krushkal,  and also Ben Cooper for many helpful conversations.  I would also like to thank the Max Planck Institute for Mathematics in Bonn for their hospitality, support, and excellent working conditions.

\section{The Temperley-Lieb algebra and its categorification}
\label{sec-TL}
The $\sl_2$ quantum invariant for tangles is defined via a braid group action on the Temperley-Lieb algebras $\TL_n$ together with certain idempotent elements $p_n\in\TL_n$ called the \emph{Jones-Wenzl projectors}.   In this section we study the categorification of $p_n$ due to Cooper-Krushkal \cite{CK12a}, in the setting of Bar-Natan's categories \cite{B-N05} or, equivalently categories of modules over Khovanov's rings $H^n$ \cite{Kh02}.  We also set up some basic theory involving the Cooper-Krushkal projectors which we will use later.

\subsection{The Temperley-Lieb algebras and Jones-Wenzl projectors}
\label{TLsubSection}
Let $\TL^m_n$ be the $\C(q)$-vector space generated by properly embedded 1-submanifolds of the rectangle $[0,1]^2$ with boundary equal to a standard set of $m$ points $\{(k,(m+1))\:|\: k=1,\ldots,m\}$ on the ``top'' of the rectangle and $n$ points $\{(k,(m+1))\:|\: k=1,\ldots,m\}$ on the ``bottom'' of the rectangle.  Here $\C(q)$ is the field of rational functions in an indeterminate $q$.  We regard the generators modulo planar isotopy and the relation $D\sqcup U = (q+q\inv)D$, where $U$ is a circle disjoint from the rest of the diagram.  By a \emph{diagram} or a Temperley-Lieb diagram, we will simply mean the image of a 1-manifold with no circle components inside $\TL^m_n$.

We have a bilinear map $\TL^m_k\times \TL^k_n\rightarrow \TL^m_n$ given by vertical stacking, which we denote by $a\cdot b$, or simply $ab$.  This makes the collection of spaces $\TL^m_n$ into a $\C(q)$-linear category $\TL$ with objects given by non-negative integers and morphisms $n\rightarrow m$ given by elements of $\TL^m_n$.  In particular, composition makes the vector space $\TL_n:=\TL^n_n$ into a unital algebra, called the \emph{Temperley-Lieb algebra} on $n$ strands.  The identity is the diagram $1_n$ consisting of $n$-vertical strands.  For a diagram $a\in \TL^m_n$, define the \emph{through degree} $\tau(a)$ to be the minimal $k$ such that $a=b\cdot c$ with $b\in\TL^m_k$, $c\in\TL^k_n$.  For a linear combination $b=\sum_a f_a a$ of diagrams, let $\tau(b):=\max\{\tau(a)\:|\: f_a\neq 0\}$.  
\begin{figure}[ht]
	\centering $\bpic{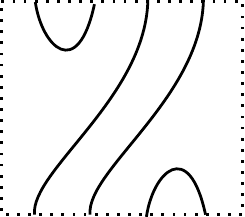} \cdot \bpic{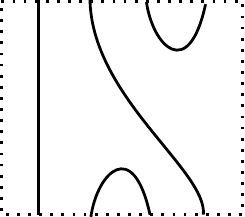} = \bpic{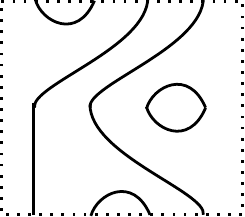}  = (q+q\inv)\ \bpic{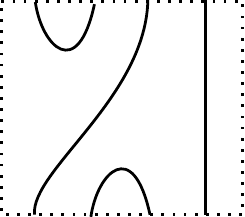}$
	\caption{Multiplication in $\TL_4$.  Each of the diagrams above has through degree 2.} 
	\label{TLfig}
	\end{figure}

The following is classical \cite{Wz87,FK97}, and defines the Jones-Wenzl projectors $p_n\in \TL_n$:
\begin{theorem}\label{thm-JWproj}
There is a unique element $p_n\in \TL_n$ satisfying
\begin{itemize}
\item[(JW1)] $p_n = 1_n + a$ with $\tau(a)<n$.
\item[(JW2)] $a\cdot p_n = p_n\cdot b = 0$ whenever $\tau(a),\tau(b)<n$.
\end{itemize}
\end{theorem}
We refer to axiom (JW2) by saying that $p_n$ \emph{kills turnbacks}.  Indeed, using the graphical notation in which we denote $a$ parallel strands by $\pic{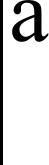}$ and $p_n:=\pic{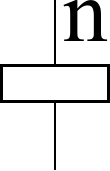}$, this axiom becomes equivalent to
\[
\wwpic{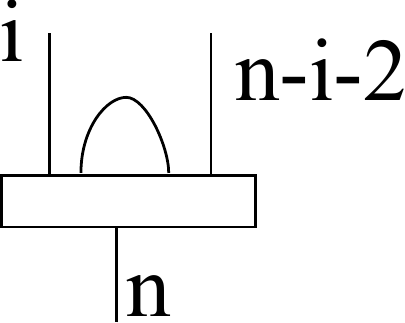} =0 = \wwpic{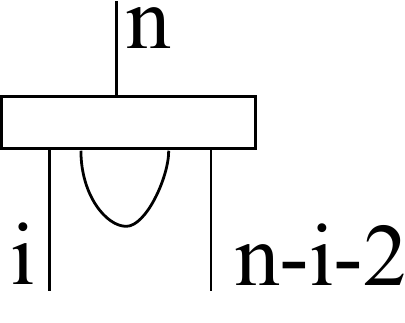} 
\]
for $0\leq i\leq n-2$.  Similarly, if $f\in \TL_n$ is such that $a\cdot f = 0$ (respectively $f\cdot a$) whenever $\tau(a)<n$, then we say $f$ kills turnbacks from above (respectively below).  An explicit description of $p_n$ is given by $p_1=1_1\in\TL_1$ together with the recursion
\begin{equation}\label{eq-JWrecursion}
p_n = \FKzero - \frac{[n-1]}{[n]}\FKone +\frac{[n-2]}{[n]}\FKtwo -\cdots \pm \frac{1}{[n]} \FKult
\end{equation}
where the white box denotes $p_{n-1}$ and $[k]=\frac{q^k-q^{-k}}{q-q\inv}$ is the quantum integer.  One can find this result in \cite{FK97} with a different sign convention.  

\subsection{The tangle categories}
\label{subsec-tangleCat}
In \cite{B-N05} Bar-Natan interprets the Temperley-Lieb diagrams as objects of a category in which the morphisms ensure that the Temperley-Lieb relations lift to isomorphisms.  The paper \cite{B-N05} contains an excellent exposition, and we refer the reader to \cite{H12a} for more details regarding our specific conventions.
\begin{definition}\label{def-cobCat}
For each integer $n\geq 0$, fix a standard set $B_n\subset \partial D^2$ of $2n$ points, and define a category $\Cob_n$ as follows:
\begin{itemize}
\item The objects of $\Cob_n$: symbols $q^{j} T$, where $T\subset D^2$ is a properly embedded 1-submanifold with boundary $\partial T = B_n$.
\item A morphism $f:q^iT\rightarrow q^j T'$ is a formal $\Z$-linear combination of cobordisms $T\rightarrow T'$ in $D^2\times [0,1]$, decorated with dots, regarded modulo (1) isotopy of the underlying surfaces (rel boundary), (2) dots are allowed to move freely about the components of the cobordism, and (3) the following local relations:
\vskip 4pt
	\begin{enumerate}
		\item $\bpic{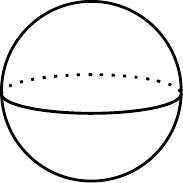}\ =0$, $\bpic{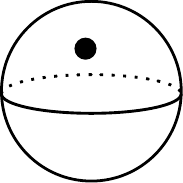}  = 1$, $\bpic{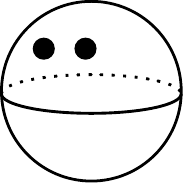} =0$, and $\bpic{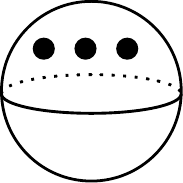} =0$\vskip6pt
		\item $\bpic{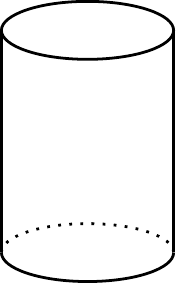} =\bpic{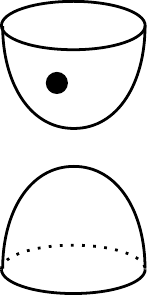} +\ \bpic{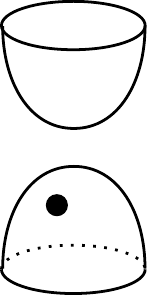} $\vskip6pt
		\item $\bpic{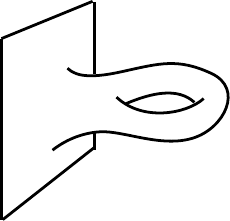}\  = 2\:\bpic{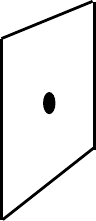} $.\vskip6pt
	\end{enumerate}
Here, a cobordism $S:q^iT\rightarrow q^jT'$ is a properly embedded surface $S\in D^2\times I$ with boundary $\partial S = (T\times\{0\})\cup (T'\times \{1\})\cup (B_n\times [0,1])$.  The degree of $S:q^iT\rightarrow q^jT'$ is defined by
\[
\deg_q(S)=n+j-i-\chi(S) + 2(\#\text{ of dots}) % 0 = 2 + 0 - 1 -1
\]
where $\chi(S)$ is the Euler characteristic of the surface $S$, and we allow only homogeneous morphisms of $\deg_q$ zero.
\end{itemize}
Composition of morphisms in $\Cob_n$ is induced by gluing of cobordisms, extended bilinearly to arbitrary morphisms.  Since Euler characteristic is additive under gluing, the composition of degree zero morphisms is again degree zero, so that $\Cob_n$ is well-defined.
\end{definition}

The categories $\Cob_n$ do not contain all direct sums, so we formally add them, obtaining a category whose objects are formal direct sums of objects of $\Cob_n$ and whose morphisms are matrices of morphisms of morphisms in $\Cob_n$ with the appropriate source and target.
\begin{definition}\label{def-TLcat}
Let $\bn{n}$ denote the category with objects the symbols $\bigoplus_{k=1}^r a_k$ where $a_i\in\Cob_n$ $(1\leq i\leq n)$.  Morphisms in $\bn{n}$ from $\bigoplus_{k=1}^r a_k\rightarrow \bigoplus_{l=1}^s b_l$ are matrices $f = (f_{ij})$ where $f_{ij}\in \Hom_{\Cob_n}(a_j, b_i)$.  Composition is given by matrix multiplication $(f\circ g)_{ij} = \sum_{k}f_{ik}\circ g_{kj}$.
\end{definition}

\begin{definition}\label{def-kom}
Let $\Kom(n):=\Kom(\bn{n})$ denote the category of potentially unbounded chain complexes over $\bn{n}$ with morphisms given by degree zero chain maps.  Similarly define $\Kom^+(n)$, $\Kom^-(n)$, $\Kom^b(n)\subset \Kom(n)$ to be the full subcategories of $\Kom(n)$ consisting of chain complexes which are bounded from below, respectively bounded from above, respectively bounded, in homological degree.  We denote homotopy equivalence of complexes by $\simeq$.
\end{definition}

We will always draw our diagrams in a rectangle, so that the set of $2n$ distinguished boundary points corresponds to $n$ standard points on the top $[0,1]\times \{1\}$,  respectively bottom $[0,1]\times \{0\}$ of $D^2 = [0,1]\times [0,1]$.  Various planar operations can be regarded as multilinear functors among the $\bn{n}$, via gluing tangles (resp. cobordisms) together in the obvious manner.
\begin{definition}\label{def-planarComp}
Let $\otimes: \bn{n}\times \bn{n}\rightarrow \bn{n}$ be the bilinear functor induced by vertical stacking, so that $a\otimes b$ is `$a$ on top of $b$.'  Let $\sqcup:\bn{n}\times \bn{m}\rightarrow \bn{n+m}$ be the bilinear functor induced by horizontal juxtaposition.  %Let $r:\bn{n}\rightarrow \bn{n}$ be induced by rotation by $180$ degrees:
%\[
%r(a) = \pic{a_rotated}
%\]
Let $T:\bn{n}\rightarrow \bn{n-1}$ be the \emph{partial trace} functor:
\[
T(a) = 
\begin{minipage}{.45in}
\labellist
\small
\pinlabel $a$ at 38 40
\endlabellist
\begin{center}\includegraphics[scale=.3]{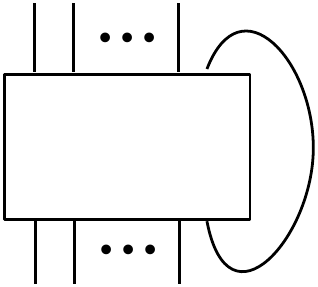}\end{center} 
\end{minipage}
\]
\end{definition}

\begin{proposition}\label{prop-planarCompExt}
Each of the (multi-linear) functors of Definition \ref{def-planarComp} has an extension to the relevant categories $\Kom^\pm(n)$ of semi-infinite chain complexes.
\end{proposition}
\begin{proof}
It is an easy fact (see, for example, \cite{H12a}) that if $\mathscr{A}_i$, $\mathscr{B}$ are additive categories then any multilinear functor $\mathscr{A}_1\times\cdots \times \mathscr{A}_r\rightarrow \mathscr{B}$ extends naturally to a multilinear functor $\Kom^-(\mathscr{A})_1\times \cdots \times \Kom^-(\mathscr{A})_r\rightarrow \Kom^-(\mathscr{B})$ of categories of semi-infinite chain complexes and degree zero chain maps.  The precise formulae recall the definition of the tensor product of chain complexes of abelian groups via the usual Koszul sign rule on differentials.
\end{proof}

In this paper, whenever we write $f:A\rightarrow B$ or $f\in\Hom(A,B)$, we mean that $f$ is a chain map which is homogenous of homological and $q$-degree zero.  It is often convenient to assemble the maps of arbitrary degree (not necessarily compatible with the differentials) into a chain complex:
\begin{definition}\label{def-homg}
For complexes $A,B\in\Kom(n)$, let $\Homg_{\bn{n}}(A,B)$ denote the chain complex generated by bihomogeneous maps of arbitrary bidegree and differential given by the super-commutator $[d,f]=d_B\circ f - (-1)^{|f|}f\circ d_A$.  By an element of this hom complex we will always mean a bihomogeneous element, and we let $\deg(f) = (\deg_h(f), \deg_q(f))$ denote the bidegree.  We often write $\deg_h(f)=|f|$ and $\Homg = \Homg_{\bn{n}}$.
\end{definition}
The $\Homg$ complex is a bigraded abelian group
\[
\Homg^{i,j}(A,B)=\prod_{k\in\Z}\Hom_{\bn{n}}(q^jA_k, B_{k+i})
\]
with a differential of bidegree $(1,0)$.

The bidegree $(i,j)$ cycles (respectively, boundaries) of $\Homg^{i,j}(A,B)$ are precisely chain maps (respectively null-homotopic chain maps) $t^i q^jA\rightarrow B$, where $t$ and $q$ denote the functors $\Kom(n)\rightarrow \Kom(n)$ given by shifting upward in homological and $q$-degree, respectively.  Our convention for behavior of the differentials is $d_{tA} = -d_A$ and $d_{qA}=d_A$, appropriately interpreted. 
\begin{definition}
Let $\Ext^{i,j}(A,B)$ denote the $(i,j)$-th homology group of $\Homg(A,B)$, which is simply the group of chain maps $t^iq^jA\rightarrow B$ modulo chain homotopy.
\end{definition}

\subsection{Some mapping cone lemmas}
\label{subsec-cones}
It seems worthwhile to pause and recall some basic facts of mapping cones which will be useful in the sequel.

\begin{definition}\label{def-cone}
Suppose $A,B$ are chain complexes over an additive category, and $f:A\rightarrow B$ is a chain map (all arrows are assumed to have degree zero).  The \emph{mapping cone on $f$} is the chain complex $\Cone(f) = t\inv A\oplus B$ with differential
\[
d_{\Cone(f)} = \matrix{-d_A &\\ f& d_B}
\]
We will also write this as $\Cone(f) = (t\inv A \buildrel f\over \longrightarrow B)$, in anticipation for similar notation for convolutions (Definition \ref{def-convolution}).
\end{definition}

\begin{lemma}\label{lemma-mappingConeEquiv}
Let $A,B$ be chain complexes over an additive category.  A chain map $f:A\rightarrow B$ is a homotopy equivalence if and only if $\Cone(f)\simeq 0$.
\end{lemma}
\begin{proof}
This is a standard property of mapping cones.  For an easily accessible proof, see \cite{Rose12}.
\end{proof}

\begin{lemma}\label{lemma-coneInvariance}
Suppose $A,A',B,B'$ are chain complexes over an additive category, and $f:A\rightarrow B$ is a chain map.  Let $\phi:A\rightarrow A'$ and $\psi:B\rightarrow B'$ be homotopy equivalences.  Then
\[
\Cone(f)\simeq \Cone(\psi\circ f\circ \phi\inv)
\]
where $\phi\inv$ denotes a homotopy inverse for $\phi$.
\end{lemma}
\begin{proof}
The differentials on $\Cone(f)$ and $\Cone(\psi f \phi\inv)$ are the matrices
\[
d_{\Cone(f)}=\matrix{-d_A & 0\\ f & d_B} \ \ \ \ \ \ d_{\Cone(\psi f \phi\inv )}=\matrix{-d_{A'} &0 \\ \psi f \phi\inv & d_{B'}} 
\]
Define a chain map $\Phi:\Cone(f)\rightarrow \Cone(\psi f\phi\inv)$ by the matrix
\[
\Phi = \matrix{\phi & 0 \\ -\psi f h & \psi}
\]
where $d_Ah+h d_A=\Id_{A}-\phi\inv \phi$.  Then one can check that $\Cone(\Phi)$ is contractible.  One way to see this is that $\Cone(\Phi)$ can be reassociated into a mapping cone $\Cone\Big(\Cone(\phi)\rightarrow \Cone(\psi)\Big)$.  Now, $\Cone(\phi)$ and $\Cone(\psi)$ are contractible since $\phi$ and $\psi$ are equivalences.  Thus, $\Cone(\Psi)\simeq 0$ by two applications of Gaussian elimination (Proposition \ref{prop-gauss}).  This implies that $\Phi$ is a homotopy equivalence.
\end{proof}

\subsection{Turnback killing and Cooper-Krushkal projectors}
\label{subsec-CKproj}
The turnback killing property plays an important role for Jones-Wenzl projectors, and the same is true of their categorified counterparts.
\begin{definition}\label{def-throughDeg}
Define the \emph{through-degree} of a complex $A\in \Kom(n)$ to be $\tau(A) = \max_a\{\tau(a)\}$, where $a$ ranges over all diagrams appearing as direct summands of chain groups of $A$.
\end{definition}
Note that if $a\in\bn{n}$ has $\tau(n)<n$, then $a$ is a direct sum of objects $e_{i_1}\otimes \cdots\otimes e_{i_r}$, where the $e_i$ are the usual Temperley-Lieb generators.  A chain complex $A$ satisfies $\tau(A)<n$ if and only if each chain group satisfies $\tau(A^i)<n$.
\begin{definition}\label{def-turnbackDeath}
We say that a complex $C\in\Kom^-(n)$ \emph{kills turnbacks from below} (resp.~\emph{above}) if $C\otimes N\simeq 0$ (resp.~$N\otimes C\simeq 0$) for each complex $N\in\Kom^-(n)$ with $\tau(N)<n$.  We say $C$ kills turnbacks if it kills turnbacks from above and below.
\end{definition}
\begin{proposition}\label{prop-turnback}
A complex $C\in \Kom^-(n)$ kills turnbacks from below (resp.~above) if and only if $C\otimes e_i\simeq 0$ (resp.~$e_i\otimes C\simeq 0$) for all Temperley-Lieb generators  $e_i := 1_{n-i}\sqcup e \sqcup 1_{i-1}$, where $e=\pic{turnback}$. 
\end{proposition}
\begin{proof}
If $C$ kills turnbacks from below, then $C\otimes e_i\simeq 0$, since $\tau(e_i)<n$ for all $i=1,\ldots,n-1$.

Conversely, suppose $C\otimes e_i\simeq 0$ for each $i=1,\ldots,n-1$. If $a\in\bn{n}$ is any diagram with $\tau(a)<n$, then $a$ is isomorphic to a direct sum of objects $e_{i_1}\otimes\cdots \otimes e_{i_r}$, since the $e_i$ generate the non-identity Temperley-Lieb diagrams.  It follows that $C\otimes a\simeq 0$ whenever $a\in\bn{n}$ satisfies $\tau(a)<n$.  If $N\in \Kom^-(n)$ has $\tau(N)<n$, then each chain group satisfies $\tau(N_i)<n$.  Then $C\otimes A = \Tot(\cdots \rightarrow C\otimes N_i\rightarrow C\otimes N_{i+1}\rightarrow \cdots)$ is contractible by Theorem \ref{convSdrthm}, so $C$ kills turnbacks from below, by definition.
\end{proof}

We now discuss the categorified Jones-Wenzl projectors, following Cooper-Krushkal \cite{CK12a}.  Our exposition differs from that in \cite{CK12a} in a few ways.  Firstly, we consider complexes which are bounded from above in homological degree, rather than below.  The reason for this choice is so that the homology of the uknot is naturally a unital algebra, rather than a counital coalgebra.  The two conventions are related by the contravariant duality functor $(-)^\vee:\Kom^\pm(n)\rightarrow \Kom^{\mp}(n)$ which reverses all degrees and flips cobordisms upside-down.  Secondly, we will work primarily with a definition which is slightly more general than the projectors considered in \cite{CK12a}.  We take the liberty of naming our objects after Cooper-Krushkal, despite these slight differences.
\begin{definition}\label{def-CKproj}
A \emph{Cooper-Krushkal projector} is a pair $(P_n,\iota)$ where $P_n\in\Kom^-(n)$ is a complex, and $\iota:1_n\rightarrow P_n$ is a chain map satisfying
\begin{itemize}
\item[(CK1)] $\Cone(\iota)$ is homotopy equivalent to a complex with through degree $<n$.  
\item[(CK2)] $P_n$ kills turnbacks.
\end{itemize}
The map $\iota$ is called the \emph{unit} of the projector $P_n$.
\end{definition}
We often drop the unit $\iota:1_n\rightarrow P_n$ from the notation, and simply call $P_n$ a Cooper-Krushkal projector.

Combining axioms (CK1) and (CK2), we see that
\begin{equation}\label{eq-PnConeIota}
P_n\otimes \Cone(1_n\rightarrow P_n)\simeq 0.
\end{equation}
This equivalence gives us a good notion of categorical idempotents.  Firsly, note that on the level of Euler characteristic, equivalence (\ref{eq-PnConeIota}) becomes $p_n(p_n-1)=0$ which characterizes the usual notion of idempotents.  Moreover (\ref{eq-PnConeIota}) implies $P_n\otimes P_n \simeq P_n$, but is stronger in the sense that the equivalence $P_n\rightarrow P_n\otimes P_n$ is induced from a map $1_n\rightarrow P_n$.  One important consequence is that the homotopy category of complexes such that $P_n\otimes C\simeq C$ is triangulated (see Remark \ref{remark-catIdemp}).

Let us relate our projectors to those considered in \cite{CK12a}.  Note that there is a projection $\Cone(\iota)\rightarrow t\inv 1_n$, and the mapping cone satisfies
\[
P_n \simeq \Cone\Big(t\Cone(\iota)\rightarrow 1_n)\Big)
\]
By axiom (CK1), there is a complex $N\simeq t\Cone(\iota)$ with $\tau(N)<n$.  By lemma \ref{lemma-coneInvariance} we thus have
\begin{equation}\label{eq-CKequivStrongCK}
P_n\simeq \Cone(N\rightarrow 1_n)
\end{equation}
Up to reversing the homological grading convention, the projectors originally considered by Cooper-Krushkal in \cite{CK12a} are all of this more restricted form on the right-hand side above.  They were called \emph{universal projectors} in \cite{CK12a}.  We will call them \emph{strong Cooper-Krushkal projectors}:
\begin{definition}\label{def-strongCKproj}
A \emph{strong Cooper-Krushkal projector} is a chain complex $P_n\in\Kom^-(n)$ such that
\begin{enumerate}
\item $P_n = \Cone(N\buildrel f\over\longrightarrow 1_n)$ for some chain complex $N\in\Kom^-(n)$ with $\tau(N)<n$.
\item $P_n$ kills turnbacks.
\end{enumerate}
\end{definition}
\begin{remark}
Definition \ref{def-strongCKproj} is not preserved under homotopy equivalences, in the sense that there exist complexes $C\in\Kom^-(n)$ which are homotopy equivalent to strong Cooper-Krushkal projectors, but which are \emph{not} strong Cooper-Krushkal projectors.  This is our main reason for defining Cooper-Krushkal projectors as we have.  Nonetheless, the existence of strong Cooper-Krushkal projectors is often useful for technical reasons, as in Proposition \ref{prop-projAbsorb}.
\end{remark}
The following is clear:
\begin{proposition}\label{prop-strongCKproj}
If $P_n\in\Kom^-(n)$ is a strong Cooper-Krushkal projector then $(P_n,\iota)$ is a Cooper-Krushkal projector, where $\iota:1_n\rightarrow P_n$ is the inclusion $1_n\rightarrow \Cone(N\rightarrow 1_n)$.\qed
\end{proposition}

As we have a slightly different set of axioms than in \cite{CK12a}, we will develop the theory from our point of view. 
\begin{proposition}\label{prop-tensorProj}
Suppose $(A,\iota_A)$ and $(B,\iota_B)$ are Cooper-Krushkal projectors in $\Kom(n)$.  Then $(A\otimes B, 1_A\otimes 1_B)$ is a Cooper-Krushkal projector.
\end{proposition}
\begin{proof}
Clearly $A\otimes B$ kills turnbacks, since $A$ and $B$ do.  Consider the following chain complex:
\[
C = \left(\begin{diagram}[size=3em] t\inv A & \rTo^{\Id_A\otimes \iota_B} & A\otimes B\\
& \rdTo^{-\Id_A} & \\
t\inv 1_n & \rTo^{\iota_A} & A  \end{diagram}\right)
\]
Contracting the isomorphism (Gaussian elimination, Proposition \ref{prop-gauss}) , we see that $C\simeq \Cone(\iota_A\otimes \iota_B)$.  On the other hand, the rows are $A\otimes \Cone(\iota_B)$ and $\Cone(\iota_A)$, each of which is homotopy equivalent to a complex with through-degree $<n$ by axiom (CK1) of definition \ref{def-CKproj}.  It follows that $\Cone(\iota_A\otimes \iota_B)$ is equivalent to a complex with through-degree $<n$, by Lemma \ref{lemma-coneInvariance}.  Thus $(A\otimes B,\iota_A\otimes\iota_B) $ is a Cooper-Krushkal projector.
\end{proof}

\begin{proposition}\label{prop-turnbackTFAE}
Let $C\in \Kom^-(n)$ be arbitrary, and let $(P_n,\iota)$ be a Cooper-Krushkal projector.  The following are equivalent:
\begin{enumerate}
\item $C$ kills turnbacks from below.
\item $C\otimes \Cone(\iota)\simeq 0$.
\item $C\otimes P_n\simeq C$.
\end{enumerate}
Similarly, the following are equivalent:
\begin{itemize}
\item[(4)] $C$ kills turnbacks from above.
\item[(5)] $\Cone(\iota)\otimes C \simeq 0$.
\item[(6)] $P_n\otimes C\simeq C$.
\end{itemize}
\end{proposition}
\begin{proof}
(1)$\Rightarrow$(2). Assume that (1) holds.  By axiom (CK2) $\Cone(\iota)$ is equivalent to a complex with $\tau<n$, hence $C\otimes \Cone(\iota)\simeq 0$, which is (2).

(2)$\Rightarrow$(3). Assume that (2) holds.  Note that $\Cone(\Id_C\otimes \iota)$ is isomorphic to the contractible complex $C\otimes \Cone(\iota)$.   It is a standard property of mapping cones that a chain map $f$ is a homotopy equivalence if and only if $\Cone(f)\simeq 0$.  Thus $\Id_C\otimes \iota$ defines a homotopy equivalence $C\simeq C\otimes P_n$, which is (3).

(3)$\Rightarrow$(1). Obviously, if $C\simeq C\otimes P_n$, then $C$ kills turnbacks from below.

A similar argument shows that (4), (5), and (6) are equivalent.
\end{proof}
We refer to the equivalence of (1) and (3) in the previous proposition as the \emph{projector absorbing} property.  As an immediate consequence we have:
\begin{corollary}\label{cor-CKprojAreIdemp}
Cooper-Krushkal projectors are idempotent: $P_n\otimes P_n\simeq P_n$.
\end{corollary}
\begin{proof}
Take $C = P_n$ in Proposition \ref{prop-turnbackTFAE}.
\end{proof}
\begin{remark}\label{remark-catIdemp}
We may think of the functor $P_n\otimes (-)$ as a categorified projection operator.   Denote by $P_n\otimes \Kom^-(n)$ the image of this functor, that is, the full subcategory of $\Kom^-(n)$ consisting of complexes which are fixed by $P_n\otimes(-)$ up to homotopy equivalence.  Proposition \ref{prop-turnbackTFAE} implies that $C$ is an object of $P_n\otimes \Kom^-(n)$ if and only if $C$ annihilates $\Cone(\iota)$.  This latter condition is obviously closed under taking mapping cones, so the homotopy category of $P_n\otimes \Kom^-(n)$ is triangulated.  This is a desirable property of categorified idempotents.
\end{remark}
As another application of Proposition \ref{prop-turnbackTFAE}, we have:

\begin{corollary}\label{cor-CKprojAreUnique}
Cooper-Krushkal projectors are unique up to homotopy equivalence.
\end{corollary}
\begin{proof}
If $(P_n,\iota)$ and $(P_n',\iota')$ are two Cooper-Krushkal projectors then $P_n'\otimes P_n\simeq P_n'$ and $P_n'\otimes P_n\simeq P_n$ by two applications of projector absorbing.
\end{proof}
\begin{remark}
In \S \ref{subsec-canMaps} we prove that any two Cooper-Krushkal projectors are, in fact,  \emph{canonically equivalent}.
\end{remark}

\begin{proposition}\label{prop-projCommutingProp}
Let $(P_n,\iota)$ be a Cooper-Krushkal projector.  We have $P_n\otimes A\simeq A\otimes P_n$ for all complexes $A\in \Kom^-(n)$.
\end{proposition}
\begin{proof}
Note that the set of complexes $C$ with $\tau(C)<n$ forms a tensor ideal in $\Kom^-(n)$.  That is, if $\tau(C)<n$, then $\tau(A\otimes C)<n$ and $\tau(C\otimes A)<n$ for all $A\in\Kom^-(n)$.

Let $(P_n,\iota)$ be a Cooper-Krushkal projector, and let $A\in \Kom^-(n)$ be arbitrary.  From the Cooper-Krushkal axioms, and the above remarks, $A\otimes \Cone(\iota)$ is equivalent to a complex with through degree $<n$, hence $P_n\otimes A\otimes \Cone(\iota)\simeq 0$ since $P_n$ kills turnbacks.  Proposition \ref{prop-turnbackTFAE} now implies that $P_n\otimes A \simeq P_n\otimes A\otimes P_n$.  An entirely symmetric argument establishes that this latter complex is homotopy equivalent to $A\otimes P_n$, as well.  This completes the proof.
\end{proof}

The following simplifies the task of showing that a complex kills turnbacks.  We remark that its proof assumes the existence of Cooper-Krushkal projectors.

\begin{corollary}\label{cor-turnbackKillingSymmetry}
Let $C\in\Kom^{\pm}(n)$ be arbitrary.  Then $C$ kills turnbacks from above if and only if $C$ kills turnbacks from below.
\end{corollary}
\begin{proof}
It suffices to prove the proposition in the case where $C\in \Kom^-(n)$.  The case $C\in \Kom^+(n)$ will follow by symmetry.  If $C$ kills turnbacks from below, then
\[
C\simeq C\otimes P_n \simeq P_n\otimes C,
\]
which kills turnbacks from above since $P_n$ does.  The first equivalence is projector absorbing, and the second is from Proposition \ref{prop-projCommutingProp}.  A similar argument shows that if $C$ kills turnbacks from above, then $C$ kills turnbacks from below as well.
\end{proof}

Finally, for technical reasons, we will want to strengthen the notion of projector absorbing.
\begin{proposition}[Projector absorbing]\label{prop-projAbsorb}
Suppose $A\in\Kom^-(n)$ kills turnbacks, and $(P_n,\iota)$ is a strong Cooper-Krushkal projector.  Then there are deformation retracts (Definition \ref{def-sdr}) $A\otimes P_n \rightarrow A$ and $P_n\otimes A \rightarrow A$.
\end{proposition}
\begin{proof}
We will only prove the first statement.  The second is similar.  From the definitions, one sees that a strong Cooper-Krushkal projector $P_n$ can be written as
\[
P_n \cong N\oplus 1_n \ \ \ \  \ \text{with differential} \ \ \ \ \ 
\matrix{d_N & 0 \\ \d & 0}
\]
for some $\d\in\Homg^{1,0}(N,1_n)$, where $\tau(N)<n$, and where $\iota$ is the inclusion of the $1_n$ subcomplex.   Thus
\[
A\otimes P_n \cong (A\otimes N)\oplus A \ \ \ \  \ \text{with differential} \ \ \ \ \ 
\matrix{d_{A\otimes N} & 0 \\ \Id_A\otimes \d & d_A}
\]
If $A$ kills turnbacks, then $A\otimes N\simeq 0$, and we can apply Gaussian elimination (Proposition \ref{prop-gauss}), obtaining a deformation retract $A\otimes P_n\rightarrow A$ as in the statement.  
\end{proof}
%\begin{remark}
%Proposition \ref{turnbackProp} illustrates what we feel to be the correct notion of categorification of eigenvalue and eigenvector.  A categorical eigenvalue $\Lambda$ of the endofunctor $A\otimes(-)$ should be an iterated extension of $C$ by shifted copies of $1_n$ such that there exists a complex $C$ such that $\Lambda\otimes C\simeq 0$.
%\end{remark}

\subsection{Canonical maps}
\label{subsec-canMaps}
We can use the existence of duals in $\bn{n}$ to compute some Hom complexes, and show that any two Cooper-Krushkal projectors $P_n, P_n'$ are canonically equivalent.  Most of the results in this section are explained in \cite{H12a} in more detail.
\begin{theorem}\label{thm-partialTraceHom}
For $M\in \Kom(\bn{n-1})$, $N\in\Kom(\bn{n})$, we have natural isomorphisms
\begin{enumerate}
\item $\Homg_{\bn{n}}(M\sqcup 1, N) \cong \Homg_{\bn{n-1}}(M, qT(N))$
\item $\Homg_{\bn{n}}(N,M\sqcup 1) \cong \Homg_{\bn{n-1}}(T(N), qM)$
\end{enumerate}
where $T(a) = \begin{minipage}{.45in}
\labellist
\small
\pinlabel $a$ at 38 40
\endlabellist
\begin{center}\includegraphics[scale=.3]{emptybox_partialTrace}\end{center} 
\end{minipage}$ denotes the partial trace functor.
\end{theorem}
\begin{proof}[Idea of proof of (1)]
Let $\phi:\Homg_{\bn{n}}(M\sqcup 1, N) \rightarrow \Homg_{\bn{n}}(M, qT(N))$ be the map sending $f\in \Homg_{\bn{n}}(M\sqcup 1, N) $ to the composition
\begin{diagram}
\begin{minipage}{.3in}
\labellist
\small
\pinlabel $M$ at 35 40
\endlabellist
\begin{center}\includegraphics[scale=.3]{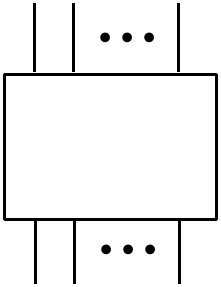}\end{center} 
\end{minipage}
& \rTo &
\begin{minipage}{.5in}
\labellist
\small
\pinlabel $M$ at 35 40
\endlabellist
\begin{center}\includegraphics[scale=.3]{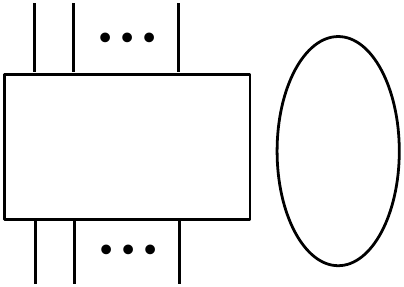}\end{center} 
\end{minipage}
& \rTo^{\begin{minipage}{.45in}
\labellist
\small
\pinlabel $f$ at 35 40
\endlabellist
\begin{center}\includegraphics[scale=.3]{emptybox_partialTrace}\end{center} 
\end{minipage}} &
\begin{minipage}{.5in}
\labellist
\small
\pinlabel $N$ at 30 40
\endlabellist
\begin{center}\includegraphics[scale=.3]{emptybox_partialTrace}\end{center} 
\end{minipage}
\end{diagram}
where the first map is $\Id_M\sqcup S$ where $S=\pic{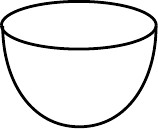}$ is the cup cobordism from $\emptyset$ to the unknot.  Let $\psi:\Homg_{\bn{n}}(M, qT(N))\rightarrow \Homg_{\bn{n}}(M\sqcup 1, N)$ to be the map which sends $g$ to the composition
\begin{diagram}
\begin{minipage}{.4in}
\labellist
\small
\pinlabel $M$ at 35 40
\endlabellist
\begin{center}\includegraphics[scale=.3]{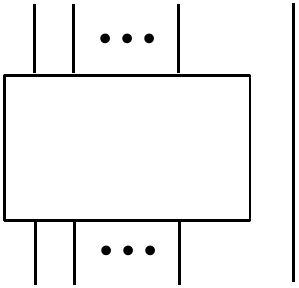}\end{center} 
\end{minipage}
& \rTo^{\begin{minipage}{.45in}
\labellist
\small
\pinlabel $g$ at 35 40
\endlabellist
\begin{center}\includegraphics[scale=.3]{emptybox-strand}\end{center} 
\end{minipage}} &
\begin{minipage}{.5in}
\labellist
\small
\pinlabel $N$ at 35 40
\endlabellist
\begin{center}\includegraphics[scale=.3]{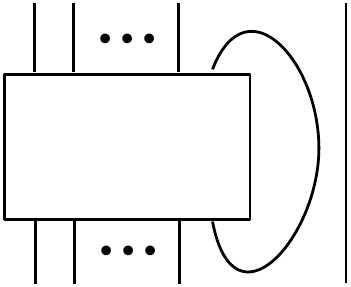}\end{center} 
\end{minipage}
& \rTo &
\begin{minipage}{.5in}
\labellist
\small
\pinlabel $N$ at 30 40
\endlabellist
\begin{center}\includegraphics[scale=.3]{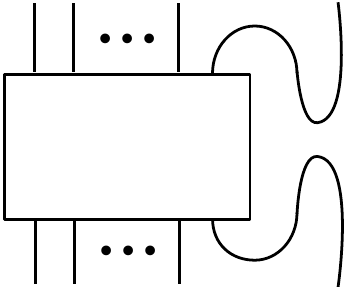}\end{center} 
\end{minipage}
\end{diagram}
where the first map is $g\sqcup 1$ and the second map is given by a saddle cobordism.  We leave it to the reader to check that $\psi$ and $\phi$ are inverse chain maps.  More details can be found in \cite{H12a}.  The proof of (2) is similar.
\end{proof}
\begin{corollary}\label{cor-homRotation}
Let $(-)^\vee:\bn{n}\rightarrow \bn{n}$ denote the contravariant functor which reverses $q$-degree shifts and reflects diagrams about a horizontal axis.  Then for each $a\in\bn{n}$ we have an isomorphism $\Homg_{\bn{n}}(A\otimes a,B)\cong \Homg(A,B\otimes a^\vee)$ which is natural in $A,B\in\Kom(n)$.
\end{corollary}
\begin{proof}
If $e_i\in\bn{n}$ denotes the Temperley-Lieb generator (Definition \ref{def-turnbackDeath}) then we have $e_i^\vee=e_i$ and $(e_{i_1}\otimes \cdots\otimes e_{i_r})^\vee \cong (e_{i_r}\otimes \cdots \otimes e_{i_1})$.  Hence the general case follows from the case $a=e_i$, which in turn follows from two applications of Theorem \ref{thm-partialTraceHom}.  
\end{proof}

\begin{proposition}\label{prop-resProp}
Suppose $A\in\Kom(n)$ kills turnbacks and $P_n$ is a strong Cooper-Krushkal projector  (Definition \ref{def-strongCKproj}).  Then precomposition with $\iota:1_n\rightarrow P_n$ gives a deformation retract $\Homg(P_n,A)\simeq \Hom(1_n,A)$.
\end{proposition}
\begin{proof}
The idea is to rewrite $\Homg(P_n,A)$ as a convolution (using direct product) of terms of the form $\Homg((P_n)_k,A)$, where $(P_n)_k$ denotes the $k$-th chain object.  The turnback killing property together with Corollary \ref{cor-homRotation} implies that most of these terms are contractible, and the only one which survives is $\Homg(1_n,A)$.  Performing the infinitely many contractions is justified by Remark \ref{remark-generalizedConv}.  For more details, see \cite{H12a}.
\end{proof}

In order to refer to homology classes of $\Endg(P_n)$ without choosing a specific representative, we need to show that any two Cooper-Krushkal projectors are not just homotopy equivalent, but \emph{canonically} equivalent.

\begin{definition}\label{def-canEquiv}
Let $(A_n,\iota_A)$ and $(B_n,\iota_B)$ be Cooper-Krushkal projectors.  Let us call a degree $(0,0)$ chain map $\phi:A_n\rightarrow B_n$ \emph{canonical} if $\phi\circ \iota_A \simeq \iota_B$.  
\end{definition}

\begin{theorem}\label{thm-canEquiv}
If $(A_n,\iota_A)$ and $(B_n,\iota_B)$ are Cooper-Krushkal projectors, then a canonical map $\phi:A_n\rightarrow B_n$ exists and is unique up to homotopy.  This map is a homotopy equivalence.  The composition of canonical maps is a canonical map.
\end{theorem}
\begin{proof}
By Proposition \ref{prop-resProp}, precomposition with $\iota_A$ gives an equivalence $\Homg(A_n,B_n)\rightarrow \Homg(1_n,B_n)$.  Taking the preimage of $\iota_B\in \Homg(1_n,B_n)$ gives a chain map $\phi:A_n\rightarrow B_n$ which is uniquely characterized up to homotopy by $\phi\circ \iota_A \simeq \iota_B$.   Thus, canonical maps exist and are unique.

If $\phi:A_n\rightarrow B_n$ and $\psi:B_n\rightarrow C_n$ are canonical maps between Cooper-Krushkal projectors, then $\psi\circ \phi\circ\iota_A\simeq \psi\circ \iota_B\simeq \iota_C$.  Hence the composition of canonical maps is canonical.  This implies immediately that the canonical maps $A_n\rightarrow B_n$ and $B_n\rightarrow A_n$ are homotopy inverses, since any  a canonical map $A_n\rightarrow A_n$ is homotopic to $\Id_A$ by uniqueness.  This completes the proof.
\end{proof}

\subsection{The Cooper-Krushkal recursion}
\label{subsec-CKrecursion}
The Cooper-Krushkal axioms force $P_1\simeq 1_1\in\Kom(1)$.  For $n\geq 2$, the projector $P_n$ is related to $P_{n-1}$ by a formula which categorifies a well-known recursion for ordinary Jones-Wenzl projectors $p_n$.
\begin{definition}
Let $P_{n-1}\in\Kom(n-1)$ be a Cooper-Krushkal projector, and define the \emph{Cooper-Krushkal sequence} (relative to $P_{n-1}$) to be the following semi-infinite sequence of chain complexes and chain maps
\begin{equation}\label{eq-CKsequence}
\begin{diagram}
&&&&&&&&\underline{\FKzero} \\
&&&&&&&& \uTo \\ \\
q^{n-1}\FKult &\rTo & q^{n-2}\FKpenult &\rTo &\cdots & \rTo & q^2\FKtwo & \rTo & q\FKone\\
\uTo &&&&&&&&\\ \\
q^{n+1}\FKult &\lTo & q^{n+2}\FKpenult &\lTo &\cdots & \lTo & q^{2n-2}\FKtwo & \lTo & q^{2n-1}\FKone\\
&&&&&&&&\uTo \\ \\
q^{3n-1}\FKult &\rTo & q^{3n-2}\FKpenult &\rTo &\cdots & \rTo & q^{2n+2}\FKtwo & \rTo & q^{2n+1}\FKone\\
\uTo &&&&&&&&\\ \\
q^{3n+1}\FKult &\lTo & q^{3n+2}\FKpenult &\lTo &\cdots &&&&
\end{diagram}
\end{equation}\vskip7pt
where the white box denotes $P_{n-1}$ and the maps are\vskip7pt
\[
\wwpic{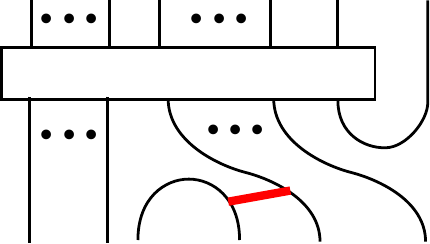} , \ \  \ \ \ \ \wwpic{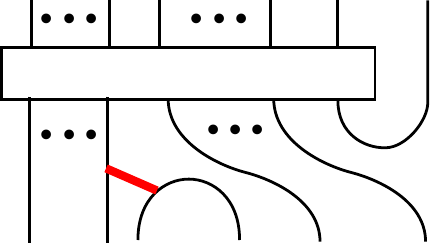} , \ \ \  \ \ \ \ \wwpic{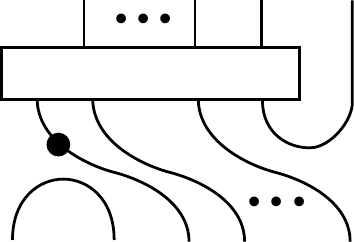} - \wwpic{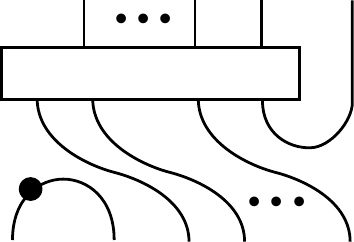}  ,\  \ \ \text{ or } \ \ \ \wwpic{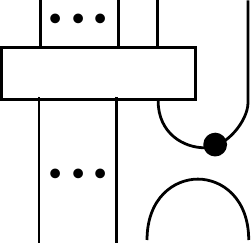}  + \wwpic{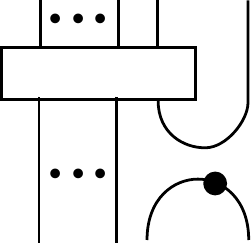} 
\]\vskip7pt
between adjacent terms.  We remind the reader that $\pic{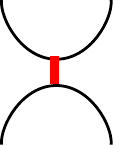}:q\pic{turnback}\rightarrow \pic{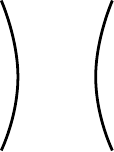}$ denotes the map corresponding to the saddle cobordism, and $\pic{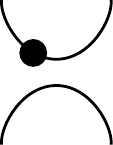}: q^2\pic{turnback}\rightarrow \pic{turnback}$ denotes an identity cobordism with a dot on one of the sheets.  Since the gluing of diagrams together in the plane is functorial, it is clear how to interpret the indicated pictures as chain maps.
\end{definition} 

\begin{proposition}\label{prop-CKhomotopyCx}
Write the sequence (\ref{eq-CKsequence}) as $E_\bullet = (\cdots \buildrel\a_{-2}\over \longrightarrow E_{-1} \buildrel\a_{-1}\over \longrightarrow E_0)$.  Then $\a_{i-1}\circ \a_i\simeq 0$, so that $(E_\bullet,\a)$ defines an object of $\Kom(\Kom(n)_{/h})$.
\end{proposition}
\begin{proof}
The proof splits up into cases.  First, note that when $n=2$ the Cooper-Krushkal recursion produces
\begin{equation}
\begin{diagram}\label{eq-P2}
P_2 & := &\Big( \cdots & \rTo^{\mpic{topdot}-\mpic{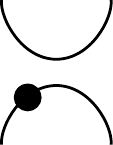}} & q^5 \pic{turnback} & \rTo^{\mpic{topdot}+\mpic{bottomdot}} & q^3 \pic{turnback} & \rTo^{\mpic{topdot}-\mpic{bottomdot}} & q \pic{turnback} & \rTo^{\mpic{isaddle}} & \underline{\pic{straightThrough}}\ \Big)
\end{diagram},
\end{equation}
which is already chain complex (not a homotopy complex).

In case $n\geq 3$, the proof that $\a_{i-1}\circ \a_i\simeq 0$ splits up into cases, depending on whether $\a_{i+1}$ and $\a_i$ are both saddle maps, or whether one of them is a sum or difference of dots.

\textbf{Case 1:}{ Both $\a_i$ and $\a_{i+1}$ are saddles.}  Up to flipping cobordisms upside down, $\a_{i+1}\circ \a_i$ is the composition of saddle cobordisms
\[
\a_{i+1}\circ \a_i =  \wwwpic{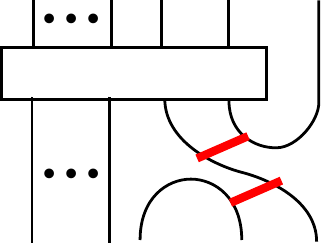}\hskip.3in\text{or}\hskip.3in\a_{i+1}\circ \a_i =  \wwwpic{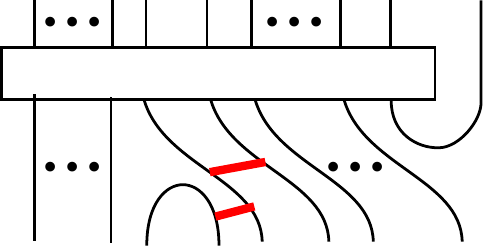}.
\]
By isotopy invariance of morphisms, the saddle maps can be performed in any order, and so $\a_{i+1}\circ \a_i$ factors through the chain complex
\[
 \wwwpic{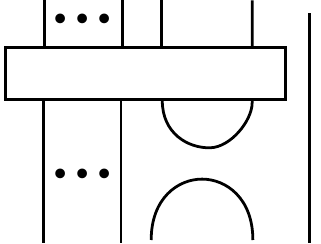}\hskip.3in \text{or}\hskip.3in \wwpic{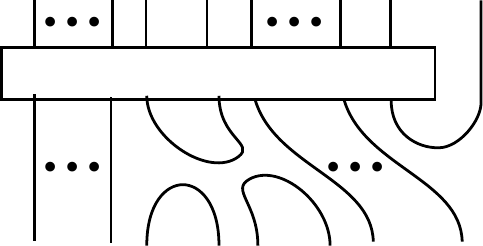},
\]
each of which is contractible since $P_{n-1}$ kills turnbacks.  Hence $\a_{i+1}\circ \a_i\simeq 0$ in this case.

\textbf{Case 2:}{ $\a_i$ or $\a_{i+1}$ is a difference of dots.}  Up to flipping cobordisms upside down, the composition $\a_{i+1}\circ \a_i$ is
\[
\a_{i+1}\circ \a_i =\wwpic{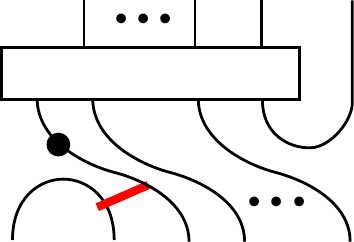} - \wwpic{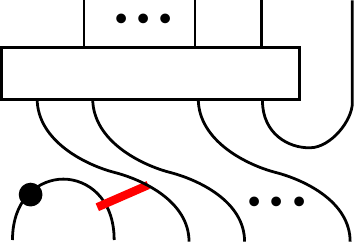},
\]
which is zero by isotopy of dots.

\textbf{Case 3:}{ $\a_i$ or $\a_{i+1}$ is a sum of dots.}  Up to flipping cobordisms upside down, the composition $\a_{i+1}\circ \a_i$ is
\begin{equation}\label{eq-sumOfDots}
\a_{i+1}\circ \a_i=\wpic{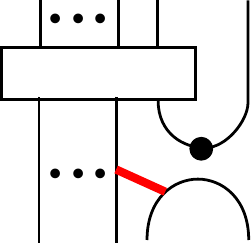} + \wpic{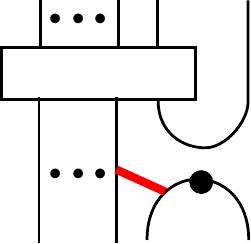} \simeq \wpic{FKpicOne_saddle_topDot} + \wpic{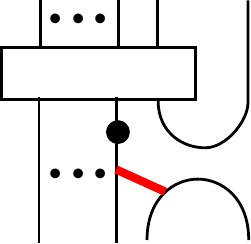},
\end{equation}
by isotopy of dots.  Now, this latter map is homotopic to zero by the following argument.  One one hand, from the neck cutting relation $\pic{cylinder} = \pic{1idemp} + \pic{xidemp}$ (Definition \ref{def-cobCat}) we have
\[
 \wwpic{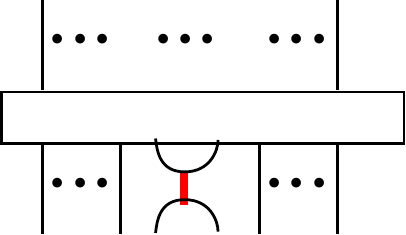}\circ \wwpic{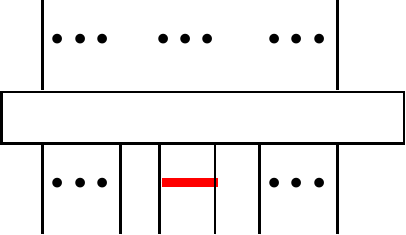} \simeq \wwpic{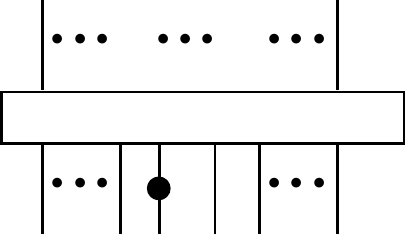} + \wwpic{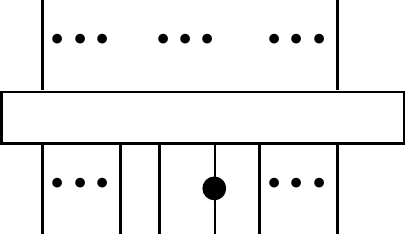} 
\]
On the other hand, this map is nulhomotopic, since it factors through the contractible complex 
\[
\wwwpic{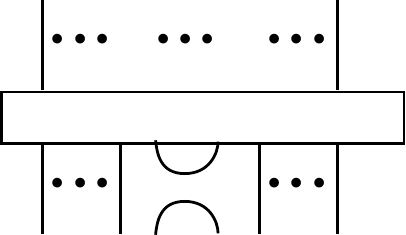}
\]
In particular (\ref{eq-sumOfDots}) is nulhomotopic.  This completes the proof.
\end{proof}
In the above proof, we obtained the following:
\begin{corollary}[Dot-hopping]\label{cor-dotHop}
If $C$ kills turnbacks, then the dotted identity maps $q^2C\rightarrow C$ satisfy the alternating property
\[
\wwpic{projector_bottomLeftDot} \simeq - \wwpic{projector_bottomRightDot} 
\]
and vertical reflections of these.  Here, we have put $\pic{projector}=C$.\qed
\end{corollary}

The Cooper-Krushkal sequence is not a bicomplex since the composition of successive maps is null-homotopic, rather than zero on the nose.  The notion of convolution replaces that of total complex in this situation (see also \S \ref{subsec-convolutions}):

\begin{definition}\label{def-convolution}
Let $E_i$ be chain complexes over an additive category and $\a_i:E_i\rightarrow E_{i+1}$ chain maps such that $\a_{i+1}\circ \a_i\simeq 0$ for all $i\in \Z$.  Any such sequence will be called a \emph{homotopy chain complex}, and will be denoted as
\begin{equation}\label{eq-homotopyCx}
E_\bullet = \cdots \buildrel\a_{i-1}\over \longrightarrow E_i \buildrel\a_{i}\over \longrightarrow E_{i+1} \buildrel \a_{i+1}\over \longrightarrow\cdots
\end{equation}
A convolution of a homotopy chain complex $E_\bullet$ is any chain complex which, as a graded object equals $\bigoplus_{i\in \Z} t^i E_i$ and whose differential $d$ satisfies the following conditions: if $d_{ij}\in\Homg^{1-i+j}(E_j,E_i)$ is the corresponding component of $d$, then
\begin{itemize}
\item $d_{ii}=(-1)^{i}d_{E_i}$.
\item $d_{i+1,i}=\a_i$.
\item $d_{ij}=0$ for $i<j$.
\end{itemize}
We will denote a convolution of (\ref{eq-homotopyCx}) by $M = \Tot(E_\bullet)$, or with  a parenthesized notation in which we write all of the degree shifts explicitly: 
\[
M = (\cdots \buildrel\a_{i-1}\over \longrightarrow t^i E_i \buildrel\a_{i}\over \longrightarrow t^{i+1}E_{i+1} \buildrel \a_{i+1}\over \longrightarrow\cdots )
\]
\end{definition}
One should think of the differential of a convolution of (\ref{eq-homotopyCx}) as a lower-triangular $\Z\times \Z$ matrix with $(-1)^i d_{E_i}$ on the diagonal and the $\a_i$ on the first subdiagonal.

\begin{theorem}[The Cooper-Krushkal recursion]\label{thm-CKrecursion}
If $P_{n-1}\in \Kom(n-1)$ is a Cooper-Krushkal projector, then there exists a convolution $P_n\in\Kom(n)$ of the Cooper-Krushkal sequence relative to $P_{n-1}$, and any such convolution is a strong Cooper-Krushkal projector (Definition \ref{def-strongCKproj}).
\end{theorem}
\begin{remark}
In [CK12, \S 7.4] the map between adjacent $\FKone$ terms is defined to be
\[
\wwpic{FKpicOne_bottomDot} - \wwpic{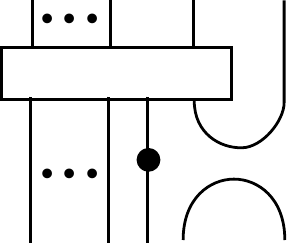}
\]
rather than our 
\[
\wwpic{FKpicOne_bottomDot} + \wwpic{FKpicOne_topDot}.
\]
But by dot-hopping (Corollary \ref{cor-dotHop}) the two maps are homotopic and, by Theorem \ref{thm-diffPerturbing}, any convolution of one sequence is isomorphic to a convolution of the other. 
\end{remark}

\begin{proof}[Proof of Theorem \ref{thm-CKrecursion}]
Up to reversing the homological grading conventions, it is shown in \cite{CK12a} that a convolution $C\in\Kom^-(n)$ of the sequence (\ref{eq-CKsequence}) relative to $P_{n-1}$ exists and, and any such convolution kills turnbacks from below.

Let $(-)^\ast:\Kom^-(n)\rightarrow \Kom^-(n)$ denote the covariant functor which reflects diagrams vertically.  Then $C^\ast\in\Kom^-(n)$ kills turnbacks from above.  Thus, $P_n':=C^\ast\otimes C$ kills turnbacks from above and below and satisfies the axioms for a strong Cooper-Krushkal projector (Definition \ref{def-strongCKproj}).  The condition on $1_n$ summands is clear from inspection of the sequence (\ref{eq-CKsequence}).  Thus, a strong Cooper-Krushkal projector $P_n'$ exists.  Given this, Corollary \ref{cor-turnbackKillingSymmetry} says that $C$ kills turnbacks, so that tensoring with the vertical reflection was unnecessary after all.  Thus, $C$ is a strong Cooper-Krushkal projector.
\end{proof}

\subsection{A well-defined Euler characteristic}
\label{subsec-eulerChar}
The usual notion of Grothendieck group is trivial for the categories $\Kom^-(n)$ of semi-infinite complexes.  However, all of the complexes in this paper can be assumed to lie in some subcategory for which there is a well-defined Euler characteristic, which takes values in $\TL_n':=\Z[q\inv]\llbracket q\rrbracket \otimes_{\Z[q,q\inv]} \TL_n^\Z$, where $\TL_n^\Z\subset \TL_n$ is the $\Z[q,q\inv]$-subalgebra generated by diagrams.  This is discussed in [CK12, \S 2.7.1].  The fact that the Euler characteristic descends to the homotopy category is discussed in \cite{Ros11}.

By clearing denominators and expanding rational functions into power series, we obtain an inclusion of rings $\Q(q)\hookrightarrow \Z[q\inv]\llbracket q\rrbracket$, hence $\Q(q)\otimes_{\Z[q,q\inv]} \TL_n^\Z$ can be regarded as a subalgebra of $\TL_n'$ and $\TL_n$.  We say that a complex $A\in\Kom^-(n)$ \emph{categorifies} $a\in\TL_n'\cap \TL_n$ if $A$ has a well-defined Euler characteristic which equals $a$.

\section{Categorified Jones-Wenzl quasi-idempotents}
\label{sec-quasiprojectors}
One can see from the recursion (\ref{eq-JWrecursion}) that certain products of elements $(1-q^{2k})p_n$ are linear combinations of Temperley-Lieb diagrams with \emph{polynomial coefficients}, as opposed to \emph{rational} coefficients.  In this section our goal is to show how to categorify the expressions $(1-q^{2k})p_n$ in terms of certain complexes $Q_k$ over $\bn{k}$, in such a way that corresponding tensor products of the $Q_k$ are equivalent to bounded complexes.  We then establish some basic properties of the $Q_k$.  One can recover the Cooper-Krushkal projector $P_n$ as a periodic complex constructed from the $Q_k$, from which originates an action of the ring $\Z[u_1,\ldots,u_n]$ on the Cooper-Krushkal projector $P_n$.

Let us describe another motivation for this polynomial action, coming from the apparent periodicity in the Cooper-Krushkal recursion, described in \S \ref{subsec-CKrecursion}.  Consider the following diagram, in which each row is the Cooper-Krushkal sequence $E_\bullet$ and we have omitted all degree shifts:
\begin{equation}\label{eq-CKperiodicityMap}
\begin{diagram}
\cdots &\rTo & \FKtwo &\rTo & \FKone &\rTo & \FKzero & \rTo & 0 &  & &&\\
\cdots & & \dTo^{\Id} & & \dTo^{\Id} & & \dTo^{\wmpic{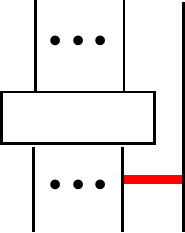} }   & & \dTo  &&&& \\
\cdots &\rTo &\FKtwo &\rTo & \FKone &\rTo & \FKone &\rTo^{\wmpic{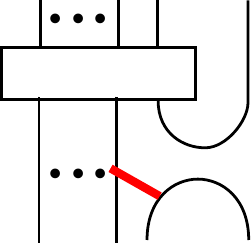} }  & \FKtwo & \rTo &\cdots &\rTo & \FKone &\rTo & \FKzero
\end{diagram}
\end{equation}
The right-most nontrivial square commutes up to homotopy, and every other square commutes on the nose.  That is to say, (\ref{eq-CKperiodicityMap}) defines a map of homotopy complexes $q^{2n}E[2-2n]_\bullet\rightarrow E_\bullet$, where $[1]$ denotes the upward grading shift, $E[1]_k = E_{k-1}$.  It is one of the goals of this paper to realize this homotopy chain map as an honest chain map.  That is to say, we wish to add some more maps pointing to the right and (non-strictly) down such that (1) the rows become the projector $P_n$ and (2) the non-horizontal components define a chain map $U_n: t^{2-2n}q^{2n}P_n\rightarrow P_n$.  Constructing $U_n$ directly is quite difficult because of the higher differentials required to make the Cooper-Krushkal sequence a chain complex.  Nonetheless, if $U_n$ were to exist, then $\Cone(U_n)$ would be homotopy equivalent to a complex
\[
Q_n = \Big(q^{2n}\FKzero \rightarrow q^{2n-1}\FKone \rightarrow\cdots  \rightarrow q^{n+1}\FKult  \rightarrow q^{n-1}\FKult\rightarrow \cdots\rightarrow q\FKone  \rightarrow \underline{\FKzero}\Big)
\]
whose differential is a sum of arrows pointing non-strictly to the right, and whose length 0 and 1 components are understood.  In this section we construct $U_n$ indirectly by first constructing such a complex $Q_n$.  We then recover $P_n$ as a periodic complex built from $Q_n$, from which we can define the map $U_n$ as desired.  We first consider the case $n=2$.

\subsection{The case $n=2$}
\label{subsec-Q2}
In case $n=2$, the homotopy chain map (\ref{eq-CKperiodicityMap}) is the following honest chain map $U_2:t^{-2}q^4P_2\rightarrow P_2$:
\begin{equation} \label{eq-u2}
U_2 : = \left(\begin{diagram}[height=2.5em]
 \cdots & \rTo^{\mpic{topdot}+\mpic{bottomdot}} &  q^7 \pic{turnback} & \rTo^{\mpic{topdot}-\mpic{bottomdot}}  & q^5 \pic{turnback} & \rTo^{\mpic{isaddle}} & q^4{\pic{straightThrough}} & \rTo & 0 & & \\
\cdots & & \dTo^\Id & & \dTo^\Id & & \dTo^{\mpic{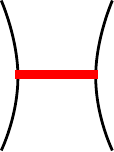}} & & \dTo & & \\
 \cdots & \rTo^{\mpic{topdot}+\mpic{bottomdot}} & q^7 \pic{turnback} & \rTo^{\mpic{topdot}-\mpic{bottomdot}} & q^5 \pic{turnback} & \rTo^{\mpic{topdot}+\mpic{bottomdot}} & q^3 \pic{turnback} & \rTo^{\mpic{topdot}-\mpic{bottomdot}} & q \pic{turnback} & \rTo^{\mpic{isaddle}} & \underline{\pic{straightThrough}}
\end{diagram}\ \right)
\end{equation}
By contracting the identity maps (Gaussian elimination, Proposition \ref{prop-gauss}) we see that $\Cone(U_2)$ deformation retracts onto the much simpler chain complex
\begin{equation}\label{eq-Q2}
\begin{diagram}
Q_2 &:= & q^4\pic{straightThrough} &\rTo^{\mpic{hsaddle}} & q^3\pic{turnback}& \rTo^{\mpic{topdot}-\mpic{bottomdot}} & q \pic{turnback} & \rTo^{\mpic{isaddle}} & \underline{\pic{straightThrough}}
\end{diagram}
\end{equation}
We can now recover $P_2$ as a periodic chain complex built out of copies of $Q_2$.  Indeed, let $P_2'$ denote the following semi-infinite chain complex:
\begin{equation}\label{eq-periodicP2}
\begin{diagram}[height=2.5em]
&&&& && q^4\pic{straightThrough} &\rTo^{\mpic{hsaddle}} & q^3\pic{turnback} & \rTo^{\mpic{topdot}-\mpic{bottomdot}} & q\pic{turnback} &\rTo^{\mpic{isaddle}} & \underline{\pic{straightThrough}}\\
 &&&&  &&&\rdTo^{-\Id}& &&&& \\
& & q^8\pic{straightThrough} &\rTo^{\mpic{hsaddle}} & q^7\pic{turnback} &\rTo^{\mpic{topdot}-\mpic{bottomdot}}& q^5\pic{turnback} &\rTo^{\mpic{isaddle}}& q^4\pic{straightThrough} &&&&\\
 &&&\rdTo^{-\Id}& &&&& &&&& \\
\dots & \rTo^{\mpic{topdot}-\mpic{bottomdot}} & q^9\pic{turnback} &\rTo^{\mpic{isaddle}} & q^8\pic{straightThrough} &&&& &&&&   
\end{diagram}
\end{equation}
One interprets the above diagram as a chain complex whose chain groups are given by summing along columns, the components of whose differential are indicated by the labelled arrows.  Contracting the identity maps, we see that in fact $P_2'$ deformation retracts onto $P_2$.  The signs above ensure that under the equivalence $P_2'\simeq P_2$, the obvious map $t^{-2}q^4P_2'\rightarrow P_2'$ (inclusion of a subcomplex) corresponds to the map (\ref{eq-u2}), and not its negative.

\subsection{The symmetric Cooper-Krushkal sequence}
\label{subsec-symmetricSeq}
It is remarkable that the above description of $P_2$ generalizes to all of the projectors $P_n$.

\begin{definition}\label{def-symmetricSeq}
Let $P_{n-1}\in\Kom(n-1)$ be a Cooper-Krushkal projector and define the \emph{symmetric Cooper-Krushkal sequence} relative to $P_{n-1}$ to be the following sequence $E_\bullet$ of chain complexes in $\Kom(n)$ and chain maps:
\vskip8pt
\begin{equation}\label{eq-symmetricSeq}
\begin{diagram}[height=3em]
q^{n-1}\FKult  & \rTo & q^{n-2} \FKpenult & \rTo &  \dots & \rTo & q^2\FKtwo & \rTo & q \FKone & \rTo &    \underline{\FKzero}\\
 \uTo   &&&&&&\\
q^{n+1} \FKult & \lTo & q^{n+2} \FKpenult  & \lTo &  \dots & \lTo & q^{2n-2}\FKtwo & \lTo &  q^{2n-1}  \FKone & \lTo &  q^{2n}  \FKzero
\end{diagram}
 \end{equation}\vskip8pt
where the white box denotes $P_{n-1}$.  The maps in this sequence are given by
\begin{itemize}
\item $\wwpic{FKpicGeneric-}$ between two terms in the top row.\vskip7pt
\item $\wwpic{FKpicGeneric+}$ between two terms in the bottom row.\vskip7pt
\item $\wwpic{FKpicUlt_topDot} - \wwpic{FKpicUlt_bottomDot} $ between the two terms in the left column. 
\end{itemize}
\end{definition}
We index this sequence as $E_{1-2n}\rightarrow \cdots\rightarrow E_0$, so that $E_0$ corresponds to the underlined term.  The word \emph{symmetric} refers to the fact that $q^{2n-i}E_{-i}=q^{i}E_{1-2n+i}$ for all $0\leq i\leq n-1$.
\begin{proposition}\label{symSeqHomotopyProp}
The symmetric Cooper-Krushkal sequence (\ref{eq-symmetricSeq}) is a homotopy chain complex.
\end{proposition}
\begin{proof}
Similar to the proof of Proposition \ref{prop-CKhomotopyCx}.
\end{proof}

 \begin{definition}\label{def-relativeSymSeq}
 Suppose $K\in\Kom^-(n-1)$ is any chain complex which kills turnbacks.  Define the \emph{symmetric Cooper-Krushkal sequence relative to $K$} to be the sequence of chain complexes $E_i\in\Kom^-(n)$ and chain maps defined precisely as in Definition \ref{def-symmetricSeq}, with $P_{n-1}$ replaced everywhere by $K$.
 \end{definition}
 The proof that the symmetric Cooper-Krushkal sequence is a homotopy chain complex uses only the turnback killing property, hence it applies to $K$ as well.

\begin{definition}\label{def-symProj}
Suppose $K\in\Kom(n-1)$ kills turnbacks.  Call a complex $Q_n\in\Kom(n)$ a \emph{symmetric projector} relative to $K$ if it is a convolution of the sequence (\ref{eq-symmetricSeq}) relative to $K$.  Call $Q_n$ simply a \emph{symmetric projector} if it is a symmetric projector relative to some $P_{n-1}$.  By convention we also call $Q_1:=\Cone(b)\in\Kom(1)$ a symmetric projector, where $b=\pic{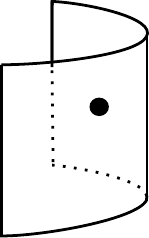}:q^21_1\rightarrow 1_1$.
\end{definition}

\subsection{Exressing $P_n$ in terms of the $Q_n$}
\label{subsec-standardCx}
We now show how to obtain $P_n$ from $Q_n$.  Then, assuming the existence of $Q_2,\ldots,Q_n$, we construct a family of Cooper-Krushkal projectors $P_1,\ldots,P_n$ which is useful for calculations.

In the following definition, assume $n\geq 2$ and let $E_\bullet = E_{1-2n}\rightarrow \cdots \rightarrow E_0$ be the symmetric Cooper-Krushkal sequence (Definition \ref{def-symmetricSeq}) relative to $P_{n-1}\in\Kom(n-1)$.  We have $E_{1-2n} = q^{2n}P_{n-1}\sqcup 1_1$ and $E_0  = P_{n-1}\sqcup 1_1$.
\begin{definition}\label{connectingDiffDef}   If $Q_n = \Tot(E_\bullet)$ is a symmetric projector then we have chain maps $\eta_n:P_{n-1}\sqcup 1_1\rightarrow Q_n$ and $\e_n:Q_n\rightarrow t^{1-2n}q^{2n}P_{n-1}\sqcup 1_1$ given by the inclusion of $E_0$, respectively the projection onto $t^{1-2n}E_{1-2n}$.  Put $\partial_n:= - \eta_n\circ \e_n$.
\end{definition}
In the sequel we will often encounter expressions $A\otimes C$, where $A$ is a bigraded abelian group and $C\in\Kom(n)$ is a chain complex.  We interpret such expressions as an appropriate direct sums of shifted copies of $C$.  For example, if $x$ is an indeterminate of bidegree $(a,b)$ then
\[
\Z[x]\otimes C:= \bigoplus_{k\geq 0}(t^{ak}q^{bk})A
\]
whenever this direct sum exists.  For $f\in\Homg^{i,j}(C,D)$, it is clear how to interpret $x^k\otimes f$ as an element of $\Homg^{ak+i,bk+j}(\Z[x]\otimes C,\Z[x]\otimes D)$.  We will simplify such complexes using Theorem \ref{thm-ringSDR}.

\begin{theorem}\label{thm-PfromQ}  Let $u_n$ denote a formal indeterminate of bidegree $(2-2n,2n)$.  If $Q_n\in\Kom(n)$ is a symmetric projector $(n\geq 2)$, then the chain complex $\Z[u_n]\otimes Q_n$ with differential $1\otimes d_{Q_n}+u_n\otimes \partial_n$ is Cooper-Krushkal projector.
\end{theorem}
\begin{remark}\label{remark-u1}
In case $n=1$ we run into a technical difficulty, which is that the direct sum $\Z[u_1]\otimes Q_1 = \bigoplus_{k\geq 0}q^{2k}Q_1$ is not finite in each homological degree, hence does not exist in $\Kom(1)$.  If we want to treat the variable $u_1$ similarly to $u_2,\ldots,u_n$, we could adjoin countable direct sums to $\bn{n}$, obtaining a category $\bn{n}^\oplus$.  Note that if $C\in\Kom^-(n)$ is any complex, then $\Z[u_1,\ldots,u_n]\otimes C$ exists in $\Kom(n)^\oplus:=\Kom(\bn{n}^\oplus)$.
\end{remark}

\begin{proof}[Proof of Theorem \ref{thm-PfromQ}]
Suppose $Q_n$ is a convolution of (\ref{eq-symmetricSeq}) relative to $P_{n-1}$, and let $C:=\Z[u_n]\otimes Q_n$ with differential $1\otimes d_{Q_n}+u_n\otimes \partial_n$.  Then $C$ is the total complex of a bicomplex as in:
\begin{equation}\label{eq-periodicPn}
\begin{diagram}[height=3em,width=1.5em]
&& &&&&&&   \Big( \FKzero & \rTo &\FKone & \rTo & \cdots & \rTo & \FKone & \rTo & \FKzero \Big) \\
&&&& &&&&&\rdTo^{-\Id}&   &&&&&&  \\
&&\Big( \FKzero & \rTo &\FKone & \rTo & \cdots &\rTo & \FKone & \rTo & \FKzero  \Big)&&&&&&  \\
&&&\rdTo^{-\Id}&  &&&&&& &&&&&& \\
\cdots &\rTo & \FKone &\rTo & \FKzero  \Big) &&&&&& &&&&&&
\end{diagram}
\end{equation}
  Contracting the isomorphisms in this expression using Gaussian elimination (Proposition \ref{prop-gauss}) produces a homotopy equivalent complex which is a Cooper-Krushkal projector by Theorem \ref{thm-CKrecursion}.  Thus $C$ is a Cooper-Krushkal projector.
\end{proof}

\begin{remark}
Here, the unit $\iota:1_n\rightarrow C$ is the composition $1_n\rightarrow P_{n-1}\sqcup 1_1\buildrel\phi\over \rightarrow C$, where $\phi = 1 \otimes \eta:P_{n-1}\sqcup 1_1\rightarrow \Z[u_n]\otimes Q_n$ is the inclusion of the right-most summand of (\ref{eq-periodicPn}).
\end{remark}

\begin{remark}\label{remark-sumIsProduct}
If $K\in\Kom^-(n)$, then $\Z[u_2,u_3,\ldots,u_n]\otimes K$ satisfies hypotheses of Theorem \ref{thm-ringSDR}.  Indeed, for $k>1$ the indeterminate $u_k$ has strictly negative homological degree.  If $K$ is bounded from above in homological degree, then $\bigoplus_{i\geq 0}u_k^i \otimes K$ is a finite direct sum in each homological degree.   Such direct sums are isomorphic to direct products in the category $\Kom^-(n)$, in which morphisms are \emph{degree preserving} chain maps.  A different argument takes care of the variable $u_1$ as well.%Now consider the case $k=1$ and $K=a\in\bn{n}$.  For a fixed object $b\in\bn{n}$, there exist nontrivial degree zero maps $b\leftrightarrow q^{2k}a$ for only finitely many $k$, hence the direct sum $\bigoplus_{i\geq 0}q^{2i}a$ is isomorpic to a direct product in $\bn{n}^\oplus$.  A similar argument shows that for each $K\in \Kom(n)$, the direct $\bigoplus_{k\geq 0}q^{2k}K\in \Kom(n)^\oplus$ is isomorphic to a direct product $\prod_{k\geq 0}q^{2k}K$.
\end{remark}

If we can construct $Q_n$, we will have succeeded in constructing the map (\ref{eq-CKperiodicityMap}):
\begin{corollary}\label{cor-QfromP}
If $Q_n\in\Kom(n)$ is a symmetric projector, then there is a Cooper-Krushkal projector $P_n$ and a chain map $U_n:t^{2-2n}q^{2n}P_n\rightarrow P_n$ such that $\Cone(U_n)\simeq Q_n$.\qed
\end{corollary}

\begin{corollary}\label{cor-QkillTurnbacks}
Any symmetric projector $Q_n$ kills turnbacks.
\end{corollary}
\begin{proof}
If $Q_n$ exists then it is equivalent to a mapping cone $\Cone(U_n)=(t^{1-2n}q^{2n}P_n\buildrel U_n\over\longrightarrow P_n)$.   This complex kills turnbacks since $P_n$ does.
\end{proof}

Now we construct a nice family of Cooper-Krushkal projectors $P_1,\ldots,P_n$, assuming the existence of $Q_2,\ldots,Q_n$.  First, a lemma:

\begin{lemma}\label{lemma-symProjReplacement}
Let $C_{n-1}\in \Kom(n-1)$ be a strong Cooper-Krushkal projector, and suppose that we are given a symmetric projector $Q_n$ relative to $C_{n-1}$.  If $K\in\Kom(n-1)$ is any complex which kills turnbacks, then there is a symmetric projector $Q_n'$ relative to $K$ and a deformation retract $(K\sqcup 1_1)\otimes Q_n\rightarrow Q_n'$.
\end{lemma}
\begin{proof}
By definition, the symmetric projector $Q_n$ is a convolution of the form:
\[
Q_n = \bigg(\ \FKzero \rightarrow \FKone \rightarrow \cdots \rightarrow \FKult \rightarrow \FKult \rightarrow \cdots \rightarrow \FKone\rightarrow \FKzero\ \bigg)
\]
where $\pic{projector}=C_{n-1}$.  Let $\iota_C:1_{n-1}\rightarrow C_{n-1}$ denote the unit of $C_{n-1}$.  Let us denote $K$ graphically by $K=\pic{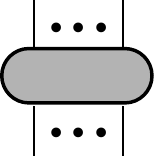}$, so that
 \begin{equation}\label{eq-symprojRoundSquare}
(K\sqcup 1_1)\otimes Q_n = \bigg(\ \wpic{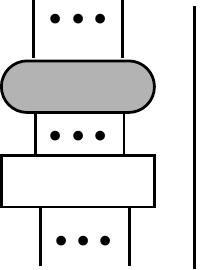}\rightarrow \wpic{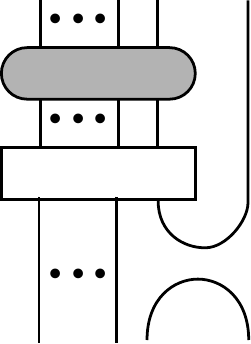}\rightarrow\cdots \rightarrow \wpic{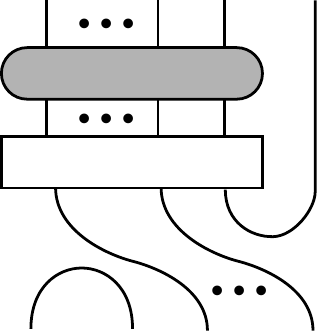}\rightarrow\wpic{FKpicUlt_round-square}\rightarrow\cdots \rightarrow \wpic{FKpicOne_round-square}\rightarrow\wpic{FKpicZero_round-square}\ \bigg)
\end{equation}
By projector absorbing (Proposition \ref{prop-projAbsorb}) the round box absorbs the white box via a deformation retract $\pic{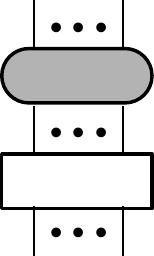}\rightarrow \pic{roundProjector}$.  Applying this deformation retract to each term of the complex (\ref{eq-symprojRoundSquare}),  (using Theorem \ref{convSdrthm}) gives a deformation retract
\[
 (C_{n-1}\sqcup 1_1)\otimes Q_n \simeq \bigg(\ \wpic{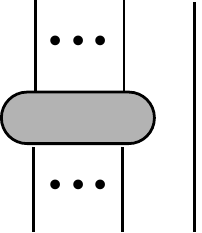}\rightarrow \wpic{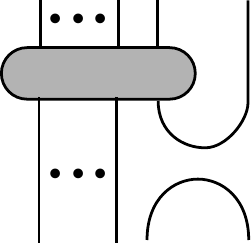}\rightarrow\cdots\rightarrow  \wpic{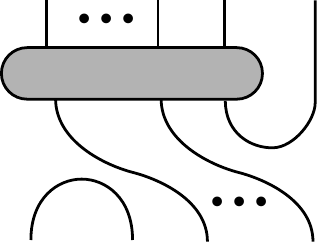}\rightarrow\wpic{FKpicUlt_round}\rightarrow\cdots \rightarrow  \wpic{FKpicOne_round}\rightarrow\wpic{FKpicZero_round}\ \bigg)
\]
This latter complex is a symmetric projector $Q_n'$ relative to $K$, as desired.
\end{proof}

\begin{construction}\label{const-standardCx}
Fix an integer $n\geq 2$ and assume that we are given symmetric projectors $Q_2,Q_3,\ldots,Q_n$.  Assume further that $Q_k$ is a symmetric projector relative to a strong Cooper-Krushkal projector $C_k$.  Let $P_k'=\Z[u_k]\otimes Q_k$ with differential $1\otimes d_{Q_n}+u_n\otimes \partial_n$ from Theorem \ref{thm-PfromQ}.  Now construct $P_1,\ldots,P_n$ by the following algorithm:
\begin{enumerate}
\item Put $P_1=1_1$.
\item Assume that $P_{n-1}$ has been constructed, and choose the data $(\pi,\sigma,h)$ of a deformation retract $P_{n-1}\otimes C_{n-1}\rightarrow P_{n-1}$.
\item Let $Q_n'$ denote the symmetric projector relative to $P_{n-1}$ constructed in Lemma \ref{lemma-symProjReplacement}.  That is, $Q_n'$ is the target of a deformation retract $(P_{n-1}\sqcup 1_1)\otimes Q_n\rightarrow Q_n'$, constructed from our chosen data $(\pi,\sigma,h)$.
\item Note that
\[
(P_{n-1}\sqcup 1_1)\otimes Q_n = \Z[u_n]\otimes ((P_{n-1}\sqcup 1_1)\otimes Q_n)
\]
with differential $1\otimes d + u_n\otimes \partial$.  Apply the deformation retract $(P_{n-1}\sqcup 1_1)\otimes Q_n\rightarrow Q_n'$ to each term using Theorem \ref{thm-ringSDR}, which applies by Remark \ref{remark-sumIsProduct}.  Define $P_n$ to be the target of this deformation retract.
\end{enumerate}
\end{construction}

\begin{theorem}\label{thm-standardCx}
Let $Q_2,\ldots,Q_n$ and $P_1,\ldots,P_n$ be as in Construction \ref{const-standardCx}.  Then there is a bounded complex $K_n\simeq (Q_2\sqcup 1_{n-2})\otimes (Q_3\sqcup 1_{n-3})\otimes \cdots \otimes Q_n$ such that
\begin{equation}\label{eq-specialPn}
P_n = \Z[u_2,u_3,\ldots,u_n]\otimes K_n \ \ \text{ with differential } \ 1\otimes K_n +\sum_{f>1}f\otimes \partial_f
\end{equation}
where the sum is over nonconstant monomials $f\in\Z[u_2,\ldots,u_n]$.
\end{theorem}
\begin{proof}
Induction on $n\geq 1$.  The claim is vacuous in the base case $n=1$.  Assume by induction that statement (1) holds for $P_{n-1}$.  %By what we have said, we may also assume that the homotopy $h_a\in\Endg(P_{n-1}\otimes a)$ commutes with the $\Z[u_2,\ldots,u_{n-1}]$-action, for each $a\in\bn{n-1}$ with $\tau(a)<n-1$.

From Construction \ref{const-standardCx} we assume that $Q_n$ is a symmetric projector relative to a strong Cooper-Krushkal projector $C_{n-1}$.  Pick a deformation retract $K_{n-1}\otimes C_{n-1}\rightarrow K_{n-1}$.  Applying this deformation retract to each term of
\[
P_{n-1}\otimes C_{n-1} = \Z[u_2,\ldots,u_{n-2}]\otimes (K_{n-1}\otimes C_{n-1})  \ \ \text{ with differential } 1\otimes d + \sum_{f>1}f\otimes \partial_f
\]
gives the data of a deformation retract $P_{n-1}\otimes C_{n-1}\rightarrow P_{n-1}$.

Let $P_n'=\Z[u_n]\otimes Q_n$ be as in Theorem \ref{thm-PfromQ}, so that
\[
(P_{n-1}\sqcup 1_1)\otimes P_n' = \Z[u_2,\ldots,u_n] \otimes ((K_{n-1}\sqcup 1_1)\otimes Q_n) \ \ \text{ with differential } 1\otimes d + \sum_{f>1}f\otimes \partial_f
\]
Now, $P_n$ is obtained from $(P_{n-1}\sqcup 1_1)\otimes P_n'$ by simply applying the deformation retract $P_{n-1}\otimes C_{n-1}\rightarrow P_{n-1}$ wherever possible.  By construction, this deformation retract is the result of simply applying the deformation retract $K_{n-1}\otimes C_{n-1}\rightarrow K_{n-1}$ wherever possible.  By Lemma \ref{lemma-symProjReplacement}, applying $K_{n-1}\otimes C_{n-1}\rightarrow K_{n-1}$ to $(K_{n-1}\sqcup 1_1)\otimes Q_n$ produces a symmetric projector $K_n$ relative to $K_{n-1}$.  Therefore, $P_n$ is simply the target of a deformation retract
\[
(P_{n-1}\sqcup 1_1)\otimes P_n'
 \rightarrow \Z[u_2,\ldots,u_n] \otimes K_n \ \ \text{ with differential } \  1\otimes d + \sum_{f>1}f\otimes \partial_f
\]
as in the statement.  This completes the proof.
\end{proof}

\begin{definition}\label{def-polyDef}
Suppose $P_n$ is as in Construction \ref{const-standardCx}.  Let $\Z[u_1,\ldots,u_n]\rightarrow \Endg(P_n)$ be the map of differential bigraded algebras in which $u_k$ $(k\geq 2)$ acts in the obvious way on (\ref{eq-specialPn}), and in which $u_1$ acts as the dotted identity $u_1 = \pic{projector_SWdot}$.  Let us denote the image of $u_k$ by $U_k^{(n)}$.  If $P_n'$ is any other Cooper-Krushkal projector, then conjugating by a canonical equivalence (Definition \ref{def-canEquiv}) $P_n\rightarrow P_n'$ gives a chain map $t^{2-2k}q^{2k}P_n'\rightarrow P_n'$ (well-defined up to homotopy, and may depend on the choice of $Q_2,\ldots,Q_n$), which we also denote by $U_k^{(n)}$, by abuse.  
\end{definition}
In light of Corollary \ref{cor-extGroups}, if we did not care about sign ambinguity, we could simply define $U_k^{(n)}$ to be a chain map which represents a generator of $\Ext^{2-2k,2k}(P_n,P_n)\cong \Z$.  We will occasionally abuse this notation and write $u_k = U_k^{(n)}$.

\begin{corollary}\label{cor-equiCor}
Let $P_n$ be as in Construction \ref{const-standardCx}. For each $a\in\bn{n}$ with $\tau(a)<n$, the homotopy which contracts $P_n\otimes a$ and $a\otimes P_n$ can be chosen to commute with the $\Z[u_2,\ldots,u_n]$-action.
\end{corollary}
\begin{proof}
Note that $K_n$ kills turnbacks since $Q_n$ does (Corollary \ref{cor-QkillTurnbacks}).  We have $P_{n}\otimes a = \Z[u_2,\ldots,u_{n}]\otimes(K_{n} \otimes a)$ with differential $1\otimes d +\sum_{f>1}f\otimes \partial_f'$.  If $a\in\bn{n}$ has through degree $\tau(a)<n$, then homotopy $h_a$ which contracts this complex can be chosen to commute with the $\Z[u_2,\ldots,u_{n}]$-action by Theorem \ref{thm-ringSDR}.
\end{proof}

\begin{remark}\label{remark-equivariance}
The fact that the homotopies which contract $P_k\otimes a\simeq 0$ and $a\otimes P_k\simeq 0$ can be chosen to be $\Z[u_2,\ldots,u_k]$-equivariant implies, together with Remark \ref{remark-sdrPreservesStruct}, that essentially any equivalence which uses only the turnback killing properties of $P_n$ can be chosen to be $\Z[u_2,\ldots,u_n]$-equivariant.
\end{remark}

\subsection{Simplification of $\Endg(P_n)$, connection to the GOR conjecture}
\label{subsec-endSimplification}
We will use the above construction of $P_n$ to simplify $\Endg(P_n)$ up to equivalence.  We will conclude, in particular, that any chain map $f\in\Endg^{i,j}(P_n)$ with $i+j=3$ is null-homotopic.  In the next section, we use this fact in an inductive argument for the existence of $Q_{n+1}$.

\begin{definition}\label{def-hom1P}
Assume $Q_2,\ldots,Q_n$ and $P_1,\ldots,P_n$ are as in Construction \ref{const-standardCx}.  For $k\in\{1,\ldots,n\}$  denote the complex $\Homg(1_k,P_k)$ simply by $E_k$.  We regard $E_k$ as a differential bigraded $\Z[u_1,\ldots,u_k]$-module, where $u_k$ acts as post-composition with $U^{(k)}_k$ from Definition \ref{def-polyDef}.
\end{definition}

\begin{definition}
Throughout, let $\xi_k$ denote a formal (odd) indeterminate of bidegree $(1-2k,2+2k)$.  Notation such as $\Z[u_k,\xi_k]$ will always denote the super-polynomial ring $\Z[u_k]\otimes_\Z \Lambda[\xi_k]$.
\end{definition}

\begin{proposition}\label{prop-partialTraceComp}
There is a deformation retract
\[
E_n \simeq \Z[u_n,\xi_n]\otimes_\Z E_{n-1}
\]
where the complex on the right has $\Z[u_n]$-equivariant differential determined by
\begin{enumerate}
\item $d(1\otimes \a) = 1\otimes d(\a)$
\item $d(\xi \otimes \a) = 2u_n\otimes (u_1\circ \a) + 1\otimes \d_n(\a) - \xi_n\otimes d(\a)$
\end{enumerate}
for all $\a\in E_{n-1}$, where $\d_n\in\Endg(E_{n-1})$ is some chain map of bidegree $(2-2n,2+2n)$.  The data of this deformation retract are $\Z[u_1,\ldots,u_n]$-equivariant.
\end{proposition}
\begin{proof}
From Theorem \ref{thm-partialTraceHom}, we have
\[
E_n =\Homg(1_n,P_n) \cong \Homg(1_{n-1},qT(P_{n-1})
\]
where $T:\Kom(n)\rightarrow \Kom(n-1)$ is the partial trace functor (Definition \ref{def-planarComp}).  Naturality implies that this isomorphism is $\Z[u_1,\ldots,u_n]$ equivariant.  Let us simplify $qT(P_n)$.

Applying $qT(-)$ to the to the periodic complex ({\ref{eq-periodicPn}}) and contracting terms of the form $\wwwpic{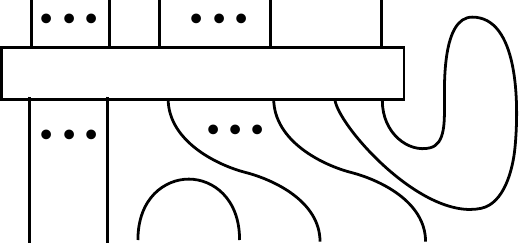}$ (using Theorem \hyperref[convSdrthm]{\ref{convSdrthm}}), we see that $qT(P_n)$ deformation retracts onto:
\begin{equation}\label{eq-partialTraceSimp}
\begin{diagram}[height=3em,width=1.5em]
&&&&  &&  \Big( \wpic{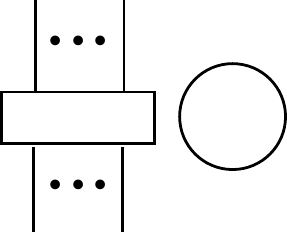} & \rTo &\wpic{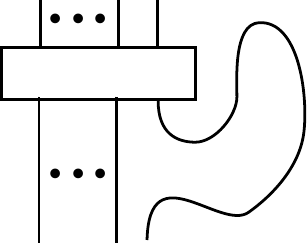} & \rTo & \wpic{FKpic_1_partialTrace} & \rTo & \wpic{FKpic_0_partialTrace} \Big) \\
&&&& &&  &\rdTo^{-\Id}&   &&&& \\
 && \Big( \wpic{FKpic_0_partialTrace} & \rTo &\wpic{FKpic_1_partialTrace} & \rTo & \wpic{FKpic_1_partialTrace} & \rTo & \wpic{FKpic_0_partialTrace} \Big) &&&& \\
 &&  &\rdTo^{-\Id}&   &&&& &&&& \\
\cdots  & \rTo & \wpic{FKpic_1_partialTrace} & \rTo & \wpic{FKpic_0_partialTrace} \Big) &&&& &&&&
\end{diagram}
\end{equation}
We can simplify this complex row-by-row (Lemma \ref{lemma-delooping}) obtaining a deformation retract of this complex onto (after inserting the correct degree shifts):
\begin{equation}\label{eq-partialTraceSimp2}
\Tot\left(
\begin{diagram}[height=2.5em]
 y P_{n-1} & \rTo^{\d_n} &  P_{n-1}\\
& \rdTo^{2\pic{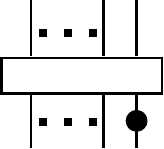}} & \\
 xy P_{n-1} & \rTo^{\d_n} & x P_{n-1} \\
 & \rdTo^{2\pic{projector_SEdot}} &  \\
\vdots & \vdots& \vdots
\end{diagram}
\right)
\end{equation}
where $y = t^{1-2n}q^{2+2n}$ and $x = t^{2-2n}q^{2n}$.  The data of the above deformation retracts commute with the $\Z[u_1,\ldots,u_n]$-action by Remark \ref{remark-equivariance} and equivariance of retracts coming from Theorem \ref{thm-ringSDR}.

The complex (\ref{eq-partialTraceSimp2}) is a mapping cone on a chain map $\Delta = 2u_n\otimes \pic{projector_SEdot} + 1\otimes \d_n \in \Endg(\Z[u_n]\otimes P_{n-1})$.  By dot-hopping (Corollary \ref{cor-dotHop}), we have $\pic{projector_SEdot}\simeq \pm \pic{projector_SWdot}$, and so we can replace $\Delta$ by $\pm 2u_n\otimes u_1+1\otimes \d_n$.  The dot-hopping uses only the turnback killing property of $P_{n-1}$ hence can be done $\Z[u_1,\ldots,u_n]$-equivariantly.  Replace $\d_n$ by $\pm \d_n$ if necessary.  Now, apply the functor $\Homg(1_{n-1},-)$ to this deformation retract to obtain a deformation retract as in the statement.
\end{proof} 

The following lemma was used in the proof of Proposition \hyperref[prop-partialTraceComp]{\ref{prop-partialTraceComp}}:

\begin{lemma}\label{lemma-delooping}
The chain complexes
\begin{equation}\label{eq-deloopRetract1}
C_1\ :=\ \begin{diagram} q^2 \ \pic{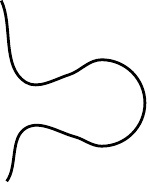} &\rTo^{\wmpic{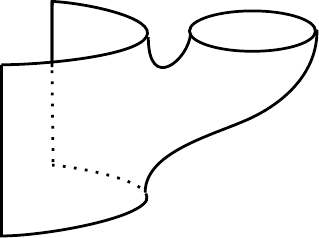} } & q \ \underline{\pic{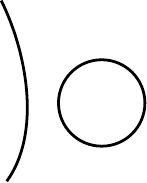}} \end{diagram}   \hskip.6in C_2\ := \ \begin{diagram} q\inv\  \underline{\pic{deloopObj1}}  &\rTo^{\wmpic{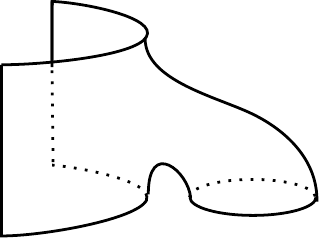} } & q^{-2}  \ \pic{deloopObj2}\end{diagram} 
\end{equation}
deformation retract onto $\pic{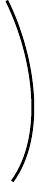}$, and applying these equivalences row-by-row to the following chain complex gives a deformation retract:
\begin{equation}\label{eq-deloopRetract2}
\begin{minipage}{1.8in}
\begin{tikzpicture}
\node (a) at (3,-2) {$ q \ \pic{deloopObj2}$};
\node(b) at (6,-2) {$\pic{deloopObj1}$};
\node (c) at (3,0) {$\pic{deloopObj1}$};
\node (d) at (6,0) {$q\inv\ \pic{deloopObj2}$};
\path[->,>=stealth',shorten >=1pt,auto,node distance=1.8cm,
  thick]
(a) edge node[label={[shift={(-.15,0.0)}]$\wmpic{unknotMerge}$}] {} (b)
(c) edge node[label={[shift={(-.15,-0.25)}]$-\wmpic{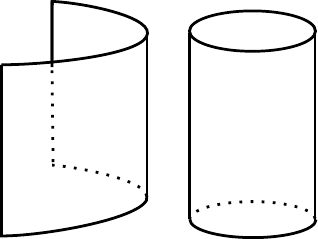}$}] {} (b)
(c) edge node[label={[shift={(0.0,0.0)}]$\wmpic{unknotSplit}$}] {} (d);
\end{tikzpicture}
\end{minipage}
\ \ \simeq\ \ 
\begin{minipage}{1.8in}
\begin{tikzpicture}
\node(b) at (6,-2) {$q\ \inv\pic{1}$};
\node (c) at (3,0) {$q \ \pic{1}$};
\path[->,>=stealth',shorten >=1pt,auto,node distance=1.8cm,
  thick]
(c) edge node[label={[shift={(-.15,-0.25)}]$2\: \wmpic{dottedIdSheet}$}] {} (b);
\end{tikzpicture}
\end{minipage}
\end{equation}
\end{lemma}
\begin{proof}
This is a straightforward computation in Bar-Natan's categories, which we leave to the reader.  This computation also appears in the proof of invariance of Khovanov homology under the Reidemeister 1 move.
\end{proof}

\begin{theorem}\label{thm-GorEnd}
Assume that symmetric projectors $Q_2,\ldots,Q_n$ exist, and let $P_1,P_2,\ldots,P_n$ be as in Construction \ref{const-standardCx}.  Then there is a deformation retract from $\Endg(P_n)$ onto a differential bigraded $\Z[u_1,\ldots,u_n]$-module $W_n=\Z[u_1,\ldots,u_n,\xi_2,\xi_3,\ldots,\xi_n]/(u_1^2)$ with differential satisfying:
\begin{enumerate}
\item $d(u_k)=0$ for each $k=1,2,\ldots,n$.
\item $d(\xi_k)\in 2u_1u_k+\Z[u_2,\ldots,u_{k-1}]$ for each $k=2,3,\ldots,n$.
\item the inclusion of bigraded abelian groups $W_{n-1}\hookrightarrow W_n$ commutes with differentials.
\end{enumerate}
The data of this deformation retract can be chosen to be $\Z[u_1,\ldots,u_n]$-equivariant.
\end{theorem}
\begin{proof}
First, note that $P_n$ is a complex of the form (\ref{eq-periodicPn}) by construction.  Contracting the isomorphisms, we see that $P_n$ deformation retracts onto a strong Cooper-Krushkal projector $C_n'$.  Applying the functor $\Homg(-,P_n)$ gives a deformation retract $\Endg(P_n)\rightarrow \Homg(C_n',P_n)$, the data of which commute with the action of $\Z[u_1,\ldots,u_n]$ on the second argument.  Now, Proposition \ref{prop-resProp} gives a deformation retract 
\[
\Homg(C_n',P_n)\rightarrow \Homg(1_n,P_n) =: E_n
\]
The data of this deformation retract commute with the $\Z[u_1,\ldots,u_n]$-action, by naturality of the isomorphism of Corollary \ref{cor-homRotation}, and Remark \ref{remark-equivariance}.  Thus, it suffices to construct a deformation retract $E_n\rightarrow W_n$ as in the statement.  This will be accomplished by induction, with inductive step provided by Proposition \ref{prop-partialTraceComp}.  In the base case: $E_1= \Homg(\emptyset, q\cdot \text{unknot})$ is the Khovanov homology of the unknot, shifted up in $q$-degree.  This is isomorphic to $\Z[u_1]/(u_1^2)$, which proves the base case.  Assume by induction that we have constructed a deformation retract $E_{n-1}\rightarrow W_{n-1}$.

From Proposition \ref{prop-partialTraceComp}, there is a $\Z[u_1,\ldots,u_n]$-equivariant deformation retract from $E_n$ onto a complex $\Z[u_n,\xi_n]\otimes E_{n-1}$ with differential of the form:
\[
\Big(\xi_n\Z[u_n]\otimes E_{n-1}\Big)\oplus \Big(\Z[u_n]\otimes E_{n-1}\Big) \text{ with differential } \matrix{1\otimes d_{E_{n-1}} & 0 \\ 2u_n\otimes u_1 + 1\otimes \d_n & 1\otimes d_{E_{n-1}}}
\]
where we are using the Koszul sign rule to evaluate tensor products of morphisms on tensor products of graded spaces.  By the induction hypothesis, we have the $\Z[u_1,\ldots,u_{n-1}]$-equivariant data $(\pi,\sigma,h)$ of a deformation retract $E_{n-1}\rightarrow W_{n-1}$.  Applying this to each $\Z[u_n]\otimes E_{n-1}$ summand above (Theorem \ref{convSdrthm} in the special case of 2-term convolutions) gives a deformation retract from $\Z[u_n,\xi_n]\otimes E_{n-1}$ onto
\[
\Big(\xi\Z[u_n]\otimes W_{n-1}\Big)\oplus \Big(\Z[u_n]\otimes W_{n-1}\Big) \text{ with differential } \matrix{1\otimes d_{W_{n-1}} & 0 \\ 2u_n\otimes u_1 + 1\otimes \bar\d_n & 1\otimes d_{W_{n-1}}}
\]
where $\bar \d_n = \pi\circ \d_n\circ \sigma$.  The data of this deformation retract can be chosen to commute with the $\Z[u_1,\ldots,u_n]$-action by Remark \ref{remark-sdrPreservesStruct}.  As a bigraded $\Z[u_1,\ldots,u_n]$-module, this latter object is isomorphic to $\Z[u_n,\xi_n]\otimes W_{n-1}\cong W_n$.  The differential on $W_n$ satisfies the properties (1), (2), (3) of the statement, by inspection.  For example: $d(\xi_n)=2u_nu_1 + \bar\d_n(1)$.  For degree reasons, $\bar\d(1)\in W_{n-1}\subset W_n$ must be some quadratic polynomial in $u_2,\ldots,u_{n-1}$.  This completes the proof.
\end{proof}

\begin{corollary}\label{cor-extGroups}
Let $P_n$ be a Cooper-Krushkal projector, and assume that symmetric projectors $Q_k\in \Kom(k)$ exist for $2\leq k\leq n$.  Then the group $\Ext^{i,j}(P_n,P_n)$ of chain maps $t^iq^jP_n\rightarrow P_n$ modulo chain homotopy satisfies
\begin{enumerate}
\item $\Ext^{k-i,i}(P_n,P_n)=0$ for all $i$, if $k<0$.
\item $\Ext^{0-i,i}(P_n,P_n)=\Z$ for $i=0$ and zero otherwise.
\item $\Ext^{1-i,i}(P_n,P_n)=0$ for all $i$.
\item $\Ext^{2-i,i}(P_n,P_n)=\Z$ for $i=2,4,\ldots,2n$ and zero otherwise.
\item $\Ext^{3-i,i}(P_n,P_n)=0$ for all $i$.
\end{enumerate}
For a generator of $\Ext^{2-2k}{2k}(P_n,P_n)\cong \Z$ we may pick $[U_k^{(n)}]$ as in Definition \ref{def-polyDef}.
\end{corollary}
\begin{proof}
From Theorem \ref{thm-GorEnd}, we see that there is some differential on $W_n=\Z[u_1,\ldots,u_n]/(u_1^2)\otimes \Lambda[\xi_2,\ldots, \xi_n]$ such that (1) $d(\xi_m) \neq 0$ $(1\leq m\leq n)$, and (2) the homology of $W_n$ is isomorphic to the homology of $\Endg(P_n)$ as bigraded $\Z[u_1,\ldots,u_n]$-modules.  The bigrading respects the algebra structure: $\deg(vv')=\deg(v)+\deg(v')$, but the differential does not necessarily respect the algebra structure.  Collapse the bigrading to the single grading $\deg_s = \deg_h + \deg_q$, so that $\deg_s(u_k)=2$ and $\deg_s(\xi_k) = 3$ for $k=1,\ldots,n$.

Clearly there are no elements $x\in W_n$ of degree $\deg_s(x)<0$ or $\deg_s(x) = 1$, and the only bihomogeneous elements with $\deg_s(x)=3$ are the multiples of $\xi_k$.  But none of the $\xi_k$ are cycles.  This proves statements (1), (3), and (5) of the Theorem. 

The only elements $x\in W_n$ with $\deg_s(x) = 0$ are multiplies of the identity.  If any of these were a boundary $d(h) = a\cdot 1$, then $\deg_s(h)=-1$ forces $h$ to be zero.  (2) follows.  Similarly, the only bihomogeneous elements $x\in W_n$ with $\deg_s(x)=2$ are multiples of some $u_k$, each of which is a cycle.  If any multiple of $u_m$ were a boundary, say $d(h) = a u_m$, then $\deg_s(h)=1$ forces $h=0$.  This shows that the $u_m$ generate homology groups isomorphic to $\Z$, which is  (4).%  Finally the $u_k$ are preserved up to homotopy under the canonical maps $\rho_m^n$, by statement (3) of Theorem \ref{thm-standardCx}.  This completes the proof.
\end{proof}

We conclude this section with a result which relates the above computations with a conjecture in \cite{GOR12} regarding Khovanov homology of torus knots.
\begin{corollary}\label{cor-gorCor}
Let $H_{p,q}$ denote the Khovanov homology of the $p,q$-torus link (oriented as a positive braid; see \S \ref{subsec-quasiHomology} for our conventions regarding Khovanov homology).  Then there is a grading shift $F_{p,q}$ and a direct system $F_{p,q}H_{p,q}\rightarrow F_{p,q+1}H_{p,q+1}$ with limit $H_{p,\infty}$, where $H_{p,\infty}$ is the homology of the chain complex $W_p$ from Theorem \ref{thm-GorEnd}.
\end{corollary}
\begin{proof}
Denote by $C_p$ the positive braid corresponding to the $n$-cycle $(1,2,\ldots,n)$.  Then $T_{p,q}$ is the braid closure of $C_p^q$.  It is known \cite{Roz10a} that the Khovanov complexes $\{\llbracket C_p^q\rrbracket\:|\: q\in\Z_{\geq 0}\}$ form a direct system up to homotopy and grading shifts, and that there is a limit given by the projector $P_p$.  Thus, the Khovanov homology of $C_p^q$ has stable limit (up to shift) as $q\to \infty$, given by the homology of $\Homg(\emptyset, T^p(P_p))$, where $T$ is the partial trace functor (Definition \ref{def-planarComp}).  By Theorem \ref{thm-partialTraceHom} and Proposition \ref{prop-resProp}, this latter homology is isomorphic to the homology of $\Endg(P_n)$, up to shift.  Finally, Theorem \ref{thm-GorEnd} says that the homology of $\Endg(P_n)$ is isomorphic to the homology of $\Z[u_1,\ldots,u_n,\xi_2,\ldots,\xi_n]/(u_1^2)$ with some $\Z[u_1,\ldots,u_n]$-equivariant differential.
\end{proof}
The conjecture in \cite{GOR12} is that $d(\xi_k) = \sum_{i+j=k+1}u_iu_j$, and that $d$ satisfies the graded Leibniz rule with respect to the obvious multiplication in $W_n:=\Z[u_1,\ldots,u_n,\xi_2,\ldots,\xi_n]/(u_1^2)$.  We do not yet have such a formula for $d(\xi_k)$, nor do we know that the differential can be assumed to respect the Leibniz rule with respect to the \emph{obvious} multiplication.  Nonetheless, the fact that $\Endg(P_n)$ deformation retracts (for appropriate choice of representative for $P_n$) onto $W_n$, hence $W_n$ inherits the structure of an $A_\infty$-algebra from $\Endg(P_n)$.  In particular, there is some multiplication on $W_n$ (associative up to homotopy) such that the Leibniz rule is satisfied.

\subsection{Existence and uniqueness of $Q_n$}
\label{subsec-QnExists}
The proof that $Q_n$ exists uses the following basic fact:

\begin{lemma}\label{lemma-obstThry}
Let $A,B,C$ be chain complexes over an additive category, and suppose we have linear maps $\a \in \Homg^1(A,B)$, $\b\in \Homg^1(B,C)$ which are homogeneous with homological degree 1.  Then a convolution $(A \buildrel \a\over\longrightarrow B\buildrel \b\over \longrightarrow C)$ exists if and only if $\a,\b$ are cycles and $\b\circ\a\in \Homg^2(A,C)$ is a boundary.  In particular, such a convolution exists if  $\a$ and $\b$ are cycles and $\Ext^2(A,C)\cong 0$. 
\end{lemma}
\begin{proof}
Recall that a convolution is simply a direct sum of complexes with lower triangular differential.  Observe
\[
\matrix{d_A &0 &0\\ \a & d_B & 0\\ -h &\b & d_C}
\]
is a differential if and only if (1) $d_B\circ\a +\a\circ d_A=0$, (2) $d_C\circ\b + \b\circ d_B=0$, and (3) $\b\circ \a = d_C\circ h+h\circ d_A$.  This is precisely the statement of the lemma.
\end{proof}

\begin{theorem}\label{thm-QnExistence} For each integer $n\geq 2$ there exists a symmetric projector $Q_n\in \Kom(n)$.
\end{theorem}
\begin{proof}
The proof that $Q_n$ exists proceeds by induction on $n\geq 2$.  The chain complex $Q_2$ is already defined in \S \ref{subsec-Q2}.  Assume by induction that $Q_m\in\Kom(m)$ exists for each $2\leq m\leq n-1$.   Write the sequence (\ref{eq-symmetricSeq}) as $E_{1-2n}\rightarrow \cdots \rightarrow E_{-1}\rightarrow E_0$.  By the Cooper-Krushkal recursion (Theorem \ref{thm-CKrecursion}) we can assume that there is a convolution
\[
M = \Big(t^{2-2n}E_{2-2n}\rightarrow \cdots \rightarrow t\inv E_{-1}\rightarrow E_0\Big)
\]
of all terms except for the left-most.  We seek a convolution of the form
\[
\Big(\:\underbrace{t^{1-2n}E_{1-2n}}_A \rightarrow \underbrace{t^{2-2n}E_{2-2n}}_B \rightarrow  \underbrace{t^{3-2n}E_{3-2n}\rightarrow \cdots \rightarrow t\inv E_{-1}\rightarrow E_0}_C\:\Big)
\]
We reassociate as indicated by the parentheses, so that $M$ is a convolution $M=(B\buildrel \b\over\longrightarrow C)$ for some degree $1$ cycle $\b\in\Homg^{1,0}(B,C)$.  The degree zero chain map $\a_{1-2n}:E_{1-2n}\rightarrow E_{2-2n}$ defines a degree 1 cycle $\a:\in\Homg^{1,0}(A,B)$.  By Lemma \ref{lemma-obstThry}, we must show that $\b\circ \a$ is a degree $(2,0)$ boundary.  In fact, we will show that the relevant homology group is zero.  Examining the symmetric Cooper-Krushkal sequence (\ref{eq-symmetricSeq}), we see that $A := t^{1-2n}q^{2n}\FKzero$, $B:=t^{2-2n}q^{2n-1}\FKone$, and 
$C:= \Big(t^{3-2n}q^{2n-2}\FKtwo\rightarrow \cdots \rightarrow t\inv q\FKone \rightarrow \FKzero\Big)$.  Abbreviate $\mathcal{F}(-,-):=\Homg(-,-)$ and $y = t^{1-2n}q^{2n}$, and compute:
\begin{eqnarray*}
\mathcal{F}(A,C)
& \buildrel(1)\over = & \mathcal{F}\Bigg(y\FKzero, \Big(t^{3-2n}q^{2n-2}\FKtwo\rightarrow \cdots \rightarrow t\inv q\FKone\rightarrow \FKzero\Big)\Bigg)\\
&\buildrel(2)\over\cong & y\inv\mathcal{F}\Bigg(\FKzero, \Big(t^{3-2n}q^{2n-2}\FKtwo\rightarrow \cdots \rightarrow t\inv q\FKone\rightarrow \FKzero\Big)\Bigg)\\
&\buildrel(3)\over\cong & qy\inv\mathcal{F}\Bigg(\pic{projector} \ , \Big(\ \pic{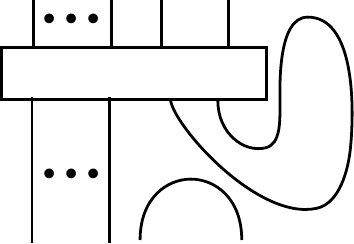}\ \ \ \  \rightarrow \cdots \rightarrow \pic{FKpicTwo_partialTrace}\ \ \ \ \rightarrow t\inv q\pic{FKpic_1_partialTrace}\ \ \ \rightarrow \pic{FKpic_0_partialTrace}\ \  \ \Big)\Bigg)\\
&\buildrel(4)\over\simeq & qy\inv\mathcal{F}\Bigg(\pic{projector} \ , \Big( t\inv q\pic{FKpic_1_partialTrace}\ \ \ \rightarrow \pic{FKpic_0_partialTrace}\ \  \ \Big)\Bigg)\\
&\buildrel(5)\over\simeq & qy\inv \mathcal{F}\Bigg(\pic{projector} \ , q\inv \pic{projector}\  \Bigg)\\
&\buildrel(6)\over\cong & y\inv \mathcal{F}\Big(\pic{projector},\pic{projector} \Big)
\end{eqnarray*}
Let us explain.  First, note that $\Homg$ is invariant under homotopy equivalence of its arguments.  (1) is by definition of $A$ and $C$. (2) and (6) hold by the easily proven fact that if $x$ and $y$  are shifts in bidegree, then $\Homg(xM,zN)\cong x\inv  z\Homg(M,N)$.  (3) holds by Corollary \hyperref[cor-homRotation]{\ref{cor-homRotation}} (we have omitted some of the grading shifts because of space limitations).  The equivalence (4) holds by contracting the contractible complexes of the form $\wwwpic{FKpic_generic_partialTrace} $ and using Theorem \hyperref[convSdrthm]{\ref{convSdrthm}}.  Finally, (5) holds by  Lemma \ref{lemma-delooping}).

By the above computation, we have an isomorphism of homology groups
\[
\Ext^{2,0}(A,C) \cong \Ext^{3-2n,2n}(P_{n-1},P_{n-1}).
\]
This group is zero by Corollary \ref{cor-extGroups}, which applies to $P_{n-1}$ since we assume that $Q_1,\ldots, Q_{n-1}$ exist.  This shows that a convolution $Q_n = (A\rightarrow B\rightarrow C)$ exists, and completes the proof.
\end{proof}

\begin{theorem}\label{thm-QnUniqueness}
Symmetric projectors are unique up to homotopy equivalence.
\end{theorem}
\begin{proof}
Suppose $P_n$ and $P_n'$ are Cooper-Krushkal projectors, and $U_n\in\Endg(P_n)$, $U_n'\in\Endg(P_n')$ represent generators of the degree $(2-2n,2n)$ homology groups (isomorphic to $\Z$).  Then any equivalence $P_n\rightarrow P_n'$ conjugates $U_n$ to $\pm U_n'$ up to homotopy, hence $\Cone(U_n)\simeq \Cone(U_n')$ by Lemma \ref{lemma-coneInvariance}.  Thus, it suffices to show that for any symmetric projector $Q_n$ there is a Cooper-Krushkal projector $P_n$ and a generator $[U_n]\in H^{2-2n,2n}(\Endg(P_n))\cong \Z$ such that $Q_n\simeq \Cone(U_n)$.  For each $Q_n$, the projector $P_n$ and map $U_n$ are provided by Corollary \ref{cor-QfromP}.  The class $[U_n]\in H^{2-2n,2n}(\Endg(P_n))$ is an integer multiple of a generator $[f]$, say $[U_n]=k[f]$.  We must show that $k=\pm 1$.  To see this, one can examine the long exact sequence in homology associated to the mapping cone:
\[
\Homg(P_n,\Cone(k f)) = \Cone\Big(\begin{diagram}\Endg(P_n) & \rTo^{ (k f) \circ (-)} &\Endg(P_n)\end{diagram} \Big)
\]
The relevant part of the long exact sequence looks like
\[
\begin{diagram}[height=2.5em]
\cdots & \rTo &  \Ext^{0,0}(P_n,P_n) & \rTo^{k[f]\circ (-)} & \Ext^{2-2n,2n}(P_n,P_n) & \rTo & \Ext^{2-2n,2n}(P_n,Q_n) & \rTo & \cdots  \\
 && \dTo^{\cong} && \dTo^{\cong} && \dTo^{\cong} && \\
\cdots & \rTo &  \Z & \rTo^k & \Z & \rTo & (\ast) & \rTo & \cdots 
\end{diagram}
\]
A slight modification of the arguments in \S \ref{subsec-endSimplification} shows that the homology of $\Homg(P_n,Q_n)$ is isomorphic to a sub-quotient of $\Z[u_1,\ldots,u_{n-1},\xi_2,\ldots,\xi_n]$.  For degree reasons, there can be no homology in degree $(2-2n,2n)$, and we see that the group ($\ast$) above is zero.  By exactness, this forces $k=\pm 1$, as desired.
\end{proof}

\begin{corollary}\label{cor-QnAreCones}
For $1\leq k\leq n$, let $U_k^{(n)}\in\Endg(P_n)$ represent a generator of the degree $(2-2k,2k)$ homology group, which is isomorphic to $\Z$ by Corollary \ref{cor-extGroups}.  Then $\Cone(U_k^{(n)})\simeq (Q_k\sqcup 1_{n-k})\otimes P_n$.
\end{corollary}
\begin{proof}
In case $k=1$, the claim is obvious, so let us assume that  $1 < k\leq n$.  Let $Q_2,\ldots,Q_n$ and $P_1,\ldots,P_n$ be as in Construction \ref{const-standardCx}.  By Corollary \ref{cor-extGroups} the map $U_k^{(n)}\in\Endg(P_n)$ from Definition \ref{def-polyDef} represents a generator of the relevant homology group.  It is easy to see that $\Cone(U_k^{(n)})\simeq (Q_{k-1}\sqcup 1_1)$.   By construction, $P_n$ is the target of a deformation retract
\[
(P_2'\sqcup 1_{n-2})\otimes \cdots \otimes P_n'\rightarrow P_n
\]
where $P_k' = \Z[u_k]\otimes Q_k$ with differential $1\otimes d+u_k\otimes \partial_k$.  The $u_k$ action on $P_n$ is inherited from the action of $u_k$ on $P_k'$, hence
\begin{eqnarray*}
\Cone(U_k^{(n)}) &\simeq &  (P_2\sqcup 1_{n-2})\otimes \cdots \otimes \Cone(u_k)\otimes \cdots P_n'\\
 &\simeq &  (P_2\sqcup 1_{n-2})\otimes \cdots \otimes (Q_k\sqcup 1_{n-k})\otimes \cdots P_n'\\
\end{eqnarray*}
By projector absorbing, this latter complex is homotopy equivalent to $(Q_k\sqcup 1_{n-k})\otimes P_n$.  This completes the proof.
\end{proof}

As an application of uniqueness, we obtain an expression for $Q_3$:

\begin{example}\label{example-Q3}
Let $C\in \Kom(3)$ denote the following chain complex:
\[
\begin{minipage}{5.8in}
\begin{tikzpicture}
\tikzstyle{every node}=[font=\small]
\node (a) at (.5,0) {$q^6\ \pic{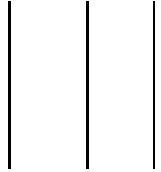}$};
\node at (3,.7) {$q^5\ \pic{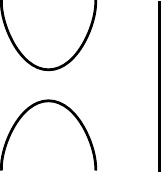}$};
\node (b) at (3.1,0) {$\oplus$};
\node at (3,-.7) {$q^5\ \pic{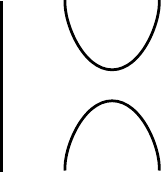}$};
\node at (6,.7) {$q^4\ \pic{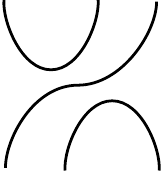}$};
\node (c) at (6.1,0) {$\oplus$};
\node at (6,-.7) {$q^4\ \pic{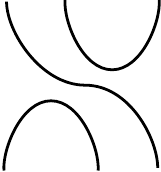}$};
\node at (9,.7) {$q^2\ \pic{e1e2}$};
\node (d) at (9.1,0) {$\oplus$};
\node at (9,-.7) {$q^2\ \pic{e2e1}$};
\node at (12,.7) {$q^5\ \pic{e1}$};
\node (e) at (12.1,0) {$\oplus$};
\node at (12,-.7) {$q^5\ \pic{e2}$};
\node (f) at (14.5,0) {$\underline{\pic{1_3}}$};
\path[->,>=stealth',shorten >=1pt,auto,node distance=1.8cm,
  thick]
(a) edge node[above] {$\a$}	  ([xshift = -6pt] b.west)
([xshift=6pt] b.east)  edge node[above] {$\b$}	 ([xshift = -9pt] c.west)
([xshift=6pt] c.east) edge node[above] {$\gamma$}	  ([xshift = -9pt] d.west)
([xshift=6pt] d.east) edge node[above] {$\b^\ast$}	  ([xshift = -9pt] e.west)
([xshift=6pt] e.east) edge node[above] {$\a^\ast$}	  (f);
\end{tikzpicture}
\end{minipage}
\]
where
\[
\a = \matrix{\mpic{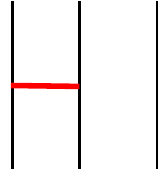} \\ \mpic{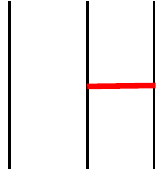}},\hskip.3in
\b = \matrix{\mpic{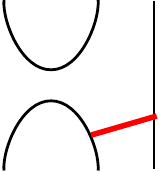} & -\mpic{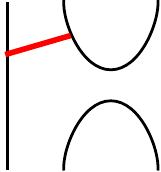}\\ -\mpic{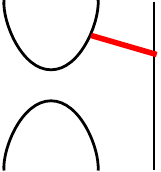} & \mpic{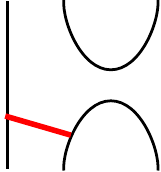}},\hskip.3in
\gamma=\matrix{\mpic{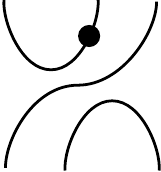} + \mpic{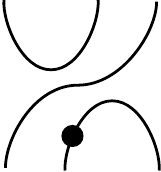} & \mpic{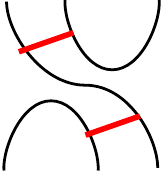} \\ \mpic{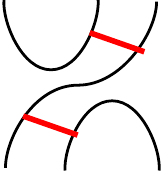} &  \mpic{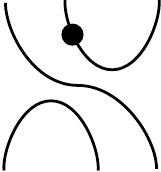} + \mpic{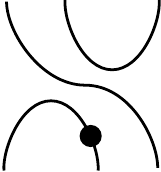} },
\]\vskip7pt
and $(-)^\ast:\bn{n}\rightarrow \bn{n}$ is the contravariant functor which is the identity on objects and flips cobordisms upside down.  In particular $\pic{hsaddle}^\ast = \pic{isaddle}$ and vice versa.  The reader is invited to compare this chain complex with the expression for $P_3$ in \cite{CK12a}. 

One can check that $C$ kills turnbacks.  By projector absorbing, we then have $(P_2\sqcup 1_1)\otimes C\simeq C$.  On the other hand, one can simplify $(P_2\sqcup 1_1)\otimes C$ by expanding $C$ into its chain groups and cancelling the turnbacks which hit $P_2$ (using Theorem \ref{convSdrthm} to contract all such terms).  The resulting complex is readily seen to be a convolution of the symmetric Cooper-Krushkal sequence relative to $P_2$.  Hence $C\simeq Q_3$ by Theorem \ref{thm-QnUniqueness}.
\end{example}
We remark that for $n\geq 4$ the complex $Q_n$ is not homotopy equivalent to a bounded complex.

\begin{definition}\label{def-quasiProjectors}
For any sequence $1\leq i_1,\ldots,i_r\leq n$, put
\[
P_n(i_1,\ldots,i_r):=\Cone(U_{i_1})\otimes \cdots\otimes \Cone(U_{i_r})
\]
where $U_k:=U_k^{(n)}$ is the chain map from Definition \ref{def-polyDef}.  By convention, associated to the empty sequence we have $P_n(\emptyset):=P_n$.  Since the $U_k$ are unique up to homotopy, different choices of $U_k$ give canonically isomorphic complexes $P_n(i_1,\ldots,i_r)$.  We call the complexes $P_n(i_1,\ldots,i_r)$ \emph{quasi-projectors}.
\end{definition}
In light of Corollary \ref{cor-QnAreCones}, these complexes can also be described in terms of tensor products of the $Q_k$.  In particular, $P_n(1,2,\ldots,n)\simeq (Q_1\sqcup 1_1)\otimes (Q_2\sqcup 1_{n-2})\otimes\cdots \otimes Q_n$.  In the remainder of \S \ref{sec-quasiprojectors} we study properties of these complexes.

\subsection{Quasi-idempotency and commuting properties}
\label{subsec-quasiProjProps}
If $e\in A$ is an idempotent of a $k$-algebra $A$ and $\a\in k$ is a scalar, then any $f=\a e$ satisfies $f^2=\a f$.  This property is called \emph{quasi-idempotency}.  Note that $Q_n\simeq (t^{1-2n}q^{2n}P_n\rightarrow P_n)$ categorifies a multiple of the Jones-Wenzl projector in the sense of \hyperref[subsec-eulerChar]{\S \ref{subsec-eulerChar}}.  It is natural to ask whether $Q_n$ is quasi-idempotent up to homotopy, and it is a pleasant surprise that it actually is.  First, a lemma:
\begin{lemma}\label{lemma-mapMerging}
Let $\mu:P_n\otimes P_n\rightarrow P_n$ be a canonical equivalence (Definition \ref{def-canEquiv}) and $\mu\inv:P_n\rightarrow P_n\otimes P_n$ a homotopy inverse.  Let $\Phi_1,\Phi_2:\Endg(P_n)\rightarrow \Endg(P_n)$ denote the maps
\[
\Phi_1(f)=\mu\circ(f\otimes \Id_{P_n})\circ \mu\inv \hskip.7in \Phi_2(f)=\mu\circ(\Id_{P_n}\otimes f)\circ \mu\inv
\]
for all $f\in \Endg(P_n)$.  Then $\Phi_1\simeq \Phi_2\simeq \Id_{\Endg(P_n)}$.
\end{lemma}
\begin{proof}
Let us first remark that $\Homg(-,-)$ respects homotopy in the following sense: suppose $A,B,A',B'$ are chain complexes over an additive category, and $\phi:A'\rightarrow A$, $\psi:B\rightarrow B'$ are chain maps.  Denote by $L_\psi:\Homg(A,B)\rightarrow \Homg(A,B')$ and $R_\phi:\Homg(A,B)\rightarrow \Homg(A',B)$ the chain maps defined by
\[
L_\psi(f)=\psi\circ f \hskip .7in R_\phi(f) = f\circ \phi
\]
for all $f\in\Homg(A,B)$.  If $\psi\simeq \psi_1$ (respectively $\phi\simeq \phi_1$), then $L_{\psi}\simeq L_{\psi_1}$ (respectively, $R_{\phi}\simeq R_{\phi_1}$).  In the present situation, we want to show that $L_{\mu}F R_{\mu\inv}\simeq \Id_{\Endg{P_n}}$ for some $F$.  By the preceding remarks, the validity of this relation is unchanged if we replace $\mu\inv$ by some $\phi\simeq \mu\inv$.

Denote the unit of $P_n$ by $\iota:1_n\rightarrow P_n$.  Then $\iota\otimes \iota:1_n\rightarrow P_n\otimes P_n$ is the unit of $P_n\otimes P_n$.  Thus, $\iota\otimes \Id_{P_n}$ and $\Id_{P_n}\otimes \iota$ are canonical equivalences $P_n\rightarrow P_n\otimes P_n$, as can easily be verified from the definitions (canonical equivalences are defined in Definition \ref{def-canEquiv}).  Uniquness of canonical equivalences (Theorem \ref{thm-canEquiv}) implies that
\[
\mu\inv \simeq \iota\otimes \Id_{P_n} \simeq \Id_{P_n}\otimes \iota
\]
Thus, $\Phi_1$ is homotopic to the map sending
\[
f\mapsto \mu\circ (f\otimes \Id_{P_n})\circ (\Id_{P_n}\otimes \iota) = \mu\circ (\Id_{P_n}\otimes \iota)\circ (\Id_{1_n}\otimes f)
\]
which is homotopic to the identity map on $\Endg(P_n)$, since $\mu\circ (\Id_{P_n}\otimes \iota)\simeq \Id_{P_n}$.  This shows that $\Phi_1$ is homotopic to the identity.  A similar argument shows that $\Phi_2$ is homotopic to the identity.
\end{proof}

\begin{theorem}\label{thm-QnProps}
Recall the complexes $P_n(i_1,\ldots,i_r)$ from Definition \ref{def-quasiProjectors}.  We have:
\begin{enumerate}
\item $P_n(i_1,\ldots,i_r)\simeq P_n(i_{\pi(1)},\ldots,i_{\pi(r)})$ for any permutation $\pi\in S_r$.
\item $P_n(k,k,i_1,\ldots,i_r)\simeq (1+t^{1-2k}q^{2k})P_n(k,i_1,\ldots,i_r)$.
\end{enumerate}
for all indices $1\leq k,i_1,\ldots,i_r\leq n$.  In particular, $Q_n^{\otimes 2} \simeq Q_n\oplus t^{1-2n}q^{2n}Q_n$.
\end{theorem}
\begin{proof}
It is clear that $P_n(\mathbf{i}\cdot\mathbf{j})\simeq P_n(\mathbf{i})\otimes P_n(\mathbf{j})$ for all sequences $\mathbf{i},\mathbf{j}$.  It therefore suffices to prove the statements
\begin{enumerate}
\item $\Cone(U_i)^{\otimes 2}\simeq \Cone(U_i)\oplus t^{1-2i}q^{2i}\Cone(U_i)$ for all $1\leq i\leq n$.
\item $\Cone(U_i)\otimes \Cone(U_j)\simeq \Cone(U_j)\otimes \Cone(U_i)$ for all $1\leq i,j\leq n$.
\end{enumerate}
Let $y_k=t^{1-2k}q^{2+2k}$ denote the grading shift functor and compute:
\[
\Cone(U_i)\otimes \Cone(U_j)=
\Tot\left(
\begin{diagram}
y_iy_j P_n\otimes P_n   & \rTo^{U_i\otimes \Id} & y_j P_n\otimes P_n\\
\dTo^{-\Id\otimes U_j} && \dTo_{\Id\otimes U_j}\\
y_i P_n\otimes P_n   & \rTo^{U_i\otimes \Id} &  P_n\otimes P_n
\end{diagram}\ \ 
\right)
\]
By projector absorbing we have the data $(\pi,\sigma,h)$ of a deformation retract $P_n\otimes P_n\simeq P_n$.  Apply $\pi: P_n\otimes P_n\rightarrow P_n$ (using Theorem \hyperref[convSdrthm]{\ref{convSdrthm}}) to each term to obtain
\[
\Cone(U_i)\otimes \Cone(U_j) \simeq 
\Tot\left(
\begin{diagram}
y_iy_j P_n & \rTo^{\rho_1(U_i)} & y_j P_n\\
\dTo^{-\rho_2(U_j)} &\rdTo^{w} & \dTo_{\rho_2(U_j)}\\
y_i P_n & \rTo^{\rho_1(U_i)} & P_n
\end{diagram}\ \ \ \ \ \ \ 
\right)
\]
where $\rho_1,\rho_2:\Endg(P_n)\rightarrow \Endg(P_n)$ denote the maps
\[
\rho_1(f)=\pi\circ(f\otimes \Id)\circ \sigma\hskip1in   \rho_2(f)=\pi\circ(\Id\otimes f)\otimes \sigma
\]
By Lemma \ref{lemma-mapMerging} we have $\rho_1\simeq \rho_2\simeq\Id_{\Endg(P_n)}$, hence by Theorem \ref{thm-diffPerturbing} we can replace the horizontal maps by $U_i$ and the vertical maps with $\pm U_j$ at the expense of affecting the length-two component of the differential:
\[
\Cone(U_i)\otimes \Cone(U_j) \simeq 
\Tot\left(
\begin{diagram}
y_iy_j P_n & \rTo^{U_i} & y_j P_n\\
\dTo^{-U_j} &\rdTo^{h} & \dTo_{U_j}\\
y_i P_n & \rTo^{U_i} & P_n
\end{diagram}\ \ 
\right)
\]
for some $h\in \Endg(P_n)$ with $[d,h]=[U_i,U_j]$.  If $h'$ is any other choice of diagonal map, then $h-h'\in \Endg(P_n)$ is a cycle of bidegree $(3-2i-2j,2i+2j)$, hence a boundary by Corollary \ref{cor-extGroups}.  Thus the isomorphism type of the right-hand side above does not depend on $h$.  It follows that swapping $i$ and $j$ in the right-hand side of the above gives an isomorphic chain complex, which proves (2).  

Now, specialize to the case $i=j$.  Since the choice of $h$ is irrelevant, we may assume $h=0$, so that
\[
\Cone(U_i)\otimes \Cone(U_i)\simeq 
\Tot\left(
\begin{diagram}
y_iy_j P_n & \rTo^{U_i} & y_j P_n\\
\dTo^{-U_i} & & \dTo_{U_i}\\
y_i P_n & \rTo^{U_i} & P_n
\end{diagram}\ \ 
\right)
\]
After performing an elementary similarity transform to the matrix $\smMatrix{ d &0&0&0\\ U_i& -d & 0 &0 \\ -U_i & 0 & -d & 0 \\ 0 &U_i &U_i & d}$ (namely add the second row to the third while subtracting the third column from the second) we can replace the vertical maps with zeroes up to isomorphism.  The result will be a chain complex which is isomorphic to $\Cone(U_i)\oplus t^{1-2i}q^{2i}\Cone(U_i)$.  This proves (1).
\end{proof}

\subsection{Improving the statement of boundedness}
\label{subsec-improvingBddness}
Sections are \S \ref{subsec-improvingBddness} and \S \ref{subsec-torusBraids} can be skipped without interrupting the logical flow of the paper.

\begin{theorem}\label{thm-nilpotence}
Let $(i_1,\ldots,i_n)$ be be the sequence obtained from $(1,2,\ldots,n)$ by removing the indices $k$ for which $U_k\in\Endg(P_n)$ is nilpotent up to homotopy.  Then $P_n(i_1,\ldots,i_n)$ is homotopy equivalent to a bounded complex.
\end{theorem}

The key observation in the proof of this theorem is the following:
\begin{lemma}\label{lemma-koszulPowers}
For any chain complex $C\in\Kom(n)$ and any closed morphism $f\in\Endg(C)$, the mapping cone $\Cone(f^m)$ can be expressed as a total complex of a bicomplex:
\[
\Cone(f^m) \simeq \Tot\Big(\Cone(f)\rightarrow \lambda\Cone(f) \rightarrow \cdots \rightarrow \lambda^m\Cone(f)\Big)
\]
\end{lemma}
\begin{proof}
Observe:
\begin{equation}\label{eq-koszulPower}
\Cone(f^m)\simeq \Tot\left(
\begin{diagram}
t\inv\lambda C & \rTo^{f} & C\\
& \rdTo^{-\Id} & \\ 
t\inv \lambda^2 C & \rTo^{f} & \lambda C\\
& \rdTo^{-\Id} & \\
&& \\
\vdots &  & \vdots\\
&& \\
& \rdTo^{-\Id} & \\
t\inv \lambda^kC & \rTo^{f} & \lambda^{k-1}C
\end{diagram}\ \ \ \ \right)
\end{equation}
where $\lambda = t^{\deg_h(f)}q^{\deg_q(f)}$ is the degree shift necessary in order to have a degree $(1,0)$ differential.  We have written $\Tot$ to emphasize that the right-hand side above is the direct sum of shifted copies of $C$, whose total differential is is the sum all of the indicated morphisms together with the differential on each copy of $C$.  The homotopy equivalence comes from canceling the isomorphism using Gaussian elimination (Proposition \ref{prop-gauss}).
\end{proof}

\begin{proof}[Proof of Theorem \ref{thm-nilpotence}]
By Theorem \ref{thm-QnProps} we can assume that the indices $i_j$ are distinct, and ordered as we please.  Throughout, we write $C\simeq \text{bounded}$ whenever $C$ is homotopy equivalent to a bounded chain complex. By Theorem \ref{thm-standardCx}, we know that $P_n(1,2,\ldots,n) \simeq \text{bounded}$.  We must show that the indices corresponding to nilpotent $[U_k]$'s can be omitted.  Assume $P_n(k,i_1,\ldots,i_r)\simeq \text{bounded}$, and $U_k\in \Endg(P_n)$ satisfies $U_k^m\simeq 0$.  We will show that $P_n(i_1,\ldots,i_r) \simeq \text{bounded}$.

By construction, $P_n(k,i_1,\ldots,i_n) \simeq  \Cone(U_k) \otimes P_n(i_1,\ldots,i_n)$ which is $\simeq \text{bounded}$ by hypothesis.  By projector absorbing, we have $P_n(i_1,\ldots,i_n)\simeq P_n\otimes P_n(i_1,\ldots,i_r)$.  The action of $U_k$ on $P_n$ gives a closed morphism $f=U_k\otimes \Id\in\Endg(P_n\otimes P_n(i_1,\ldots,i_r))$.  Observe:
\begin{enumerate}
\item $\Cone(f)=\Cone(U_k)\otimes P_n(i_1,\ldots,i_r)\simeq \text{bounded}$. 
\item $f^m\simeq 0$ by hypothesis, so $\Cone(f^m)$ splits as two copies of $P_n\otimes P_n(i_1,\ldots,i_r)$, with degree shifts.
\item By Lemma \ref{lemma-koszulPowers}, $\Cone(f^m)$ is homotopy equivalent to a complex built out of finitely many copies of $\Cone(f)$.
\item Since $\Cone(f)\simeq \text{bounded}$ (by item (1)), it follows that $\Cone(f^m)\simeq \text{bounded}$. 
\end{enumerate}
Comparing items (2) and (4), the only possibility is that $P_n(i_1,\ldots,i_r)\simeq \text{bounded}$.  This completes the proof.
\end{proof}

The converse of Theorem \ref{thm-nilpotence} is also true, though we will not prove it.  Thus, Example \ref{example-Q3} shows that the action of $U_2$ on $P_3$ is nilpotent up to homotopy.  In general we have:
\begin{conjecture}\label{conj-nilpBdd}
For each $1\leq k\leq n$, let $U_k^{(n)}\in\Endg(P_n)$ denote the map from Definition \ref{def-polyDef}.  Then $[U_k^{(n)}]$ is nilpotent if and only if $k\leq (n+1)/2$.  Correspondingly, $P_n(i_1,\ldots,i_r)\simeq \text{ bounded}$ if and only if $\{i_1,\ldots,i_r\}$ contains each $k$ with $(n+1)/2< k\leq n$.
\end{conjecture}

\subsection{Relation with torus braids}
\label{subsec-torusBraids}
It is well-known that categorified idempotents are related to infinite torus braids \cite{Roz10a,Ros11, Cau12}.  It turns out that our quasi-idempotent complexes $Q_n$ are also related to torus braids.

Recall the chain complex appearing in the unoriented Bar-Natan-Khovanov skein relation: $\llbracket \pic{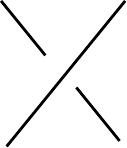}\rrbracket := q\pic{turnback}\buildrel\mpic{isaddle}\over\longrightarrow \underline{\pic{straightThrough}}$.  We will usually omit the braces; all of our pictures should be interpreted as objects of the appropriate $\Kom(n)$.
\begin{remark}
The complex which we associate to the unoriented crossing differs from the complexes (\ref{eq-posCrossing}) and (\ref{eq-negCrossing}) up to an overal degree shift.
\end{remark}
Denote $P_{n-1}$ graphically by a white box $\pic{projector}$ with $n-1$ strands attached to the top and bottom.  By expanding crossings and contracting a large contractible summand, one can show that $\wpic{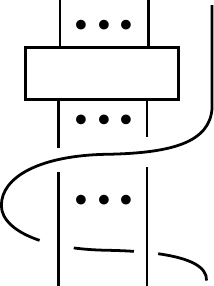}$ is homotopy equivalent to a convolution of a truncation of the homotopy chain complex (\hyperref[eq-symmetricSeq]{\ref{eq-symmetricSeq}}):
\begin{equation}\label{eq-partialTwist}
\wpic{projector-partialTwist} \simeq \Big(q^{2n-1}\FKone\rightarrow \cdots \rightarrow q^{n+1}\FKult \rightarrow q^{n-1}\FKult \rightarrow \cdots \rightarrow q \FKone\rightarrow \underline{\FKzero}\Big)
\end{equation}
Recall the partial trace functor $T:\bn{n}\rightarrow \bn{n-1}$.  By contracting terms of the form $\wwwpic{FKpic_generic_partialTrace}$ and delooping (Lemma \ref{lemma-delooping}) we see that
\[
\wpic{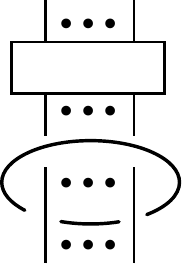} = T\Big(\wpic{projector-partialTwist}\Big) \simeq \Big(t^{2-2n}q^{2n-1}\pic{projector}\ \buildrel \b\over\longrightarrow  q\inv \pic{projector}\Big).
\]
Examining degrees, we see that $\b$ is a bidegree $(3-2n,2n)$ element of $\Endg(P_{n-1})$.  The condition that the above is a chain complexes implies that $\b$ is a cycle, which by Corollary \ref{cor-extGroups} we know must be a boundary.  It follows that:
\begin{proposition}
Let $\pic{projector}=P_{n-1} \in\Kom(n-1)$ denote a Cooper-Krushkal projector and $\pic{crossing}$ denote the two-term complex defined at the beginning of this section.  Then
\[
q\wpic{projector-ring}  \simeq \pic{projector}\oplus t^{2-2n}q^{2n}\pic{projector}
\]
\end{proposition}
This result can be used to give an alternate construction of symmetric projectors:
\begin{proposition}
There is a chain map $\phi_n:t^{2-2n}q^{2n}\FKzero \rightarrow  \wpic{projector-partialTwist}$ such that $\Cone(\phi_n)\in\Kom(n)$ is homotopy equivalent to a symmetric projector, where $\pic{projector}=P_{n-1}$.
\end{proposition}
\begin{proof}
Theorem \ref{thm-partialTraceHom} implies that
\[
\Homg\Big(t^{2-2n}q^{2n}\FKzero, \wpic{projector-partialTwist}\Big) \cong t^{2n-2}q^{1-2n}\Homg\Big(\pic{projector},\wpic{projector-ring}\Big) \simeq (1+t^{2n-2}q^{-2n})\Endg(\pic{projector})
\]
Taking the degree $(0,0)$ homology groups gives
\[
\Ext^{0,0}\Big(t^{2-2n}q^{2n}\FKzero, \wpic{projector-partialTwist}\Big) \cong \Ext^{0,0}(P_{n-1},P_{n-1})\oplus \Ext^{2-2n,2n}(P_{n-1},P_{n-1})\cong \Z\oplus 0
\]
Let $\phi_n$ denote a generator of the $\Ext$ group on the left above.  We leave it to the reader that $\Cone(\phi_n)$ is equivalent to a symmetric projector (compare with (\ref{eq-partialTwist})).
\end{proof}
This result can be used to give an independent proof that the Cooper-Krushkal projectors are limits of powers of the full twist, as in \cite{Roz10a}.  Moreover, composing the map $\phi_n$ with $P_n$ gives a description of the map $U_n$ from Definition \ref{def-polyDef} as the composition
\[
t^{2-2n}q^{2n}\wpic{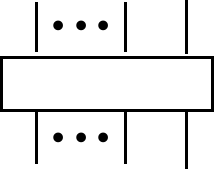} \buildrel\simeq\over\longrightarrow t^{2-2n}q^{2n}\wpic{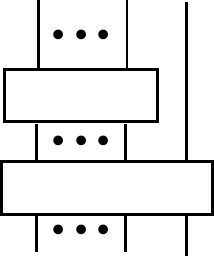} \buildrel (\phi_n\otimes \Id)\over\longrightarrow \wpic{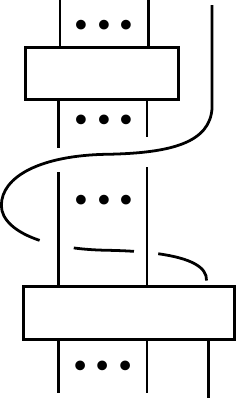} \buildrel\simeq\over\longrightarrow \wpic{uTwistPic0}
\]

\section{Local and quasi-local colored $\sl_2$ link homology}
\label{sec-linkHomology}
In this section we show that the action of $\Z[u_1,\ldots,u_n]$ on $P_n$ induces a well-defined action on the colored Khovanov homology due to Cooper-Krushkal \cite{CK12a}.  We establish properties of the symmetric projectors $Q_n$ and show that they can be used to define a new homology theory, in which they play a role analogous to $P_n$.  We use the adjective \emph{quasi-local} to mean that the link invariants which we construct can be defined for tangles, but they respect gluing of tangles only up to taking direct sums.  Another way of saying this is that we construct an invariant of pairs $(T,X)$ where $T$ is a tangle and $X\subset T$ is a finite collection of marked points.  There is a mild penalty for merging two marked points (see Proposition \ref{prop-mergingPoints}).

\subsection{A new family of link homologies}
\label{subsec-quasiHomology}
Recall the complexes associated to crossings in Khovanov homology:
\begin{eqnarray}
\Big\llbracket \wpic{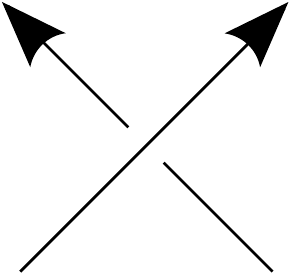}\Big\rrbracket &:=&\bigg(0 \longrightarrow q^2\bpic{turnback}\buildrel\mpic{isaddle}\over\longrightarrow q\underline{\bpic{straightThrough}}\longrightarrow 0\bigg) \label{eq-posCrossing}\\
\Big\llbracket \wpic{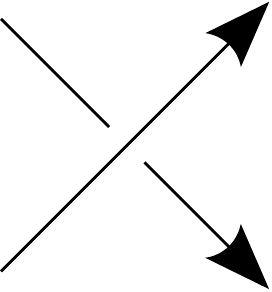}\Big\rrbracket &:=& \bigg(0\longrightarrow q\inv \underline{\bpic{turnback}} \buildrel\mpic{isaddle}\over\longrightarrow q^{-2}\bpic{straightThrough}\longrightarrow 0\bigg) \label{eq-negCrossing}
\end{eqnarray}
where we have underlined the terms in homological degree zero.  Notice that the positive and negative crossings differ only in an overall degree shift.

\begin{definition}[The bracket complex]\label{def-brackeCx}
Fix as initial data a family $\mathcal{K} = \{K_n\in \Kom^-(n)\:|\: n=0,1,2,\ldots\}$ of complexes.  Let $D\subset D^2$ be an oriented tangle diagram whose components are labeled with non-negative integers, called the colors.  Assume $D$ is equipped with some marked points $\{v_i\}$ away from the crossings, with at least one on each component of the underlying tangle.  We will define a chain complex $\llbracket D; \mathcal{K}\rrbracket$ over the appropriate Bar-Natan's category as follows.  Obtain a new diagram by replacing an $n$-colored component by $n$-parallel copies of itself, with alternating orientations:
\[
\bigg\llbracket
\begin{minipage}{.6in}
\labellist
\small
\pinlabel $n$ at 2 22
\pinlabel $m$ at 78 22
\endlabellist
\begin{center}\includegraphics[scale=.4]{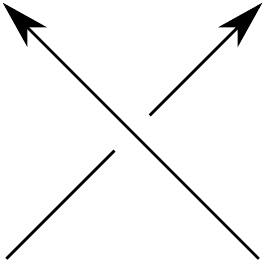}\end{center} 
\end{minipage}
\bigg\rrbracket :=
\Bigg\llbracket
\begin{minipage}{.8in}
\labellist
\small
\pinlabel $n$ at -3 -8
\pinlabel $m$ at 110 -10
\pinlabel {$\rotatebox{-45}{\Big\}}$} at 100 5
\pinlabel {$\rotatebox{45}{\Big\{}$} at 7 7
\endlabellist
\begin{center}\includegraphics[scale=.4]{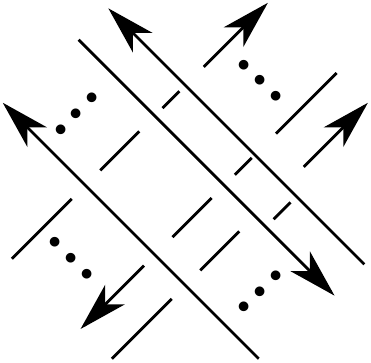}\end{center} 
\end{minipage}
\Bigg\rrbracket
\]
To each marked point we insert a labelled white box in a corresponding location in the cabled diagram:
\[
\bigg\llbracket
\begin{minipage}{.3in}
\labellist
\small
\pinlabel $n$ at -10 30
\endlabellist
\begin{center}\includegraphics[scale=.3]{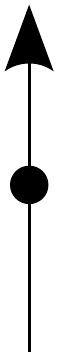}\end{center} 
\end{minipage}
; \mathcal{K} \bigg\rrbracket := \bigg\llbracket
\begin{minipage}{.45in}
\labellist
\small
\pinlabel $K_n$ at 35 40
\endlabellist
\begin{center}\includegraphics[scale=.3]{emptybox}\end{center} 
\end{minipage}
\bigg\rrbracket
\]
Now, by taking the planar composition of oriented crossings and $K_n$'s, we obtain an object $\llbracket D ;\mathcal{K} \rrbracket$ of some $\Kom^-(N)$ (recall that gluing of diagrams in the plane induces multilinear functors on Bar-Natan's categories $\bn{n}$; this categorical planar algebra structure is called a canopoly in \cite{B-N05}).  Strictly speaking, this planar composition requires an ordering on the set of crossings and marked points of $D$, but any two choices give canonically isomorphic complexes.
\end{definition}
In this paper we consider only the case where $\mathcal{K}$ is a collection of quasi-projectors (Definition \ref{def-quasiProjectors}).  That is, $\mathcal{K} = \{P_n(\mathbf{i}_n)\:|\:n=1,2,\ldots\}$ for some sequences $\mathbf{i}_n$.  In this case the chain homotopy type of $\llbracket D;\mathcal{K}\rrbracket$ is an invariant of the colored, framed, oriented, marked tangle represented by $D$.  We will prove this in steps, by establising a number of local relations that $\llbracket D; \mathcal{K}\rrbracket$ satisfies.  Below, $\mathcal{K}=\{K_1,K_2,\ldots \}$ denotes a fixed family of quasi-projectors.
\begin{proposition}\label{prop-reid2and3}
Away from the marked points, the bracket complex is invariant under the framed, colored Reidemeister moves
\[
\bigg\llbracket
\begin{minipage}{.35in}
\labellist
\small
\pinlabel $n$ at -13 40
\endlabellist
\begin{center}\includegraphics[scale=.4]{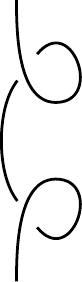}\end{center} 
\end{minipage}
\bigg\rrbracket
\simeq
\bigg\llbracket
\begin{minipage}{.35in}
\labellist
\small
\pinlabel $n$ at -10 30
\endlabellist
\begin{center}\includegraphics[scale=.4]{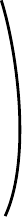}\end{center} 
\end{minipage}
\bigg\rrbracket
\simeq
\bigg\llbracket
\begin{minipage}{.35in}
\labellist
\small
\pinlabel $n$ at -13 40
\endlabellist
\begin{center}\includegraphics[scale=.4]{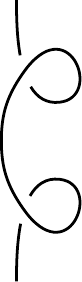}\end{center} 
\end{minipage}
\bigg\rrbracket
\hskip.6in
\bigg\llbracket
\begin{minipage}{.55in}
\labellist
\small
\pinlabel $n$ at -15 30
\pinlabel $m$ at 55 30
\endlabellist
\begin{center}\includegraphics[scale=.4]{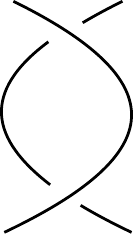}\end{center} 
\end{minipage}
\bigg\rrbracket
\simeq
\bigg\llbracket
\begin{minipage}{.55in}
\labellist
\small
\pinlabel $m$ at -15 30
\pinlabel $n$ at 50 30
\endlabellist
\begin{center}\includegraphics[scale=.4]{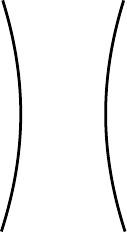}\end{center} 
\end{minipage}
\bigg\rrbracket
\]
and
\[
\bigg\llbracket
\begin{minipage}{.7in}
\labellist
\small
\pinlabel $n$ at 80 50
\pinlabel $m$ at 80 10
\pinlabel $k$ at -10 35
\endlabellist
\begin{center}\includegraphics[scale=.4]{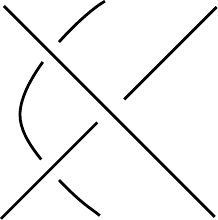}\end{center} 
\end{minipage}
\bigg\rrbracket
\simeq
\bigg\llbracket
\begin{minipage}{.7in}
\labellist
\small
\pinlabel $m$ at -15 50
\pinlabel $n$ at -15 10
\pinlabel $k$ at 75 35
\endlabellist
\begin{center}\includegraphics[scale=.4]{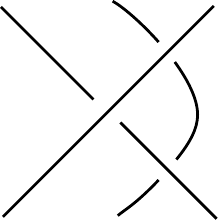}\end{center} 
\end{minipage}
\bigg\rrbracket
\]
with arbitrary orientations.
\end{proposition}
\begin{proof}
This follows from repeated use of the invariance of the usual Khovanov tangle invariant under the Reidemeister moves.
\end{proof}

\begin{proposition}\label{prop-framingChange}
The dependence of $\llbracket D ;\mathcal{K}\rrbracket$ on framing is given by:
\[
\bigg\llbracket
\begin{minipage}{.4in}
\labellist
\small
\pinlabel $n$ at -10 50
\endlabellist
\begin{center}\includegraphics[scale=.3]{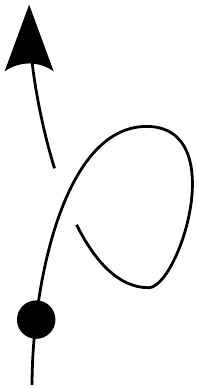}\end{center} 
\end{minipage}
; \mathcal{K} \bigg\rrbracket
\simeq
G_n
\bigg\llbracket
\begin{minipage}{.3in}
\labellist
\small
\pinlabel $n$ at -15 50
\endlabellist
\begin{center}\includegraphics[scale=.3]{upstrandWithDot}\end{center} 
\end{minipage}
; \mathcal{K} \bigg\rrbracket
\hskip.8in
\bigg\llbracket
\begin{minipage}{.4in}
\labellist
\small
\pinlabel $n$ at -10 50
\endlabellist
\begin{center}\includegraphics[scale=.3]{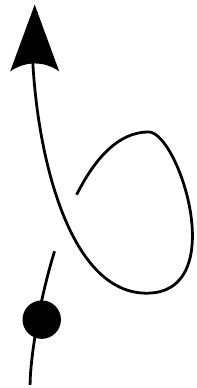}\end{center} 
\end{minipage}
; \mathcal{K} \bigg\rrbracket
\simeq
G\inv_n
\bigg\llbracket
\begin{minipage}{.3in}
\labellist
\small
\pinlabel $n$ at -15 50
\endlabellist
\begin{center}\includegraphics[scale=.3]{upstrandWithDot}\end{center} 
\end{minipage}
; \mathcal{K} \bigg\rrbracket
\]
where $G_n = t^{n^2/2}q^{-(n^2+2n)/2}$ for $n$ even and $G_n =  t^{(n^2-1)/2}q^{-(n^2+2n-3)/2}$ for $n$ odd.
\end{proposition}
\begin{proof}
Let $\Tw_n$ and $\Tw_n\inv$ denote the braids corresponding to the right-handed (respectively, left-handed) full twist on $n$ strands with alternating orientations.   We must check that $\llbracket \Tw_n^{\pm}\rrbracket \otimes K_n\simeq G_n^{\pm 1}(K_n)$, where $G_n$ is the stated grading shift functor.  By Reidemeister invariance of $\llbracket-\rrbracket$, we have $\llbracket \Tw_n\rrbracket\otimes \llbracket \Tw_n\inv\rrbracket \simeq 1_n$.  Thus, either of these equivalences implies the other.

Note that $\Tw_n$ is the $n$-cycle
\[
C = \wwpic{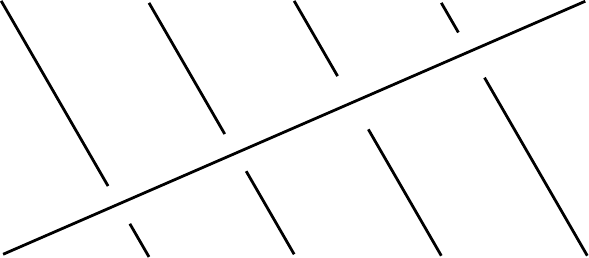}
\]
composed with itself $n$ times.  This is a pure braid with $n(n-1)$ crossings.  Once we give $\Tw_n$ the alternating orientations ``up, down, up, $\ldots$, '' the number of positive and negative crossings in $\Tw_n$ is
\[
m_+(n) = \begin{cases}\frac{1}{2}(n^2-2n) & \text{ if $n$ is even} \\ \frac{1}{2}(n^2-2n+1) & \text{ if $n$ is odd}\end{cases}
\hskip1in
m_-(n) = \begin{cases}\frac{1}{2}n^2 & \text{ if $n$ is even}  \\ \frac{1}{2}(n^2-1)  & \text{ if $n$ is odd}\end{cases}
\]
Define complexes
\[
X^+:= \Big\llbracket \wpic{+crossing3}\Big\rrbracket\hskip.6in X^-:=\Big\llbracket \wpic{minusCrossing3}\Big\rrbracket\hskip.6in X_i^\pm:=1_{i-1}\sqcup X^\pm\sqcup 1_{n-i-1}
\]
as in (\ref{eq-posCrossing}) and (\ref{eq-negCrossing}).  Observe that $\llbracket\Tw_n \rrbracket\in\Kom(n)$ is a tensor product of $m_+(n)$ complexes of the form $X_i^+$ and $m_-(n)$ complexes of the form $X_i^-$, for various $i$.    Since $K_n$ kills turnbacks, we see directly from the definitions of the crossing $X_i^{\pm}$ that
\[
X_i^+\otimes K_n \simeq qK_n\hskip 1in X_i^-\otimes K\simeq  tq^{-2}K_n
\]
It follows that $\Tw_n\otimes K_n \simeq t^{m_-(n)}q^{m_+(n)-2m_-(n)} K_n$.  This proves the proposition, given our formulas for $m_{\pm}(n)$.
\end{proof}

\begin{proposition}\label{prop-orientationChange}
If $D$ and $D'$ are identical except for the choice of orientation, then $\llbracket D; \mathcal{K}\rrbracket$ and $ \llbracket D; \mathcal{K}\rrbracket $ are homotopy equivalent up to an overall degree shift.
\end{proposition}
\begin{proof}
First, note that the complexes associated to positive (\ref{eq-posCrossing}) and negative crossings (\ref{eq-negCrossing}) differ only up to an overall degree shift, so the proposition holds if $D$ has no marked points.  If $D$ has marked points, then a change of orientation has the local effect of replacing some complexes $K_n$ by their rotated versions:
\[
\bigg\llbracket
\begin{minipage}{.3in}
\labellist
\small
\pinlabel $n$ at -15 50
\endlabellist
\begin{center}\includegraphics[scale=.3]{upstrandWithDot}\end{center} 
\end{minipage}
; \mathcal{K} \bigg\rrbracket
 = \bigg\llbracket
\begin{minipage}{.45in}
\labellist
\small
\pinlabel $K_n$ at 35 40
\endlabellist
\begin{center}\includegraphics[scale=.3]{emptybox}\end{center} 
\end{minipage}
\bigg\rrbracket\hskip.7in
\bigg\llbracket
\begin{minipage}{.3in}
\labellist
\small
\pinlabel $n$ at -15 50
\endlabellist
\begin{center}\includegraphics[scale=.3]{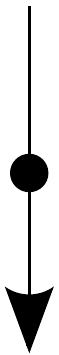}\end{center} 
\end{minipage}
; \mathcal{K} \bigg\rrbracket
\cong \Bigg\llbracket
\begin{minipage}{1in}
\labellist
\small
\pinlabel $K_n$ at 100 80
\endlabellist
\begin{center}\includegraphics[scale=.3]{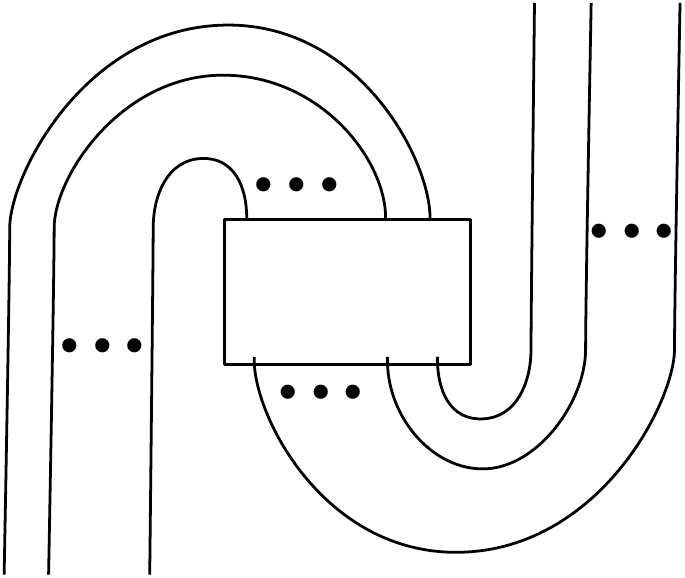}\end{center} 
\end{minipage}
\Bigg\rrbracket
\]
Let us prove that $K_n$ is equivalent to its rotation.  It will follow that a change of orientation preserves $\llbracket D;\mathcal{K}\rrbracket$ up to homotopy equivalence and a well-understood overall degree shift.

Let $r:\Kom(n)\rightarrow \Kom(n)$ be the covariant functor given by rotation by $\pi$ radians.  Note that $r^2$ is isomorphic to the identity functor, since isotopic tangles give isomorphic objects in $\bn{n}$.  Now, if $P_n\in\Kom(n)$ is a Cooper-Krushkal projector, then so is $r(P_n)$.  So $r(P_n)\simeq P_n$ by uniqueness of Cooper-Krushkal projectors.  This completes the proof in case $\mathbf{i}=\emptyset$.  Thus, we may assume $\mathbf{i}=(i_1,\ldots,i_r)$ is non-empty.  In this case, compute:
\[
r\Big(P_n(i_1,\ldots,i_r)\Big) = r\Big(P_n(i_1)\otimes\cdots \otimes P_n(i_r)\Big) \cong r(P_n(i_r))\otimes \cdots \otimes r(P_n(i_1))
\]
where we have used the graphically obvious fact that $r(A\otimes B)\cong r(B)\otimes r(A)$.   The $P_n(i_j)$ commute up to homotopy by Theorem \ref{thm-QnProps}, so it suffices to show that $r(P_n(i))\simeq P_n(i)$ for each $i=1,\ldots,n$.  That is to say, suppose $[U_i]\in \Ext^{2-2i,2i}(P_n,P_n)\cong \Z$ is a generator.  We must show that $r(\Cone(U_i))\simeq \Cone(U_i)$.

Note that $r$ is an automorphism of categories, and commutes with grading shifts.   It follows that $[r(U_i)]\in \Ext^{2-2i,2i}(r(P_n),r(P_n))$ is a generator.  Corollary \ref{cor-QnAreCones} now says $\Cone(U_i)\simeq \Cone(r(U_i))\simeq (Q_i\sqcup 1_{n-i})\otimes P_n$.   On the other hand, since $r$ is a linear functor, it commutes with mapping cones, and we have
\[
r(\Cone(U_i)) \cong \Cone(r(U_i)) \simeq \Cone(U_i),
\]
as desired. This completes the proof.
\end{proof}

In the following, we will regard any Laurent polynomial $f(q,t)\in\Z[q^{\pm 1},t^{\pm 1}]$ as a functor $\Kom(n)\rightarrow \Kom(n)$ which sends any $A\in\Kom(n)$ to a finite direct sum of copies of $A$ with degree shifts given by the terms of $f(q,t)$.
\begin{proposition}\label{prop-mergingPoints}
If $K_n = P_n(i_1,\ldots,i_r)$, then
\[
\bigg\llbracket
\begin{minipage}{.3in}
\labellist
\small
\pinlabel $n$ at -15 50
\endlabellist
\begin{center}\includegraphics[scale=.3]{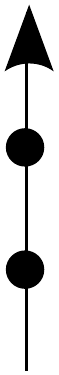}\end{center} 
\end{minipage}
; \mathcal{K} \bigg\rrbracket
 \simeq f(q,t) 
 \bigg\llbracket
\begin{minipage}{.3in}
\labellist
\small
\pinlabel $n$ at -15 50
\endlabellist
\begin{center}\includegraphics[scale=.3]{upstrandWithDot}\end{center} 
\end{minipage}
; \mathcal{K} \bigg\rrbracket
\]
where $f(q,t) = \prod_{j=1}^r(1+t^{1-2i_j}q^{2i_j})$.
\end{proposition}
\begin{proof}
We must show that $P_n(\mathbf{i})^{\otimes 2}\simeq f(q,t)P_n(\mathbf{i})$, where $f(q,t)$ is as in the hypotheses.  In case $\mathbf{i}=\emptyset$, this reduces to idempotency of $P_n$.  So assume that $\mathbf{i}=(i_1,\ldots,i_r)$ is non-empty.  Recall that $P_n(i_1,\ldots,i_r)=P_n(i_1)\otimes \cdots\otimes P_n(i_r)$,  and the $P_n(i_j)$ commute up to equivalence by Theorem \ref{thm-QnProps}.  Thus, it suffices to assume that $r=1$.  Quasi-idempotency of $P_n(i)=\Cone(U_i)$ was proven in the proof of Theorem \ref{thm-QnProps}.
\end{proof}

We have our final local relation:
\begin{proposition}\label{prop-slidingPoints}
Marked points can be slid past crossings:
\[
\bigg\llbracket
\begin{minipage}{.6in}
\labellist
\small
\pinlabel $n$ at -10 20
\pinlabel $m$ at 90 20
\endlabellist
\begin{center}\includegraphics[scale=.3]{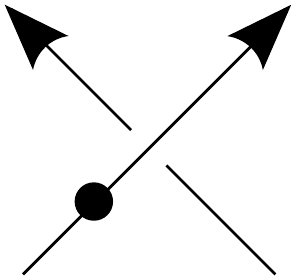}\end{center} 
\end{minipage}
; \mathcal{K} \bigg\rrbracket
\simeq
\bigg\llbracket
\begin{minipage}{.6in}
\labellist
\small
\pinlabel $n$ at -10 20
\pinlabel $m$ at 90 20
\endlabellist
\begin{center}\includegraphics[scale=.3]{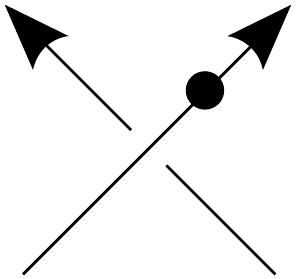}\end{center} 
\end{minipage}
; \mathcal{K} \bigg\rrbracket
\hskip.8in
\bigg\llbracket
\begin{minipage}{.6in}
\labellist
\small
\pinlabel $n$ at -10 20
\pinlabel $m$ at 90 20
\endlabellist
\begin{center}\includegraphics[scale=.3]{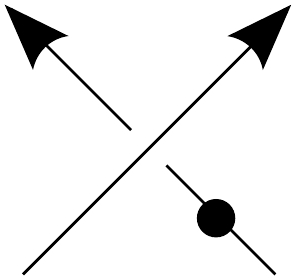}\end{center} 
\end{minipage}
; \mathcal{K} \bigg\rrbracket
\simeq
\bigg\llbracket
\begin{minipage}{.6in}
\labellist
\small
\pinlabel $n$ at -10 20
\pinlabel $m$ at 90 20
\endlabellist
\begin{center}\includegraphics[scale=.3]{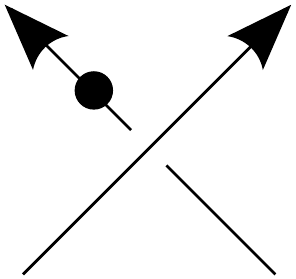}\end{center} 
\end{minipage}
; \mathcal{K} \bigg\rrbracket
\]
and similarly for the negative crossing.
\end{proposition}
We will postpone the proof of this proposition until \S \ref{subsec-slidingQprojectors}.  Granting this, we obtain the following:

\begin{theorem}\label{thm-quasiLocalHomology}
Let $\mathcal{K} = \{P_n(\mathbf{i}_n)\:|\: n=1,2,\ldots\}$ be a family of quasi-projectors.  The chain homotopy type of $\llbracket D; \mathcal{K} \rrbracket$ is an invariant of the colored, framed, oriented, marked tangle represented by $D$ up to framed isotopy.  This invariant satisfies:
\begin{enumerate}
\item A change of framing or orientation affects $\llbracket D;\mathcal{K}\rrbracket$ only up to an overall degree shift.
\item Suppose $D$ and $D'$ are identical, except that $D'$ has 1 fewer marked point than $D$, on a component colored by $n$.  Then $\llbracket D; \mathcal{K} \rrbracket$ is chain homotopy equivalent to a direct sum of copies of $\llbracket D'; \mathcal{K} \rrbracket$ with degree shifts, depending only on $n$.
\item For special choices of $\mathcal{K}$ (not depending on $D$), $\llbracket D;\mathcal{K}\rrbracket$ is homotopy equivalent to a bounded complex.
\end{enumerate}\qed
\end{theorem}
See Theorem \ref{thm-nilpotence} (respectively Conjecture \ref{conj-nilpBdd}) what we know (respectively expect) to be true about boundedness of the complexes $P_n(\mathbf{i})$.

\subsection{Sliding quasi-idempotents past crossings}
\label{subsec-slidingQprojectors}
This section is dedicated to the proof of Proposition \ref{prop-slidingPoints}.

It will be useful to set up some notation, which we will use throughout the rest of this section.  Fix an integer $n\geq 1$.  For each $1\leq k\leq n$, let $U_k:t^{2-2k}q^{2k}P_n\rightarrow P_n$ represent a fixed generator of $\Ext^{2-2k,2k}(P_n,P_n)\cong \Z$.  $X_n:= \llbracket\wpic{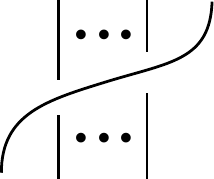}\rrbracket\in \Kom(n+1)$ in which the ``under'' strands have been given alternating orientations.  Define a bilinear functor $F(A,B):=(A\sqcup 1_1)\otimes X_n\otimes (1_1\otimes B)$.  That is to say:
\[
F(A,B) = \left\llbracket
\begin{minipage}{.6in}
\labellist
\small
\pinlabel $A$ at 55 110
\pinlabel $B$ at 55 35
\endlabellist
\begin{center}\includegraphics[scale=.3]{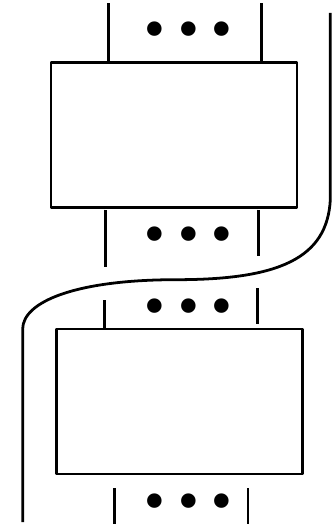}\end{center} 
\end{minipage}
\right\rrbracket
\]
\begin{lemma}[Sliding projectors past crossings]\label{lemma-Pslide}
Let $(P_n,\iota)$ be a Cooper-Krushkal projector, and suppose $K\in\Kom^-(n)$ kills turnbacks.  Then $F(\Id_K,\iota):F(K,1_n)\rightarrow F(K,P_n)$ is a homotopy equivalence.  Similarly, $F(\iota,\Id_K)$ is a homotopy equivalence.
\end{lemma}
In case $K=P_n$, we recover the fact that projectors slide past crossings:
\[
F(P_n,1_n)\simeq F(P_n,P_n) \simeq F(1_n,P_n)
\]
\begin{proof}
We will prove only the first statement.  The proof of the second is similar.

It is a standard fact in homotopy theory that a chain map $f$ is a homotopy equivalence if and only if $\Cone(f)$ is contractible.  Observe:
\[
\Cone(F(\Id_{K},\iota))\cong F(K,\Cone(\iota))
\]
by linearity of the functor $F(K,-)$.  We must show that this latter complex is contractible.

By invariance of the Khovanov tangle invariant under the Reidemeister II move, we have $F(K,e_{i+1})\simeq F(K\otimes e_i,1_n)$ for each $1\leq i\leq n-1$, where $e_i$ denotes the Temperley-Lieb ``cup-cap'' generator.  If $K$ kill turnbacks, then this implies that $F(K,e_{i+1})\simeq 0$ for each $1\leq i\leq n-1$.  An argument similar to the proof of Proposition \ref{prop-turnbackTFAE} now implies that $F(K,\Cone(\iota))\simeq 0$, as desired.
\end{proof}

\begin{lemma}\label{lemma-UslideHom}
There exist homotopy equivalences $\Phi_L,\Phi_R:\Endg(F(P_n,P_n))\simeq \Endg(P_n)\otimes \Endg(1_1)$ such that $\Phi_L(f,\Id_{P_n}) \simeq f\otimes \Id_1\simeq \Phi_R(\Id_{P_n},f)$ for all closed morphisms $f\in\Endg(P_n)$.
\end{lemma}
\begin{proof}
Let $\phi:F(\Id,\iota):F(P_n,1_n)\simeq F(P_n,P_n)$ denote the homotopy equivalence from Lemma \ref{lemma-Pslide}, and let $\phi\inv$ denote a homotopy inverse.  Note that
\[
\phi\circ F(f,\Id_{P_n}) = F(\Id,\iota)\circ F(f,\Id) = F(f,\Id)\circ F(Id,\iota) = F(f,\Id_{1_n})\circ \phi
\]
for all $f\in \Endg(P_n)$. Conjugating with $\phi$ gives a homotopy equivalence
\[
\a:\Endg(F(P_n,P_n)) \buildrel\simeq \over\longrightarrow \Endg(F(P_n,1_n))
\]
which satisfies
\[
\a(F(f,\Id_{P_n})) := \phi\circ F(f,\Id_{P_n})\circ \phi\inv = F(f,\Id_{1_n})\circ \phi \circ \phi\inv \simeq F(f,\Id_{1_n})
\]
for all closed morphisms $f\in\Endg(P_n)$.

Now, by definition, $F(P_n,1_n) = (P_n\sqcup 1)\otimes X_n$, where $X_n = \llbracket \pic{crossingsBridge} \rrbracket$ with some orientations.  There is an inverse complex $X_n\inv$ with $X_n\otimes X_n\inv \simeq 1_{n+1}$.  Using this inverse complex we can define a homotopy equivalence
\[
\b:\Endg(F(P_n,1_n))\rightarrow \Endg(P_n\sqcup 1)
\]  
as follows: let $\psi:X_n\otimes X_n\inv \rightarrow 1_{n+1}$ denote a fixed homotopy equivalence and $\psi\inv$ a homotopy inverse.  Now, we may define $\b$ by commutativity of the following square, for each $g\in \Endg(F(P_n,1_n)) = \Endg((P_n\sqcup 1)\otimes X_n)$:
\[
\begin{diagram}[size=3em]
 (P_n\sqcup 1)\otimes X_n\otimes X_n\inv & \rTo^{g\otimes\Id_{X_n\inv}} & (P_n\sqcup 1)\otimes X_n\otimes X_n\inv  \\
\uTo^{\Id_{P_n\sqcup 1}\otimes \psi\inv} && \dTo_{\Id_{P_n\sqcup 1}\otimes \psi} \\
(P_n\sqcup 1)\otimes 1_n & \rTo^{\b(g)\otimes \Id_{1_n}} & (P_n\sqcup 1)\otimes 1_n
\end{diagram}
\]

It is easy to see that $\b(F(f,\Id_{1_n})) \simeq f\sqcup \Id_1$ for all closed morphisms $f\in \Endg(P_n)$.

Finally, there is an isomorphism $\gamma:\Endg(P_n\sqcup 1)\cong \Endg(P_n)\otimes \Endg(1_1)$ uniquely characterized by $\gamma\inv(f\otimes g)=f\sqcup g$.  In particular, $\gamma(f\sqcup \Id_1) = f\otimes \Id_1$.

The composition $\Phi_L:=\gamma\circ \b\circ \a$ is a homotopy equivalence
\[
\Phi_L:\Endg(F(P_n,P_n))\simeq \Endg(P_n)\otimes \Endg(1_1)
\]
such that $\Phi_L(f,\Id) \simeq f\otimes \Id_1$ for all closed morphisms $f\in\Endg(P_n)$.  This constructs $\Phi_L$ as in the statement.  A similar argument shows that there is a $\Phi_R$ as in the statement as well.  This completes the proof.
\end{proof}

\begin{proposition}\label{prop-Uslide}
For each $1\leq k\leq n$, we have $F(U_k,\Id_{P_n})\simeq \pm F(\Id_{P_n},U_k)$.  Graphically this is
\[
\left\llbracket
\begin{minipage}{.6in}
\labellist
\small
\pinlabel $U_k$ at 55 113
\pinlabel $\Id_{P_n}$ at 55 37
\endlabellist
\begin{center}\includegraphics[scale=.3]{emptyboxesWithCrossing}\end{center} 
\end{minipage}
\right\rrbracket
\simeq \pm
\left\llbracket
\begin{minipage}{.6in}
\labellist
\small
\pinlabel $\Id_{P_n}$ at 53 115
\pinlabel $U_k$ at 55 37
\endlabellist
\begin{center}\includegraphics[scale=.3]{emptyboxesWithCrossing}\end{center} 
\end{minipage}
\right\rrbracket
\]
Each is a closed element of $\Endg(F(P_n,P_n))$.
\end{proposition}
\begin{proof}
If $k=1$, the statement reduces to the well known fact that $\pic{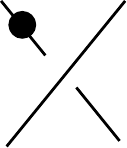} \simeq -\pic{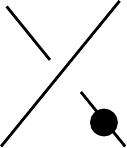}$.  So assume that $1<k\leq n$.  

Suppose $1< k\leq n$.  Consider the homotopy equivalences $\Phi_L,\Phi_R:\Endg(F(P_n,P_n))\rightarrow \Endg(P_n)\otimes \Endg(1_1)$ from Lemma \ref{lemma-UslideHom}.  We can compute the degree $(2-2k,2k)$ homology groups of these complexes as follows.  Recall that $\Endg(1_1)\cong \Z\oplus q^2\Z$ (generated by the identity and the dotted identity).  Then, by Corollary \ref{cor-extGroups}, the degree $(2-2k,2k)$ homology group of $\Endg(P_n)\otimes \Endg(1_1)$ is isomorphic to
\[
\Ext^{2-2k,2k}(P_n,P_n) \oplus \Ext^{2-2k,-2+2k}(P_n,P_n) \cong \Z\oplus 0
\]
generated by the class of $U_k\otimes \Id_1$.   Since an isomorphism must send generators to generators, it follows that the degree $(2-2k,2k)$ homology group of $\Endg(F(P_n,P_n))$ is isomorphic to $\Z$, generated by $[F(U_k,\Id_{P_n})]=\Phi_L\inv([U_k\otimes \Id_1])$.  This group is also generated by $[F(\Id_{P_n}, U_k)]=\Phi_R([U_k\otimes \Id_1])$.  Since any two generators of $\Z$ coincide up to sign, we must have $[F(\Id, U_k)] = \pm [F(U_k,\Id)]$.  This completes the proof.
\end{proof}

We are now ready to prove Proposition \ref{prop-slidingPoints}.
\begin{proof}[Proof of Proposition \ref{prop-slidingPoints}]
We will prove that $P_n(\mathbf{i})$ can be slid under crossings.  A similar argument will show that $P_n(\mathbf{i})$ can be slid over crossings.  We must show that $F(P_n(\mathbf{i}),1_n)\simeq F(1_n,P_n(\mathbf{i}))$ for each sequence $\mathbf{i}\in\{1,\ldots,n\}^r$.  If $\mathbf{i}$ is empty, then $P_n(\mathbf{i})=P_n$, and the result holds by Lemma \ref{lemma-Pslide} and the comments following.  So assume $\mathbf{i}=(i_1,\ldots,i_r)$ is non-empty.

By definition, $P_n(\mathbf{i}) = \Cone(U_{i_1})\otimes \cdots \otimes \Cone(U_{i_r})$, where $U_k$ is as in Definition \ref{def-polyDef}.  So it suffices to show that $F(\Cone(U_k),1_n)\simeq F(1_n,\Cone(U_k))$.

By Proposition \ref{prop-Uslide}, we have $F(U_k,\Id_{P_n})\simeq \pm F(\Id_{P_n},U_k)$.  Taking mapping cones, we have:
\[
F(\Cone(U_k),P_n) \cong F(P_n,\Cone(U_k))
\]
Since $\Cone(U_k)$ kills turnbacks, Lemma \ref{lemma-Pslide} says that the LHS above is equivalent to $F(\Cone(U_k),1_n)$, and the RHS is equivalent to $F(1_n,\Cone(U_k))$.  Therefore, $F(\Cone(U_k),1_n)$ is homotopy equivalent to $F(1_n,\Cone(U_k)$.  This completes the proof.
\end{proof}

\subsection{Connection to Cooper-Krushkal homology}
\label{subsec-CKhomology}
Let $\mathcal{K}=\{K_1,K_2,\ldots\}$ be a collection of quasi-projectors (Definition \ref{def-quasiProjectors}).  If $D$ is a colored, framed, oriented, marked link diagram, then $\llbracket D;\mathcal{K}\rrbracket$ is an object of $\Kom^-(0)$.  In order to  define homology, we must first apply the functor $\Homg(\emptyset,-)$, (called the tautological TQFT in \cite{B-N05}), which lands in the category of complexes of graded abelian groups.

\begin{definition}\label{def-linkHomology}
Suppose $L\subset \R^3$ is a colored, framed, oriented link, and let $\mathcal{K}$ be a collection of quasi-projectors.  Let $H_{\sl_2}(L;\mathcal{K})$ denote the homology of $\Homg(\emptyset, \llbracket D;\mathcal{K}\rrbracket)$, where $D$ is a diagram for $L$ with precisely one marked point on each component.  The link invariant $L\mapsto H_{\sl_2}(L;\mathcal{K})$ takes values in isomophism classes of bigraded abelian groups.
\end{definition}

\begin{observation}\label{obs-CKhomology}
In case $P=\{P_1,P_2,\ldots\}$ is the collection of Cooper-Krushkal projectors, the complex which computes $H_{\sl_2}(L;\mathcal{P})$ is dual to a complex which computes the homology theory defined in \cite{CK12a}, since our $P_n$ is dual to the projectors originally constructed by Cooper-Krushkal.  Thus, a universal coefficient theorem relates our homology with theirs.  Nonetheless, we continue to refer to our homology $H_{\sl_2}(L;\mathcal{P})$ as Cooper-Krushkal homology.
\end{observation}

The purpose of this section is to show that the action of $\Z[u_1,\ldots,u_n]$ on $P_n$ descends to an action on Cooper-Krushkal homology, and that this homology is finitely generated over a tensor product of such rings.  That is to say, we show that Cooper-Krushkal can be lifted to an invariant taking values in isomorphism classes of finitely generated modules over a polynomial ring.

Two marked diagrams represent isotopic links if they are related by a sequence of Reidemeister moves (away from the crossings), merging and splitting of marked points, and sliding marked points past crossings.  We must show that the action of $\Z[u_1,\ldots,u_n]$ commutes with each of these equivalences.

\begin{proposition}\label{prop-UnEquivariance}
Let $U_k:t^{2-2k}q^{2k}P_n\rightarrow P_n$ be as in Definition \ref{def-polyDef}. The following equivalences are unique up to homotopy and sign, and each commutes with the action of $U_k$ up to homotopy and sign:
\[
\pic{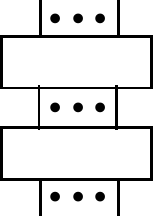}\buildrel(1)\over \simeq \pic{projector} \hskip.5in \wwpic{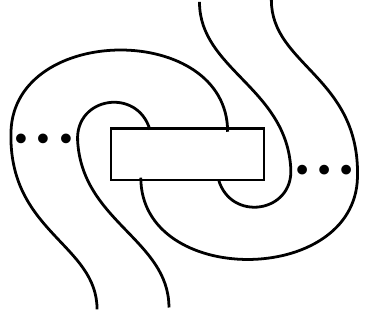} \buildrel(2)\over\simeq \pic{projector}\hskip.5in  \wpic{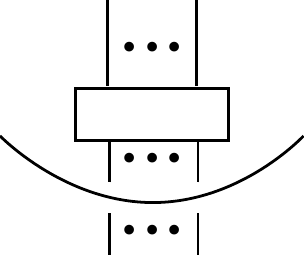}    \buildrel(3)\over \simeq    \pic{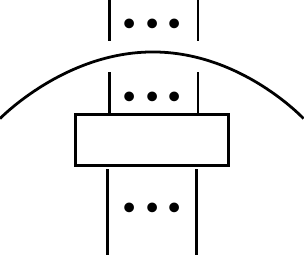} \hskip.5in \wpic{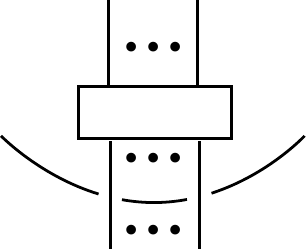}  \buildrel(4)\over \simeq    \wpic{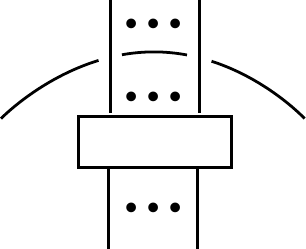} 
\]
where $\pic{projector}=P_n$.
\end{proposition}
\begin{proof}
For the first statement one can show that the homology groups $\Homg^{0,0}$ are all isomorphic to $\Z$ using Theorem \hyperref[thm-partialTraceHom]{\ref{thm-partialTraceHom}} and the fact that $\Ext^{0,0}(P_n,P_n)\cong \Z$ (Corollary \ref{cor-extGroups}).  This implies that each homotopy equivalence is unique up to homotopy and sign.

Proof of (1).  Consider the map $\Psi:\Endg(P_n\otimes P_n)\rightarrow \Endg(P_n)$ given by conjugating by a canonical equivalence $\mu:P_n\otimes P_n\rightarrow P_n$ .  By Lemma \ref{lemma-mapMerging}, we have $\Psi(f\otimes \Id)\simeq f$ for all closed morphisms $f\in\Endg(P_n)$.   In particular $\mu\circ (f\otimes \Id) \simeq f\circ \mu$ for any closed morphism $f\in\Endg(P_n)$.  A similar argument shows that $\mu\circ(\Id\otimes f)\simeq f$ for all closed morphisms $f\in\Endg(P_n)$.  Taking $f=U_k$ gives the result for equivalence (1).

Proof of (2).  Let $r(-):\bn{n}\rightarrow \bn{n}$ denote the covariant functor which rotates diagrams by 180 degrees.  Then we have an equivalence $\phi:r(P_n)\simeq P_n$ by uniqueness of Cooper-Krushkal projectors.  Conjugating by this equivalence gives an isomorphism $\Ext(r(P_n))\simeq \Ext(P_n)$.  Note that $U_k$ and $\phi\circ r(U_k)\circ \phi\inv$ each represent generators of an $\Ext$ group which is isomorphic to $\Z$, by Corollary \ref{cor-extGroups}.  It must be that $\phi\circ r(U_k)\circ \phi\inv\simeq \pm U_k$, hence $\phi\circ r(U_k)\simeq \pm U_k\circ \phi$.  That is to say, $\phi$ commutes with the action of $U_k$ up to homotopy and sign.

Proof of (3). We have the following equivalences
\[
\wpic{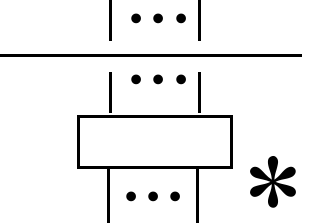} \ \simeq \  \wpic{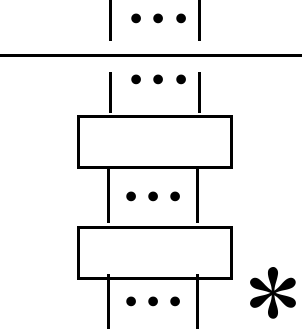} \ \simeq \  \wpic{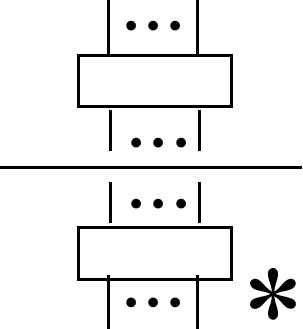}\ = \  \wpic{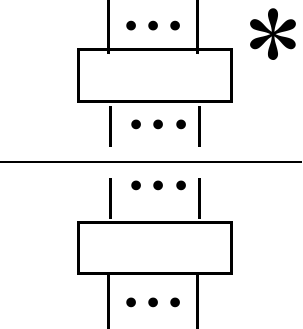}\ \simeq \  \wpic{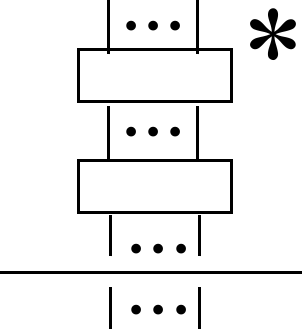} \ \simeq \  \wpic{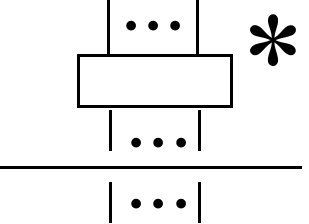} 
\]
Each complex above is endowed with a chain map $U_k$ acting on the projector marked with a $\ast$.  The first and last equivalences commute with the maps $U_k$ up to homotopy and sign by statement (1) of the theorem.   The second and fourth maps honestly commute with $U_k$ since distant maps commute.  The map in the middle commutes with $U_k$ up to homotopy and sign by Proposition \ref{prop-Uslide}.   Thus the composition of these equivalences commutes with $U_k$ up to homotopy and sign.  Proof of (4) is similar.
\end{proof}

We now have our main result on Cooper-Krushkal homology:
\begin{theorem}\label{thm-CKhomology}
Let $L\subset \R^3$ be a framed, oriented link whose components are colored $n_1,\ldots,n_r$.  Let $R(L)$ denote the tensor product
\[
R(L):=\bigotimes_{i=1}^r\Z[u_1,\ldots,u_{n_i}]
\]
graded so that $\deg(u_k)=(2-2k,2k)$.  Then the Cooper-Krushkal homology $H_{\sl_2}(L;\mathcal{P})$ is a well-defined isomorphism class of finitely generated, bigraded $R$-modules.
\end{theorem}
\begin{proof}
We will actually show that the Cooper-Krushkal homology is finitely generated over a tensor product $R$ of rings $\Z[u_2,\ldots,u_n]$, as $n$ ranges over the set of colors with multiplicity.

Let $\mathcal{P}=\{P_1,P_2,\ldots\}$ and $\mathcal{K}=\{K_1,K_2,\ldots\}$ denote the collection of projectors, respectively bounded complexes, from Theorem \ref{thm-standardCx}.  By definition, $\llbracket D;\mathcal{P}\rrbracket$ is the planar composition of a bounded complex (Khovanov complex of a tangle), together with the projectors $P_{i_1},\ldots,P_{i_r}$.  Multilinearity of planar composition, together with the construction of $P_n$ in terms of $K_n$ implies that:
\[
\llbracket D;\mathcal{P}\rrbracket \cong R\otimes\llbracket D ;\mathcal{K}\rrbracket \hskip.4in\text{with differential}\hskip.4in 1\otimes d+\sum_f f\otimes \partial_f
\]
where the sum is over non-constant monomials $f\in R$ and the $\partial_f\in \Endg(\llbracket D;\mathcal{P}\rrbracket)$ are some elements, all but finitely many of which are equal to zero.  Applying the functor $\Homg(\emptyset,-)$ gives
\[
C_{\mathcal{P}}(D) = R\otimes_\Z C_{\mathcal{K}}(D)
\]
with some $R$-equivariant differential.  We have abbreviated $C_{\mathcal{P}}(D):=\Homg(\emptyset,\llbracket D;\mathcal{P}\rrbracket)$ and similarly for $C_{\mathcal{K}}(D)$.  Since each $K_n$ is bounded, it follows that $C_\mathcal{K}(D)$ is a finitely generated $R$-module.  Since polynomial rings are Noetherian, this implies that the homology of $C_{\mathcal{P}}(D)$ is finitely generated over $R$ also.    Thus, $H_{\sl_2}(D;\mathcal{P})$ is finitely generated over $R$.

It remains to show that different diagrams $D$ produce isomorphic bigraded $R$-modules $H_{\sl_2}(D;\mathcal{P})$.  To see this, first note that Reidemeister equivalences occur away from the crossings, and induce isomorphisms in homology which commute with the $R$-action.

The only other equivalence which we must account for is the sliding of marked points over or under crossings.  To this end, suppose $D$ differs from $D'$ by an equivalence from Proposition \ref{prop-slidingPoints}.  Then the homotopy equivalence $C_\mathcal{P}(D)\simeq C_{\mathcal{P}}(D')$ commutes with the $R$ actions up to homotopy and sign, by Proposition \ref{prop-UnEquivariance}.  Thus, $H_{\sl_2}(D;\mathcal{P})$ is well-defined up to isomorphism of $R$-modules, and also replacing some of the polynomial generators by their negatives.

At this point, we appeal to the fact that $H_{\sl_2}(D;\mathcal{P})$ is the homology of a complex which is free as an $R$-module.  For a free $\Z[x]$-module $M$, one can construct an explicit $\Z$-module isomorphism $\phi:M\rightarrow M$ such that $\phi(xm)=-x\phi(m)$.  In the same way, any two $R$-module structures on $C_\mathcal{P}(D)$ which differ by a sign on some generators are isomorphic.  Thus, $H_{\sl_2}(D;\mathcal{P})$ does not depend on the choice of $D$ up to isomorphism of bigraded $R$-modules.  This completes the proof.
\end{proof}

\begin{remark}
Obviously if some $U_k\in \Endg(P_n)$ is nilpotent up to homotopy, then the generator $u_k$ can be omitted from the ring $\Z[u_1,\ldots,u_n]$, while retaining finite generation.  For example, $U_1^2=0$, so this variable is not needed for finite generation.  Also, Example \ref{example-Q3} can be used to show that $U_2$ acting on $P_3$ is nilpotent up to homotopy, so this variable can be omitted.  See Conjecture \ref{conj-nilpBdd} for what we expect to be true in general.
\end{remark}

\begin{theorem}\label{thm-spectralSequence}
Fix a family $\mathcal{K}=\{K_1,K_2,\ldots,\}$ of quasi-projectors, and let $L\subset S^3$ be a colored, framed, oriented link.  There is a polynomial ring $R(L)$ and a spectral sequence of bigraded $R(L)$-modules $R(L)\otimes_\Z H_{\sl_2}(L;\mathcal{K})\Rightarrow H_{\sl_2}(L;\mathcal{P})$.
\end{theorem}
\begin{proof}
Let $D$ be a diagram representing $L$, and assume $D$ is marked with exactly one marked point on each component, away from the crossings.  For each $n$, write $K_n=P_n(i_1,\ldots,i_r)$ with $1\leq i_1<\ldots<i_r\leq n$, and define rings $R_n=\Z[u_1,\ldots,u_r]$.  Put $R=R_{n_1}\otimes \cdots\otimes R_{n_k}$, where $n_1,\ldots, n_k$ are the colors appearing on the components of $L$.

From the proof of Theorem \ref{thm-CKhomology}, we see that there exist chain complexes $C_\mathcal{K}(D)$ and $C_\mathcal{P}(D)$ of abelian groups such that
\begin{enumerate}
\item The homology of $C_\mathcal{P}(D)$ and $C_\mathcal{K}(D)$ are $H_{\sl_2}(L;\mathcal{P})$ and $H_{\sl_2}(L;\mathcal{K})$, respectively.
\item $C_\mathcal{P}(D)= R\otimes C_{\mathcal{K}}(D)$ with an $R$-equivariant differential of the form $1\otimes d +\sum_f f\otimes \partial_f$, where the sum is over nonconstant monomials $f\in R$.
\end{enumerate} 
Regard $R$ as being filtered by degree of polynomials.  Then equip $R\otimes C_\mathcal{K}(D)$ with the tensor product filtration.  The above says that $C_\mathcal{P}(D)$ is equal to $R\otimes C_\mathcal{K}(D)$ with an $R$-equivariant, filtered differential whose filtration preserving part is precisely $1\otimes d$.  Standard arguments produce a spectral sequence with $E_2$ page $R\otimes H_{\sl_2}(L;\mathcal{K})$ and $E_\infty$ page the associated graded of $H_{\sl_2}(L;\mathcal{P})$.
\end{proof}

\section{Appendix: Convolutions and deformaton retracts}
\label{sec-homAlg}

A category $\mathscr{A}$ is called \emph{$\Z$-linear} if (1) the morphism spaces are abelian groups and (2) composition is bilinear.  A $\Z$-linear category $\mathscr{A}$ is called \emph{additive} if, in addition, (3) $\mathscr{A}$ is closed under finite direct sums (equivalently direct products).  For a $\Z$-linear category, let $\Kom(\mathscr{A})$ denote the category of possibly unbounded chain complexes over $\mathscr{A}$ with differentials of degree $+1$, with morphisms given by degree zero chain maps.

Let $t:\Kom(\mathscr{A})\rightarrow \Kom(\mathscr{A})$ denote the upward shift functor satisfying, $(tC)_k = C_{k-1}$ with differential $-d_C$.

\subsection{Convolutions}
\label{subsec-convolutions}
Suppose $A,B$ are chain complexes over an additive category $\mathscr{A}$, and $f:B\rightarrow A$ is a chain map.  The mapping cone on $f$ is the chain complex $\Cone(f)_i = B_{i+1}\oplus A_i$ with differential
\[
d_{\Cone(f)} = \matrix{-d_B & 0 \\ f & d_A}
\]
Note that, as graded objects $\Cone(f) = t\inv B\oplus A$.  Now, suppose that $g:C\rightarrow B$ is a chain map.  This $g$ extends to a chain map $\~g:t\inv C\rightarrow \Cone(f)$ if and only if $f\circ g\simeq 0$, in which case
\[
\Cone(\~g) = t^{-2}C\oplus t\inv B\oplus A \text{ with differential }
\left[\begin{tabular}{c: c c}
$d_C$ & 0 & 0 \\ \hdashline
$g$ & $-d_B$ & 0 \\
$-h$ & $f$ & $d_A$
\end{tabular}
\right]
\]
where $-h$ is the corresponding component of $\~g$.  Iterated mapping cones of this sort are called \emph{convolutions}, and can be desribed as follows:

\begin{definition}\label{def-convolutions2}
Let $E_i$ be chain complexes over an additive category and $\a_i:E_i\rightarrow E_{i+1}$ chain maps such that $\a_{i+1}\circ \a_i\simeq 0$ for all $i\in \Z$.  Any such sequence will be called a \emph{homotopy chain complex}, and will be denoted as
\begin{equation}\label{eq-Ebullet}
E_\bullet = \cdots \buildrel\a_{i-1}\over \longrightarrow E_i \buildrel\a_{i}\over \longrightarrow E_{i+1} \buildrel \a_{i+1}\over \longrightarrow\cdots
\end{equation}
A convolution of a homotopy chain complex $E_\bullet$ is any chain complex which, as a graded object equals $\bigoplus_{i\in \Z} t^i E_i$ and whose differential $d$ satisfies the following conditions: if $d_{ij}\in\Homg^{1-i+j}(E_j,E_i)$ is the corresponding component of $d$, then
\begin{itemize}
\item $d_{ii}=(-1)^{i}d_{E_i}$.
\item $d_{i+1,i}=\a_i$.
\item $d_{ij}=0$ for $i<j$.
\end{itemize}
We will denote a convolution using parenthesized notation in which we write all of the degree shifts explicitly: 
\[
M = (\cdots \buildrel\a_{i-1}\over \longrightarrow t^i E_i \buildrel\a_{i}\over \longrightarrow t^{i+1}E_{i+1} \buildrel \a_{i+1}\over \longrightarrow\cdots )
\]
We have a dual notion, in which $\bigoplus$ is replaced by $\prod$ in the above definition.  We will refer to convolutions using $\bigoplus$ (respectively $\prod$) as being of type I (respectively type II).
\end{definition}

\begin{example}
A bicomplex is a special case of homotopy complex, in which the ``horizontal'' chain maps $\a_i:E_i\rightarrow E_{i+1}$ satisfy $\a_{i+1}\circ \a_i = 0$.  The total complex of $E_\bullet$ is an example of a convolution, and exists precisely when the direct sum $\bigoplus_{i\in \Z} t^iE_i$ exists in $\Kom(\mathscr{A})$.
\end{example}
Even though most homotopy complexes are not bicomplexes, we will continue to use the suggestive notation $\Tot(E_\bullet)$ to denote a convolution of $E_\bullet$.  We warn the reader that there may exist many convolutions of a given $E_\bullet$, or none at all.  If the infinite direct sum $M:=\bigoplus_{i\in\Z}t^iE_i$ exists in $\Kom(\mathscr{A})$, then a convolution $\Tot(E_\bullet)$ exists if and only if all  of the Massey products \cite{May69} $\langle [\a_{i+r}], [\a_{i+r-1}],\ldots, [\a_i]\rangle $ vanish, where $\a_i$ is regarded as a degree 1 cycle in the differential graded algebra $\Endg(M)$.
 
The notion of convolution is standard in the theory of triangulated categories, where the term is used more generally to mean the following: if $E_\bullet\in\Kom(\mathscr{T})$ is a chain complex over a triangulated category, then a convolution $\Tot(E_\bullet)\in \mathscr{T}$ is an iterated mapping cone, which represents a ``flattening'' of $E_\bullet$; see \cite{Kap88}.

We give a special name to the ``convolution degree'' of a map $\Tot(E_\bullet)\rightarrow \Tot(F_\bullet)$:
\begin{definition}\label{def-morphismLength}
Suppose $M = \Tot(E_\bullet)$ and $N =\Tot(F_\bullet)$ are convolutions.  Say that an element $f\in \Homg_\mathscr{A}(M,N)$ of $M$ has \emph{length $k$} if the component
\[
f_{ij}\in \Homg_\mathscr{A}(E_j,F_i)
\]
vanishes unless $i-j = k$.
\end{definition}
We can write any element $f\in\Homg_\mathscr{A}(M,N)$ in terms of its length $k$ components, $f=\sum_{k\in\Z} f_k$, where $f_k := (f_{i+k,i})_i\in \prod_i\Homg(E_i,F_{i+k})$ is regarded as an element of $\Homg_\mathscr{A}(M,N)$ of length $k$.  Let us say that $f$ is a map of convolutions if $f_k=0$ for $k<0$.  Suppose $F_\bullet$ is bounded from above, i.e.~$F_i=0$ for $i\gg 0$, and $f_k\in \Homg(M,N)$ are any elements of length $k$, each of some fixed homological degree $r$.  Then any infinite sum $f_0+f_1+\cdots $ is finite on restriction to each $E_j$, hence is a well-defined element of $\Homg_\mathscr{A}(M,N)$ by the universal property of direct sums.

Moreover, length is additive under composition of morphisms, so that if $f= f_0+f_1+\cdots $, $g=g_0+g_1+\cdots$, and $f\circ g = (f\circ g)_0+(f\circ g)_1+\cdots$ are written in terms of length $k$ components, then $(f\circ g)_k = \sum_{i+j=k}f_i\circ g_j$.  We have proven:
\begin{lemma}\label{lemma-mapsFromSeries}
Let $M$ and $N$ be convolutions which are bounded above, fix $r\in\Z$, and suppose we have elements $f_k\in \Homg^r(M,N)$ of length $k$, for each $k\in \Z_{\geq 0}$.  Then the series $f = f_0+f_1+\cdots$ is a well defined element of $\Homg^r(M,N)$.  In particular, if $\a\in \Endg^0(M)$ has length $k>0$, then $\Id_M-\a$ and $\Id + \a +\a^2 +\cdots$ are mutual inverses.\qed
\end{lemma}
If $E_\bullet$ is a homotopy chain complex as in equation (\ref{eq-Ebullet}), then the differential on a convolution $M = \Tot(E_\bullet)$ can be written in terms of its length $k$ components as
\[
d_M = \Delta_0 + \Delta_1+\cdots
\]
where $\Delta_k\in \Endg^1(M)$ has length $k$.  In particular $\Delta_0|_{E_i} = (-1)^id_{E_i}$ and $\Delta_1|_{E_i} = \a_i$.  

%The following proposition says that we can perturb the length $k\geq 1$ component of the differential of a convolution up to homotopy, at the expense of introducing higher length components.  In particular, if $E_\bullet$ and $F_\bullet$ represent the same object of $\Kom(\Kom(\mathscr{A})_{/h})$, then any convolution of $E_\bullet$ is isomorphic to a convolution of $F_\bullet$ and vice versa.

Below, we will absorb the explict grading shifts into the complexes $E_i$, so that our convolutions will be $\bigoplus_i E_i$ with lower triangular differential, rather than $\bigoplus_i t^iE_i$.

\begin{theorem}[Perturbing the differentials]\label{thm-diffPerturbing}
Suppose we are given $E_i\in \Kom(\mathscr{A})$, $E_i=0$ for $i\gg 0$, and cycles $\a_i\in \Homg^1(E_i, E_{i+1})$ such that $\a_{i+1}\circ \a_{i}\simeq 0$ for all $i$.  Suppose $M = (\cdots \buildrel\a_{i-1}\over \longrightarrow E_i \buildrel\a_{i}\over \longrightarrow E_{i+1} \buildrel \a_{i+1}\over \longrightarrow\cdots )$ is a convolution of the corresponding homotopy chain complex.  Fix an integer $k\geq 1$ and assume there are elements $\phi_i\in\Homg^1(E_i,E_{i+k})$ such that
\[
d_{i+k,i}-\phi_i\simeq 0
\]
for all $i\in \Z$, where $d_{i+k,i}\in\Homg^{1}(E_i,E_{i+k})$ is the component of $d_M$.  Then up to isomorphism of convolutions, each $d_{i+k,i}$ can be replaced by $\phi_i$ at the expense of affecting only the length $>k$ components of $d_M$.
\end{theorem}
\begin{proof}
Fix an integer $k\geq 1$; we will perturb the length $k$ component of $d_M$.  Write the differential on $M$ as $\Delta = \Delta_0 + \Delta_1 +\Delta_2+ \cdots$ in terms of its length $l$ components, so that in particular $\Delta_0|_{E_i} = d_{E_i}$ and $\Delta_1|_{E_i} = \a_i$.   Now, fix an element $H \in \Endg^0(M)$ of length $k$.   By Lemma \ref{lemma-mapsFromSeries}, the infinite sum $\Id_M+H+ H^2+\cdots $ is a well defined element of $\Endg^0(C)$, and is a two sided inverse for $(\Id_M-H)$.  Conjugating the differential $\Delta$ by $(\Id_M-H)$ gives
\begin{equation}\label{eq-diffConjing}
 \Delta':=(\Id_M-H) \circ (\Delta_0 + \Delta_1 + \Delta_2 +\cdots )\circ(\Id_M +H+H^2 +\cdots) \ =\  d_0 + \Delta_1' + \Delta_2'+ \cdots
\end{equation}
Recall that length is additive under function composition and $H$ has length $k$, so $\Delta_l' = \Delta_l$ for $0\leq l<k$ and $\Delta_k' = \Delta_k - \Delta_0 H + H \Delta_0$.   This is to say, a perturbation of the length $k$ part of $\Delta$ up to homotopy is realized by the isomorphism $(\Id_M - H):(M, \Delta)\buildrel \cong \over\rightarrow (M,\Delta')$ of convolutions, where the length $l$ components of $\Delta$ and $\Delta'$ agree for $0\leq l<k$. 
\end{proof}

\subsection{Deformation retracts}
\label{subsec-sdr}
Here we recall the standard notion of (strong) deformation retracts, which are a particular nice class of chain homotopy equivalences which interact nicely with convolutions.
\begin{figure}[ht]
\[
\begin{tikzpicture}
\tikzstyle{every node}=[font=\small]
\node (a) at (0,0) {$M$};
\node (b) at (3,0) {$N$};
\path[->,>=stealth',shorten >=1pt,auto,node distance=1.8cm,
  thick]
(a) edge [loop left, looseness = 5,in=145,out=215] node[left] {$h$}	(a)
([yshift=3pt] a.east) edge node[above] {$\pi$}		([yshift=3pt] b.west)
([yshift=-2pt] b.west) edge node[below] {$\sigma$}		([yshift=-2pt] a.east);
\end{tikzpicture}
\]
\caption{The data $(\pi,\sigma,h)$ of a deformation retract $M\rightarrow N$}
\end{figure}

\begin{definition}\label{def-sdr}
Let $\mathscr{A}$ be a $\Z$-linear category, and $M,N$ chain complexes over $\mathscr{A}$.  A chain map $\pi:M\rightarrow N$ is called a \emph{deformation retract} if there exist a chain map $\sigma:N\rightarrow M$ and a homotopy $h\in \Endg^{-1}_\mathscr{A}(M)$ such that
\begin{itemize}
\item $h\circ \sigma = \pi\circ h = 0$.
\item $\pi\circ \sigma = \Id_N$.
\item $\Id_M - \sigma\circ \pi = d_M\circ h+h\circ d_M$.
\end{itemize}
In this case we say $(\pi,\sigma,h)$ give the data of the deformation retract.
\end{definition}

\begin{lemma}\label{lemma-sdr}
Suppose $A,B\in\Kom(\mathscr{A})$ and $(\pi,\sigma,h)$ give the data of a strong deformation retract $A\rightarrow B$.  Then:
\begin{enumerate}
\item $\Id_A = \sigma \circ  \pi + d_A\circ h+h\circ d_A$ is a decomposition of $\Id_A$ into mutually orthogonal idempotents.
\item $h$ may be assumed to satisfy $h^2=0$.
\end{enumerate}
\end{lemma}
\begin{proof}
The proof of (1) is straightforward.  For part (2), put $h'=hdh$.  Then
\begin{itemize}
\item $(h')^2 = (hdh)(hdh) = h(dh)(hd)h=0$ since $hd$ and $dh$ are orthogonal, and
\item $dh'+h'd = d(hdh)+(hdh)d = (dh)^2+(hd)^2 = dh+hd = \Id_A - \sigma\pi$ since $hd$ and $dh$ are idempotent.
\end{itemize}
Thus, $h'$ has the desired properties.
%Put $e_1 = ir$, $e_2 = dh$, $e_3=hd$.  By hypothesis we have
\end{proof}

\begin{proposition}[Gaussian elimination]\label{prop-gauss}
Suppose we have graded objects $A = (A^k)_{k\in\Z}$ and $B = (B^k)_{k\in\Z}$ over an additive category $\mathscr{A}$, and suppose $C = A\oplus B$ is a chain complex with differential $\smMatrix{{}_Ad_A & {}_A d_B \\ {}_B d_A & {}_B d_B})$.  
Suppose also that ${}_Bd_B^2=0$.  If $(B,{}_Bd_B)$ is a contractible chain complex, then there is a deformation retract $C\rightarrow A'$ where $ A'= A$ with differential $d_A' = {}_A d_A - {}_A d_B\circ h \circ {}_B d_A$, where $h$ is a nulhomotopy for $B$ which satisfies $h^2=0$.
\end{proposition}

\begin{proof}
By hypotheses there is some nulhomotopy $h\in\Endg^1(B)$, and by Lemma \ref{lemma-sdr} we may assume $h^2=0$.  The relevant maps are defined in the following diagram:
\[
\begin{tikzpicture}
\tikzstyle{every node}=[font=\small]
\node at (0,.35) {$A$};
\node at (0,-.35) {$B$};
\node (a) at (0,0) {$\oplus $};
\node (b) at (4,0) {$A$};
\path[->,>=stealth',shorten >=1pt,auto,node distance=1.8cm,
  thick]
(a) edge [loop left, looseness = 5,in=145,out=215] node[left] {$H=\matrix{0 & 0 \\ 0 & h}$}	(a)
([yshift=3pt] a.east) edge node[above] {$\pi = \matrix{\Id&-{}_A d_B\circ h}$}		([yshift=3pt] b.west)
([yshift=-2pt] b.west) edge node[below] {$\sigma = \matrix{\Id\\-h\circ {}_B d_A}$}		([yshift=-2pt] a.east);
\end{tikzpicture}
\]
It is straightforward to check that (1) $\pi$ and $\sigma$ are chain maps, (2) $\pi\circ \sigma = \Id_A$, (3) $\Id_{A\oplus B}= \sigma\circ \pi + \smMatrix{{}_Ad_A & {}_A d_B \\ {}_B d_A & {}_B d_B} H+ H\smMatrix{{}_Ad_A & {}_A d_B \\ {}_B d_A & {}_B d_B}$, (3) $\pi\circ H =0$, (4) $H\circ \sigma =0$.  That is, $(\pi,\sigma,H)$ give the data of a strong deformation retract.
\end{proof}

We want to see how convolutions interact with deformation retracts.  It is useful to first make a definition:
\begin{definition}\label{def-assGrad}
If $M$ is a convolution with differential $d$ then the \emph{associated graded} complex is the complex $\gr(M)=(M,d_0)$ where $d_0$ is the length zero part of $d$ (Definition \ref{def-morphismLength}).  In other words, the associated graded of $M = (\cdots \rightarrow E_i\rightarrow E_{i+1}\rightarrow \cdots)$ is $\bigoplus_{i\in\Z} E_i$.
\end{definition}
Suppose we are given the data $(\pi,\sigma, h)$ of a deformation retract $M\rightarrow N$, and that $\pi,\sigma,h$ respect the convolution filtration.  Then the length zero parts $(\pi_0,\sigma_0,h_0)$ give the data of a deformation retract $\gr(M)\rightarrow \gr(N)$.  Under a certain finiteness assumption on $M$, the converse also holds:
\begin{theorem}\label{convSdrthm}
Suppose $M$ is a convolution with $\gr(M)=\bigoplus_i E_i$, and that each $E_i$ deformation retracts onto some $F_i$.  If $E_i=0$ for $i\gg 0$, then there is a convolution $N$ with $\gr(N)=\bigoplus_i F_i$ onto which $M$ deformation retracts.  Further:
\begin{enumerate}
\item The length 1 part of the differential on $N$ is induced by the composition $F_i\simeq E_i\buildrel\a_i\over\longrightarrow E_{i+1}\simeq F_{i+1}$, where $\a_i$ is the component of $d_M$.
\item For each integer $k$, the length $k$ components of the data $(\pi,\sigma,h)$ of the retract $M\rightarrow N$ are zero for $k<0$, and in general are polynomials in the $\pi_0,\sigma_0,h_0$, and the components $d_l$ of $d_M$.  These polynomials are universal in the sense that they do not depend on any of the initial data.
\end{enumerate}
\end{theorem}
Before proving, we note that there are various extensions of this theorem:

\begin{remark}\label{remark-generalizedConv}
One could allow generalized convolutions, in which the terms are indexed by any partially ordered set $X$, rather than $\Z_{\leq 0}$.  If each $x\in X$ has finitely many descendants, then the theorem holds as is.  If instead each $x\in X$ has finitely many ancestors, then the theorem holds provided we replace $\bigoplus$ with $\prod$ in the definition of convolution.  
Here a \emph{descendant} (resp.~\emph{ancestor}) of $x\in X$ is a $y\in X$ such that $x<y$ (resp.~$y<x$).
\end{remark}
\begin{remark}\label{remark-sdrPreservesStruct}
Informally speaking if $\pi_0,\sigma_0,h_0$, and $d_M$ preserve some additional structure on $M$, then $N$ inherits a similar structure, and $\pi,\sigma,h, d_N$ preserve this structure.  This is because the components of the latter morhisms are certain polynomials evaluated on the components of the former.  For example, suppose $M$ is a dg $A$-module for some dg-algebra $A$, and suppose $\pi_0,\sigma_0,h_0$ commute with the $A$-action.  Then $N$ is a dg $A$-module, and the maps $\pi,\sigma,h$ are $A$-equivariant.
\end{remark}

\begin{proof}
Let $M, E_k$, $F_k$ be as in the hypotheses, and put $N= \bigoplus_k F_k$.  We can write the differential $d_M$ in terms of its length $k$ components as $d_M = d_0+d_1+\cdots $.  By hypothesis we have length zero maps $\pi_0  \in \Homg^0(M,N)$, $\sigma_0  \in \Homg^0(N,M)$, and $h_0\in \Homg^{-1}(M,M)$ such that
\begin{itemize}
\item[(i)] $\Id_N=\pi_0\circ \sigma_0$.
\item[(ii)] $\Id_M-\sigma_0\circ \pi_0 = d_0\circ h_0 +h_0\circ d_0$
\item[(iii)] $\pi_0\circ h_0 = 0$.
\item[(iv)] $h_0\circ \sigma_0=0$.
\end{itemize}
Put $e: = \sigma_0\circ \pi_0$, and consider following statement, where $k\in\N\cup\{\infty\}$:
\vskip8pt
\begin{hypk}
There exist elements $\a_l\in \Endg^0(M)$ of length $l$, for $1\leq l < k$ such that (a) each $\a_l$ is a polynomial in $h_0, e, d_0,d_1,d_2,\ldots$, and (b) the length $l$ components of
\begin{equation}\label{eq-conjugatedDiff}
\Delta:=(\Id_M + \a_k)\circ \cdots \circ (\Id_M+\a_1) \circ d_M\circ (\Id_M+\a_1)\inv \circ \cdots \circ (\Id_M+\a_k)\inv %d_0+\Delta_1+\Delta_2 + \cdots
\end{equation}
satisfy $\Delta_0 = d_0$ and $\Delta_l=e\circ  \Delta_l\circ e$ for all $1\leq l< k$.
\end{hypk}
Let us assume that \textbf{Hyp}$(\infty)$ holds.  Then define $\Phi$ to be the infinite composition $\Phi: = \cdots \circ (\Id_M+\a_2)\circ (\Id_M+\a_1)$, which is a well defined series $\Phi = \Id_M +\Phi_1+\Phi_2+\cdots$.  By Lemma \ref{lemma-mapsFromSeries} $\Phi$ and $\Phi\inv$ are well defined elements of $\Endg^0(M)$.   Put
\[
\pi:=\pi_0\circ \Phi, \hskip.3in \sigma:=\Phi\inv \circ \sigma_0, \hskip.3in h=\Phi\inv \circ h_0\circ \Phi,  \hskip.3in d_N = \pi_0\circ \Phi\circ d_M \circ \Phi\inv \circ \sigma_0
\]
An elementary calculation shows that $(\pi,\sigma,h)$ give the data of a deformation retract $(M,d_M)\rightarrow (N, d_N)$.  All of the computations are immediate except for the following:
\begin{eqnarray*}
d_M\circ h+ h \circ d_M
&=& d_M\circ (\Phi\inv \circ h_0 \circ\Phi ) + (\Phi\inv\circ h_0\circ \Phi)\circ d_M\\
&=&\Phi\inv\circ(\Delta \circ h_0 + h_0 \circ\Delta) \circ\Phi \\
&=& \Phi\inv\circ( d_0\circ h_0 + h_0\circ d_0)\circ\Phi\\
&=& \Phi\inv\circ (\Id_M - \sigma_0\circ \pi_0)\circ \Phi\\
&=& \Id_M - \sigma \circ \pi
\end{eqnarray*}
The first, second, and fifth equalities follow from the definitions. The fourth holds since $(\pi_0,\sigma_0,h_0)$ are the data of a deformation retract $(M,d_0)\rightarrow (N,\pi_0\circ d_0\circ \sigma_0)$.  Let us convince ourselves that the third equality holds.  By statement (b) of \textbf{Hyp}$(\infty)$ the conjugated differential $\Delta:=\Phi\circ d_M\circ \Phi\inv$ satisfies $\Delta = d_0 +\Delta_1+\Delta_2+\cdots $ with $\Delta_k = e\circ \Delta_k\circ e$ for all $k\geq 1$.  Since $h_0\circ e =0 = e\circ h_0$, we have $h_0\circ \Delta_k = 0 = \Delta_k\circ h_0$ for all $k\geq 1$.   The length zero part of $\Delta$ is $\Delta_0=d_0$, hence $h_0\circ \Delta = h_0\circ d_0$, and $\Delta\circ h_0 = d_0\circ h_0$.  The third equality above follows, and we have a deformation retract, as claimed.

Note also that the length $k$ components of $\pi,\sigma,h$, and $d_N$ are polynomials in $\pi_0,\sigma_0,h_0$ and the $d_k$ since the same is true of $\Id_M + \a_k$ and $(\Id_M + \a_k)\inv$.  Thus, we have proven the theorem, assuming that \textbf{Hyp}$(\infty)$ holds.  The remainder of the proof is taken care of by the following lemma.
\end{proof}

\begin{lemma}
The statement \textbf{Hyp}$(\infty)$ holds.
\end{lemma}
\begin{proof}
We will construct by induction on $k\geq 1$ a stable family of elements $\{\a_1,\ldots,\a_{k-1}\}$ for which \textbf{Hyp}$(k)$ holds.  The base case $k=1$ is vacuous.  Assume by induction that $\{\a_1,\ldots,\a_{k-1}\}$ satisfy \textbf{Hyp}$(k)$, and define $\Delta:=d_0+\Delta_1 + \Delta_2+ \ldots$ to be the differential $d_M$ conjugated by $\prod_{k> l\geq 1}(\Id_M+\a_l)$ as in equation (\ref{eq-conjugatedDiff}).  By the induction hypothesis, the $\a_l$ are polynomial in the $h_0,e,d_0,d_1,\ldots$, hence so are the $\Delta_l$.

Taking the length $k$ part of the equation $\Delta^2=0$ gives
\begin{equation}\label{eq-deltaCommutator}
d_0\circ \Delta_{k} +\Delta_{k}\circ d_0 = - \sum_{i,j}\Delta_i\circ \Delta_j
\end{equation}
where the sum on the right-hand side is over $1\leq i,j< k$ such that $i+j=k$.  Now, by the induction hypothesis we have $\Delta_l = e\Delta_le$ for $1\leq l < k$.  Since $h_0e = eh_0=0$, composing equation (\ref{eq-deltaCommutator}) on the left (resp.~right) with $h_0$ gives
\begin{equation}\label{h0eq}
h_0\circ [d_0,\Delta_{k}] = 0 \ \ \ \ \ \ \ \ (\text{resp.~}  [d_0,\Delta_{k}]\circ h_0=0).
\end{equation}
Define
\[\a_{k}:=h_0\circ \Delta_{k}-e\circ \Delta_{k}\circ h_0
\]
Since $\Delta_k$ is polynomial in $h_0,e,d_0,d_1,\ldots$, the same is true of $\a_k$.  Compute
\begin{eqnarray*}
[d_0,\a_{k}]
&=& [d_0,h_0]\circ \Delta_{k} - h_0\circ [d_0,\Delta_{k}] - e\circ [d_0,\Delta_{k}]\circ h_0+e\circ \Delta_{k}\circ [d_0,h_0]\\
&=&  (\Id_M-e)\circ \Delta_{k} + e\circ \Delta_{k}\circ (\Id_M-e)\\
& = & \Delta_{k} - e\Delta_{k}e
\end{eqnarray*}
Here we have used that the super-commutator $[d_0,-]$ satisfies the graded Leibniz rule with respect to function composition, together with (\hyperref[h0eq]{\ref{h0eq}}) and the facts that $[d_0,h_0] = \Id_M - e$ and $[d_0,e]=0$.  Therefore
\[
\Delta':=(\Id_M + \a_{k})\circ \Delta \circ (\Id_M -\a_{k}+\a_{k}^2-\cdots) = d_0' +\Delta_1'+\Delta_2'+\cdots
\]
with $\Delta_l' = \Delta_l$ for $1\leq l < k$ and $\Delta_{k}' = \Delta_{k} + \a_{k}\circ d_0 - d_0\circ \a_{k} = e\circ \Delta_{k} \circ e$.  This shows that $\a_1,\ldots,\a_{k}$, satisfy the conditions of \textbf{Hyp}$(k+1)$.  This completes the inductive step and completes the proof.
\end{proof}

We conclude with a useful application of Theorem \ref{convSdrthm}.  Put a bigrading on $R=Z[x_1,\ldots,x_r]$ by declaring that $\deg(x_i)=(a_i,b_i)$.  Let $S\subset R$ denote the set of monomials $x_1^{i_1}\cdots x_r^{i_r}$.  For any $K\in\Kom(n)$, we put
\begin{equation}\label{eq-TensorIsSum}
R\otimes K := \bigoplus_{f\in S} f\otimes K
\end{equation}
whenever this infinite direct sum exists in $\Kom(n)$.  Put a partial order on $S$: say $f\leq g$ if $f$ divides $g$.  We often consider complexes of the form
\[
M = \Z[x_1,\ldots,x_r]\otimes K \ \ \text{with differential } 1\otimes d_K + \sum_{f>1} f\otimes \partial_f
\]
Such a complex can be regarded as a convolution with indexing set $\Z_{\geq 0}^r$, and associated graded $R\otimes K$.  Thus in good situations, a simplification of $K$ up to homotopy equivalence will produce a simplification of $M$.

\begin{observation}\label{obs-contracibleWarning}
One must be careful when attempting to simplify complexes as described above.  For example, suppose $C$ is some complex, and put $K := \Cone(\Id_C)$.  There is a closed morphism $\partial\in\Endg^{1,0}(K)$ such that
\[
C \simeq \Z[x]\otimes K \text{ with differential } 1\otimes d_K + x\otimes \partial
\]
where $x$ is an indeterminate of bidegree $(0,0)$.  Since $K$ is the mapping cone on an isomorphism, we have $K\simeq 0$.  On the other hand $C$ was arbitrary, so the above complex is certainly not always contractible.
\end{observation}
\begin{theorem}\label{thm-ringSDR}
Let $R=\Z[x_1,\ldots,x_r]$, $K$, and $M$ be as in the discussion preceding Observation \ref{obs-contracibleWarning}.  Assume that the direct sum $\Z[x_1,\ldots,x_r]\otimes K = \bigoplus_{f} f\otimes K$ is isomorphic to the categorical direct product $\prod_f f\otimes K$, then there is a deformation retract $M\rightarrow N$, such that
\begin{enumerate}
\item $N=\Z[x_1,\ldots,x_r]\otimes L$ with differential $1\otimes d_L+\sum_{f>1} f\otimes \partial_f'$.
\item the data of the deformation retract $M\rightarrow N$ are equivariant with respect to the $\Z[x_1,\ldots,x_r]$-action
\end{enumerate}
\end{theorem}
\begin{proof}
Recall that $S\subset R$ denotes the set of monomials $x_1^{i_1}\cdots x_r^{i_r}$, and order the monomials by declaring that $f\leq g$ if $f$ divides $g$.  By hypothesis,
\begin{equation}\label{eq-sumEqualsProd}
R\otimes K = \bigoplus_{f\in S} f\otimes K\cong \prod_{f\in S} f\otimes K
\end{equation}
Now, suppose we have $M=R\otimes K$ with differential $1\otimes d_K +\sum_{f>1}f\otimes \partial_f$ as in the hypotheses.  This differential respects the partial ordering on the direct summands (or factors) of (\ref{eq-sumEqualsProd}), hence $M$ can be regarded as a convolution over the indexing set $S$.  Note that the part of the differential on $M$ which preserves the $S$-degree, rather than raises it, is precisely $1\otimes d_K$.  Hence, the associated graded complex is $R\otimes K$.   Any monomial $f\in S$ is divisible by only finitely many distinct monomials, hence we are in a situation to which Theorem \ref{convSdrthm} applies, by Remark \ref{remark-generalizedConv}.  

By hypothesis, we have the data $(\pi,\sigma,h)$  of a deformation retract $K\rightarrow L$.  These give the $R$-equivariant data $(1\otimes\pi, 1\otimes\sigma, 1\otimes h)$ of a deformation retract $R\otimes K\rightarrow R\otimes L$.  Since $R\otimes K$ is the associated graded part of $M$, Theorem \ref{convSdrthm} gives a deformation retract of $M$ onto some $S$-indexed convolution with associated graded part $R\otimes L$.  The data of this deformation retract commute with the $R$ action by Remark \ref{remark-equivariance}, hence $N = R\otimes L$ with some $R$-equivariant differential.  The \emph{only} possibility is $d_N = 1\otimes d_L + \sum_{f>1}f\otimes \partial_f$ for some $\partial_f'\in\Endg(L)$.  This completes the proof.
\end{proof}

\bibliographystyle{hep}
\bibliography{../gen}
 \end{document}